\documentclass[11pt]{amsart}

\usepackage[dvipsnames,table]{xcolor}

\usepackage{xspace}
\usepackage{graphicx}
\usepackage{longtable}
\usepackage{enumitem}
\usepackage{mathdots}
\usepackage{mathtools}
\usepackage{latexsym}
\usepackage{mathrsfs}
\usepackage{stmaryrd}
\usepackage{fullpage}
\usepackage{pict2e}
\usepackage{amsfonts,amsmath,amscd,amssymb,todonotes}
\usepackage{euscript}
\input{xy}
\xyoption{all}
\xyoption{curve}

\usepackage[pdftex,colorlinks,backref,pagebackref,hypertexnames=false,linkcolor=blue,urlcolor=blue,citecolor=blue]{hyperref}

\newcommand{\C}{{\mathbb C}}
\renewcommand{\O}{{\mathbb O}}
\newcommand{\bZ}{{\mathbb Z}}
\newcommand{\bQ}{{\mathbb Q}}

\newcommand{\N}{{\mathbb N}}

\renewcommand{\b}{{\mathfrak b}}
\newcommand{\g}{{\mathfrak g}}

\newcommand{\h}{{\mathfrak h}}
\renewcommand{\l}{{\mathfrak l}}

\newcommand{\so}{{\mathfrak s}{\mathfrak o}}
\renewcommand{\sp}{{\mathfrak s}{\mathfrak p}}
\renewcommand{\P}{{\mathfrak P}}

\newcommand{\p}{{\mathfrak p}}

\renewcommand{\r}{{\mathfrak r}}
\renewcommand{\u}{{\mathfrak u}}
\newcommand{\X}{{\mathfrak X}}
\newcommand{\z}{{\mathfrak z}}
\newcommand{\I}{{\mathfrak I}}

\newcommand{\zg}{\z(\g)}

\newcommand{\cL}{{\mathcal L}}

\newcommand{\ad}{{\operatorname{ad}}}
\newcommand{\Ad}{\operatorname{Ad}}
\renewcommand{\sc}{\operatorname{sc}}
\newcommand{\II}{{\operatorname{II}}}
\renewcommand{\I}{{\operatorname{I}}}
\newcommand{\Lie}{\operatorname{Lie}} 
\newcommand{\Hom}{\operatorname{Hom}} 
\newcommand{\Spec}{\operatorname{Spec}}
\newcommand{\Spin}{{\operatorname{Spin}}}

\newcommand{\Nilp}{\operatorname{Nilp}}
\newcommand{\Cov}{\operatorname{Cov}}
\newcommand{\Dih}{{\mathsf{D}}}
\newcommand{\Alt}{{\mathsf{A}}}
\newcommand{\rig}{\operatorname{rig}}
\newcommand{\Rig}{\operatorname{Rig}}
\newcommand{\brig}{\operatorname{brig}}
\newcommand{\Brig}{\operatorname{Brig}}
\newcommand{\Sym}{{\mathsf{S}}}
\newcommand{\Cyc}{{\mathsf{C}}}
\newcommand{\tw}{{{\mbox{\rm tw}}}}
\newcommand{\Ind}{{{\mbox{\rm Ind}}}}
\newcommand{\GL}{{{\mbox{\rm GL}}}}
\newcommand{\SL}{{{\mbox{\rm SL}}}}
\newcommand{\Spe}{{{\mbox{\rm S}}}}
\newcommand{\PCO}{{{\mbox{\rm PCO}}}}
\newcommand{\PGL}{{{\mbox{\rm PGL}}}}
\newcommand{\SO}{{{\mbox{\rm SO}}}}
\newcommand{\Sp}{{{\mbox{\rm Sp}}}}
\newcommand{\BInd}{{{\mbox{\rm Bind}}}}
\newcommand{\codim}{{{\mbox{\rm Codim}}}}

\newcommand{\reg}{{{\mbox{\rm reg}}}}

\newcommand{\covO}{\widetilde{\O}}
\newcommand{\hatO}{\widehat{\O}}
\newcommand{\breveO}{\breve{\O}}

\newtheorem{theorem}{Theorem}[section]
\newtheorem{prop}[theorem]{Proposition}
\newtheorem{lemma}[theorem]{Lemma}

\newtheorem{cor}[theorem]{Corollary}

\newtheorem{exa}{Example}
\newtheorem{rmk}{Remark}

\newcommand{\arxiv}[1]{{\tt arXiv:#1}}

\numberwithin{equation}{section}

\title{Birational Induction of Nilpotent Orbit Covers in Exceptional Types}
\author{Matthew Westaway}
\email{mw2915@bath.ac.uk}
\address{Department of Mathematical Sciences, University of Bath, Claverton Down, Bath, BA2 7AY, UK}
\date{\today}

\begin{document}
	
	\begin{abstract}
		Let $G$ be a semisimple simply connected algebraic group over $\C$ of exceptional type. For each $G$-equivariant nilpotent cover of a nilpotent coadjoint $G$-orbit $\O$, we determine the unique birationally rigid induction datum from which it is birationally induced.
	\end{abstract}
	
	\maketitle
	
	\section{Introduction}\label{s: Intro}
	
	Lusztig-Spaltenstein induction was introduced in 1979 by Lusztig and Spaltenstein \cite{LS}. While it was originally defined in the context of unipotent classes, it can be easily translated into the language of nilpotent orbits in reductive Lie algebras \cite{CM}, and in this context has had significant applications in Lie-theoretic representation theory. It works in this setting as follows. Let $G$ be a reductive algebraic group over $\C$; this acts on its Lie algebra $\g=\Lie(G)$ via the adjoint action\footnote{In the substance of this paper we prefer to work with the coadjoint action of $G$ on $\g^{*}$, but these two actions can be identified by means of a non-degenerate $G$-invariant symmetric bilinear form.} and thus partitions $\g$ into orbits. Of these, we pay particular attention to the {\bf nilpotent} orbits, i.e. those orbits consisting of nilpotent elements of $\g$. Given a parabolic subgroup $P$ of $G$ with Levi decomposition $P=LU$, where $L$ is itself a reductive algebraic group over $\C$, we may also consider the nilpotent $L$-orbits in $\l=\Lie(L)$. Lusztig-Spaltenstein induction is then a process which takes as input a nilpotent $L$-orbit $\O_L$ and produces from it a nilpotent $G$-orbit $\O=\Ind_L^G(\O_L)$ (we call $(L,\O_L)$ an induction datum for $\O)$. This procedure satisfies a number of nice properties, including that it is transitive, independent of the parabolic subgroup $P$, and behaves predictably on the dimensions of orbits.
	
	One reason why this procedure has turned out to be so useful is that it often allows us to reduce questions about nilpotent orbits to the case of {\bf rigid} nilpotent orbits, i.e. those orbits that cannot be obtained (non-trivially) via Lusztig-Spaltenstein induction. For example, this was a significant tool in resolving the longstanding problem of showing that all finite $W$-algebras have a one-dimensional representation \cite{LoPI, Pr10,PrMF} and the related problem in modular representation theory of finding the minimal dimension for representations of reduced enveloping algebras of Lie algebras over fields of positive characteristic \cite{Pr1,PT2}. Lusztig-Spaltenstein induction also plays an essential role in the classification of sheets \cite{Bo}, i.e. the irreducible components of $\g_{(m)}:=\{x\in\g\mid \dim G\cdot x=m\}$ for $m\in\N$, and is closely related to parabolic induction of primitive ideals \cite{BJ}.
	
	One aspect of Lusztig-Spaltenstein induction which can be both a feature and a bug is that a given nilpotent orbit $\O$ can be induced from multiple different rigid orbits. In order to avoid some of the problems this causes, one option is to look instead at {\bf birational induction}. As initially introduced (see, for example, \cite{LoOM}), being birational was a property that a given induction from one nilpotent orbit to another could have or not have, but this was refined in \cite{LMM} to a procedure similar to Lusztig-Spaltenstein induction except applying to different objects: nilpotent orbit covers. A ($G$-equivariant) nilpotent orbit cover (for a reductive algebraic group $G$) is a homogeneous space $\covO$ equipped with a $G$-equivariant finite morphism to a nilpotent $G$-orbit $\O$; in particular, nilpotent orbits are (trivial) covers of themselves. Birational induction then sends a nilpotent orbit cover $\covO_L\to\O_L$ for a Levi subgroup $L$ in $G$ to a nilpotent orbit cover $\BInd_L^G(\covO_L)\to \Ind_L^G(\O_L)$, and we call $(L,\covO_L)$ a birational induction datum for $\BInd_L^G(\covO_L)$. This procedure has similar properties to Lusztig-Spaltenstein induction: it is transitive, it is independent of the parabolic subgroup $P$ of which $L$ is a Levi factor, and it behaves predictably on degrees of nilpotent covers. It also has an additional property which we don't see for Lusztig-Spaltenstein induction, namely that each nilpotent orbit cover has a {\bf unique} birational induction datum (up to conjugacy) which is birationally rigid, i.e. which cannot be non-trivially obtained through birational induction.
	
	This refinement of Lusztig-Spaltenstein induction is used significantly in \cite{LMM,MM}. Losev, Mason-Brown and Matvieievskyi introduce in \cite{LMM} a definition of unipotent ideals in $U(\g)$, in order to define the notion of unipotent representations of a complex reductive Lie algebra. Birational induction is one of the key tools which allows for the computation of the central characters of these unipotent ideals; this also has connections with the central characters for one-dimensional representations of finite $W$-algebras. Parabolic induction of unipotent ideals and of representations of finite $W$-algebras as in \cite{LoPI} also appears to be related to birational induction. It has also proved key in constructing and understanding Losev's orbit method map \cite{LoOM}.
	
	When $G$ is a semisimple algebraic group of exceptional type, the rigid induction data for each nilpotent $G$-orbit $\O$ are known and can be found in tables in \cite{El,EdG}. In this paper, for $G$ a semisimple simply connected algebraic group of exceptional type, we determine the birationally rigid birational induction datum for each nilpotent orbit cover of an induced nilpotent orbit, which we compile in Tables~\ref{ta: G2} through \ref{ta: E8}. In other words, we have the following theorem.
	\begin{theorem}
		Let $G$ be a semisimple simply connected algebraic group over $\C$ of exceptional type and let $\covO\to \O$ be a $G$-equivariant nilpotent orbit cover, where $\O$ is not a rigid nilpotent orbit. Then the (unique) birationally rigid birational induction datum for $\covO$ is as given in Tables ~\ref{ta: G2} through \ref{ta: E8}.
	\end{theorem}
	We note that Tables~\ref{ta: G2} through \ref{ta: E8} only include the nilpotent orbit covers for induced nilpotent orbits, since nilpotent orbit covers of rigid nilpotent orbits are always birationally rigid. In order to describe the nilpotent covers, we appeal to the fact that such covers of an orbit $\O$ (up to isomorphism) are in bijection with subgroups of the component group of the centralizer of an element $e\in\O$ (up to conjugacy). 
	
	We finish this introduction with a word about $G$ of classical type. A parametrisation of the birationally rigid nilpotent orbits and orbit covers in these cases is explored in \cite[\S\S 7.6.1--7.6.2]{LMM}; for example, for $G=\SL_n$ the birationally rigid nilpotent orbit covers are precisely the universal $G$-equivariant covers of the nilpotent orbits corresponding to partitions of the form $(d^m)$ with $dm=n$. Furthermore, for nilpotent $\SL_n$-{\em orbits} (trivially viewed as covers of themselves) birational induction has the same effect as Lusztig-Spaltenstein induction and so is well-understood. The picture for the other classical types is more involved; we note here only the results of \cite{Nam22}, indicating that the effect of birational induction is related to the number of so-called Type 1 and Type 2 reductions in the Kempken-Spaltenstein algorithm for determining rigid induction data in these cases (see \cite{PT} and \cite{GTW} for more detail on this algorithm).
	
	This paper is structured as follows. In Section~\ref{s: Prelim}, we discuss various preliminary matters, including the precise definition of birational induction and certain (mostly known) results which can be used to understand the effect of birational induction on a nilpotent orbit cover. In Section~\ref{s: CompGrps}, for each standard Levi subgroup $L$ of simply connected $G$ of exceptional type and each nilpotent $L$-orbit $\O_L$ we compute $\pi_1^L(\O_L):=L_e/(L_e)^\circ$ (where $e\in\O_L$). In Section~\ref{s: CbC}, we proceed case-by-case through the nilpotent orbits in exceptional Lie algebras, determining the birationally rigid covers which induce to each of their covers. Some classes of orbits we are able to deal with all at once, but we need to consider some orbits individually. Finally, in Section~\ref{s: Tables} we compile the results of Section~\ref{s: CbC} into tables.
	
	\subsection*{Acknowledgements}
	The author would like to thank Simon Goodwin, Lucas Mason-Brown and Lewis Topley for discussions which were helpful for this paper, and would also like to thank the referee for their incredibly helpful comments on a previous version of this paper -- particularly for explaining how the arguments in Section~\ref{s: CompGrps} should work. The author was supported during this research by a research fellowship from the Royal Commission for the Exhibition of 1851, and during edits as a postdoctoral researcher on a UKRI Future Leaders Fellowship, grant number MR/Z000394/1.

	\section{Preliminaries}\label{s: Prelim}
	
	\subsection{Algebraic groups and Levi subgroups}\label{ss: AlgGps}
	
	In this subsection, we establish some notation and conventions for this paper and recall some well-known facts regarding algebraic groups and their Levi subgroups. 
	
	Throughout this paper, $G$ denotes a reductive algebraic group over $\C$ and $\g$ denotes its Lie algebra. We fix a maximal torus $T$ of $G$ and a positive Borel subgroup $B$ of $G$ containing $T$, and let $\h$ and $\b$ be the Lie algebras thereof. We denote by $\Phi$ the root system of $G$ corresponding to $T$ and by $\Phi^{+}$ the subset of positive roots corresponding to $B$. Furthermore, we denote by $\Pi$ the subset of simple roots of $\Phi^{+}$. Set $W$ to be the Weyl group of $G$, and denote by $\zg$ the centre of $\g$.
	
	The {\bf parabolic subgroups} of $G$ are the closed subgroups of $G$ which contain a conjugate of the Borel subgroup $B$; we call those parabolic subgroups containing $B$ itself the {\bf standard parabolic subgroups} of $G$. Such subgroups are in bijection with subsets of $\Pi$ and each parabolic subgroup is conjugate to a standard parabolic subgroup. Each parabolic subgroup $P$ has a Levi decomposition $P=LU$, where $U$ is the unipotent radical of $P$ and $L$ is reductive. We call $L$ the {\bf Levi factor} of $P$; by a {\bf Levi subgroup} of $G$ we will mean a Levi factor of some parabolic subgroup of $G$. We call $L$ a {\bf standard Levi subgroup} of $G$ when it is the Levi factor of a standard parabolic subgroup of $G$; clearly each Levi subgroup of $G$ is conjugate to a standard Levi subgroup of $G$. To each subset $\Delta\subseteq \Pi$ we may define a standard Levi subgroup $L_\Delta$ as the Levi factor of the standard parabolic subgroup of $G$ corresponding to $\Delta$ (which we denote $P_\Delta$); each Levi subgroup of $G$ is conjugate to some $L_\Delta$. The Weyl group $W$ acts on the set of roots $\Phi$; given two subsets $\Delta$, $\Gamma$ of $\Pi$, the standard Levi subgroups $L_\Delta$ and $L_\Gamma$ are $G$-conjugate if and only if there exists $w\in W$ such that $w(\Gamma)=\Delta$. 
	
	We may therefore associate to a Levi subgroup $L$ of $G$ the (unique) Dynkin type of the root system $\bZ \Delta \cap \Phi$, where $\Delta\subseteq \Pi$ is such that $L$ is conjugate to $L_\Delta$. When $G$ is of exceptional type, this determines $L$ up to conjugacy in almost all cases; however, in type $G_2$ there are two non-conjugate Levi subgroups of type $A_1$, in type $F_4$ there are two non-conjugate Levi subgroups for each type $A_1$, $A_2$ and $A_2+A_1$, and in type $E_7$ there are two non-conjugate Levi subgroups for each type $3A_1$, $A_3+A_1$ and $A_5$. In types $G_2$ and $F_4$ these Levi subgroups can be distinguished based on root lengths, and we use a tilde to denote the short roots. In type $E_7$, we label the two conjugacy classes of Levi subgroups by $(3A_1)'$, $(A_3+A_1)'$ and $(A_5)'$ or by $(3A_1)''$, $(A_3+A_1)''$ and $(A_5)''$, where the latter notation is used when the subset of simple roots is $W$-conjugate to a subset of the black vertices in the below Dynkin diagram.
	
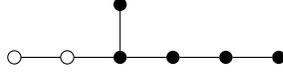
\begin{center}
		\begin{picture}(150,30)
			\multiput(1, 2)(20, 0){2}{\circle{5}} 
			\multiput(41, 2)(20, 0){4}{\circle*{5}}  
			\multiput(3.5,2)(20,0){5}{\line(1,0){15}}
			\put(41,22){\circle*{5}}
			\put(41,4.5){\line(0,1){15}}
		\end{picture}
	\end{center}
	Equivalently, the Levi subgroups $(3A_1)''$, $(A_3+A_1)''$ and $(A_5)''$ are those for which the corresponding set of roots is orthogonal to an $A_2$ subsystem of $\Phi$. We use this notation throughout the paper, including in the tables in Section~\ref{s: Tables}.
	
	The {\bf rank} of a reductive algebraic group $G$ is the dimension of a maximal torus $T\subseteq G$; the semisimple rank of $G$ is the rank of the derived subgroup $[G,G]$. Given a Levi subgroup $L$ of $G$, the {\bf semisimple corank} of $L$ is defined to be the difference between the semisimple rank of $G$ and the semisimple rank of $L$. If $L$ is $G$-conjugate to $L_\Delta$, this equals $\left\vert\Pi\right \vert-\left\vert\Delta\right \vert$.
	
	\subsection{Nilpotent orbits and covers}\label{ss: NilpOrbs}
	
	In this subsection we recall the basics of the theory of nilpotent orbits and nilpotent covers for complex simple Lie algebras. The material on nilpotent orbits is standard and can be found, for example, in \cite{CM}. For the discussion of orbit covers, we follow \cite{LMM}.
	
	The algebraic group $G$ acts on $\g$ via the adjoint action and on $\g^{*}$ via the coadjoint action. Since $\g$ and $\g^{*}$ are $G$-equivariantly isomorphic (via a $G$-invariant non-degenerate symmetric bilinear form) there is a correspondence between adjoint and coadjoint $G$-orbits (in $\g$ and $\g^{*}$ respectively); in this paper, we prefer to work with coadjoint $G$-orbits. A coadjoint $G$-orbit, usually denoted $\O$ in this paper, is called {\bf nilpotent} if $0\in \overline{\O}$ (see \cite[\S1.3]{CM} for various equivalent characterisations). There are only finitely many nilpotent $G$-orbits in $\g^{*}$; we denote by $\Nilp(G)$ the set of nilpotent $G$-orbits in $\g^{*}$.
	
	The following proposition is standard and describes what $\Nilp(G)$ looks like for each simple algebraic group $G$ -- note that $\Nilp(G)$ depends only on the root system $\Phi$, since all nilpotent $G$-orbits lie inside $[\g,\g]$ and the $G$-action factors through $\Ad(G)$.
	
	\begin{prop}\label{p: nilporbs}
		The nilpotent $G$-orbits, when $\Phi$ is of classical type, can be indexed as follows:
		\begin{enumerate}
			\item When $\Phi=A_n$, $n\geq 1$, there is a bijection between nilpotent $G$-orbits and partitions of $n+1$.
			\item When $\Phi=B_n$, $n\geq 2$, there is a bijection between nilpotent $G$-orbits and partitions $\underline{p}$ of $2n+1$ such that each even part of $\underline{p}$ occurs an even number of times.
			\item When $\Phi=C_n$, $n\geq 3$, there is a bijection between nilpotent $G$-orbits and partitions $\underline{p}$ of $2n$ such that each odd part of $\underline{p}$ occurs an even number of times.
			\item When $\Phi=D_n$, $n\geq 4$, there is a bijection between nilpotent $G$-orbits and partitions $\underline{p}$ of $2n$ such that each even part of $\underline{p}$ occurs an even number of times, except that we count twice all partitions in which all parts of $\underline{p}$ are even (we call such partitions {\bf very even}). We use the labels $\I$ and $\II$ to differentiate such partitions.
		\end{enumerate}
		When $\Phi$ is of exceptional type, the nilpotent $G$-orbits are listed explicitly (in Bala-Carter notation) in \cite[\S 8.4]{CM}. There are  20 non-zero nilpotent $G$-orbits in type $E_6$, 44 in type $E_7$, 69 in type $E_8$, 15 in type $F_4$ and 4 in type $G_2$.
	\end{prop}
	
	We follow \cite{CM} in labelling the nilpotent $G$-orbits in exceptional cases via the Bala-Carter labelling -- details of this can be found in \cite[\S8]{CM}. Nilpotent orbits in all cases can also be labelled by weighted Dynkin diagrams and we shall often use this notation as well (see \cite[\S 3.5]{CM} for an explanation of how this description works). This is a labelling of the nodes of the Dynkin diagram of $\g$ by numbers from the set $\{0,1,2\}$. We call $\O$ {\bf even} if all labels on the nodes of its weighted Dynkin diagram are either 0 or 2 (see \cite[\S 3.8]{CM} for equivalent definitions). 
	
	The {\bf $G$-equivariant fundamental group} $\pi_1^G(\O)$ of $\O$ is defined by $$\pi_1^G(\O)\coloneqq G_e/G_e^\circ$$ where $e$ is an element in $\O$ and $G_e$ denotes the stabiliser of $e$ in $G$. Different choices of $e\in\O$ give isomorphic $G$-equivariant fundamental groups, thus we only define $\pi_1^G(\O)$ up to isomorphism.
	
	Given a nilpotent $G$-orbit $\O\subseteq \g^{*}$, a  {\bf $G$-equivariant nilpotent cover} of $\O$ is defined to be a homogeneous space $\covO$ equipped with a finite $G$-equivariant map $\covO\to \O$ (see \cite[\S2.2]{LMM} for more details). We often shorten this just to {\bf nilpotent cover} when $G$ is clear from context. Two such nilpotent covers $\covO$ and $\covO'$ are said to be isomorphic if there exists a $G$-equivariant isomorphism $\covO\to \covO'$ such that the following diagram commutes:
	\begin{eqnarray*}
		\label{e: pidiagprop}
		\begin{array}{c}\xymatrix{
				\covO \ar@{->}[dr] \ar@{->}[rr]^{\sim} & &  \covO' \ar@{->}[dl] \\
				& \O. &
			}
		\end{array}
	\end{eqnarray*}
	We denote by $\Cov(G,\O)$ the set of $G$-equivariant nilpotent covers of $\O$ up to isomorphism and $$\Cov(G):=\bigcup_{\O\in\Nilp(G)}\Cov(G,\O).$$
	
	Fix $e\in \O$. Isomorphism classes of $G$-equivariant nilpotent covers of $\O$ are in bijection with conjugacy classes of subgroups of $\pi_1^G(\O)$ via the following map:
	\begin{equation}\label{e: coviso}
		\Cov(G,\O) \xrightarrow{\sim} \{\mbox{Conjugacy classes of subgroups of }\pi_1^G(\O)\}
	\end{equation}
	$$\covO\longmapsto \pi_1^G(\covO)\coloneqq G_x/G_e^\circ\subseteq G_e/G_e^\circ=\pi_1^G(\O).$$ Here, $x\in\covO$ lies over $e\in \O$; different choices of $x$ give rise different representatives of the conjugacy class of subgroups of $\pi_1^G(\O)$. The inverse map is given by $$G/\widetilde{H}\longmapsfrom H$$ where $\widetilde{H}$ is the preimage of $H$ under the map $G_e\to G_e/G_e^\circ$.
	
	Under this bijection, the subgroup $\pi_1^G(\O)\leq \pi_1^G(\O)$ corresponds to the orbit $\O$, viewed trivially as a cover of itself. On the other hand, the trivial subgroup $1\leq \pi_1^G(\O)$ corresponds to the universal $G$-equivariant cover of $\O$. Following \cite{LMM} and \cite{MM}, we frequently shorten this to the ``{\bf universal cover}'' of $\O$, since we never use any other notions of universal cover in this paper.
	
	\subsection{Lusztig-Spaltenstein induction}\label{ss: LSInd}
	
	In this subsection, we recall the basics of Lusztig-Spaltenstein induction, following \cite[\S 7]{CM}.
	
	Let $P$ be a parabolic subgroup of $G$ with Levi decomposition $P=LU$, where $L$ is the Levi factor of $P$ and $U$ is the unipotent radical of $P$. Denote the corresponding decomposition of Lie algebras by $\p=\l\oplus\u$. Furthermore, set $P^{-}$ to be the opposite parabolic of $P$, with unipotent radical $U^{-}$ and Levi decompositions $P^{-}=U^{-}L$ and $\p=\u^{-}\oplus\l$.  Given a subspace $V$ of $\g$, denote by $V^{\perp}$ the subspace of $\g^{*}$ consisting of those $\chi$ such that $\chi(V)=0$; note that we may identify $\p^{\perp}$ with $(\u^{-})^{*}$ and $\u^{\perp}$ with $(\p^{-})^{*}$. Using a non-degenerate symmetric bilinear form, these subspaces of $\g^{*}$ may be identified with the subspaces $\u$ and $\p$ of $\g$.
	
	Let $\O_L\subseteq \l^{*}$ be a nilpotent $L$-orbit. We may then form the variety $G\times^{P} (\overline{\O}_L\times \p^\perp)$, where $P$ acts on $G$ via right multiplication and on $\overline{\O}_L\times \p^\perp$ via the coadjoint action. There exists a $G$-equivariant morphism \begin{equation}\label{e: mu}
		\mu:G\times^{P} (\overline{\O}_L\times \p^\perp)\to \g^{*},\qquad (g,\chi)\mapsto g\cdot\chi.
	\end{equation} The image of this map coincides with $\overline{\O}$ for some nilpotent $G$-orbit $\O\subseteq \g^{*}$, by Joseph's irreducibility theorem \cite{JoAV}. We then say that $\O$ is {\bf induced} from $\O_L$ and that $(L,\O_L)$ is an {\bf induction datum} for $\O$. This defines a map (called {\bf Lusztig-Spaltenstein induction}, or {\bf induction}) $$\Ind_L^G:\Nilp(L)\to\Nilp(G), \quad \O_L\mapsto \Ind_L^G(\O_L)=\O.$$ This map has three properties of particular note. First, as implied by the notation, Lusztig-Spaltenstein induction is independent of the parabolic subgroup $P$ of $G$ in which $L$ is a Levi factor. Second, it is transitive, in the sense that $\Ind_L^G=\Ind_M^G\circ \Ind_L^M$ whenever $L\subseteq M\subseteq G$ is a tower of Levi subgroups in $G$. Finally, we have the equality $$\codim_{\l^{*}}(\O_L)=\codim_{\g^{*}}(\Ind_L^G(\O_L)),$$ which can be rephrased as $\dim(\Ind_L^G(\O_L))=\dim\O_L + 2\dim \u$.

	A nilpotent $G$-orbit $\O$ is called {\bf rigid} if it cannot be induced from a nilpotent $L$-orbit for a proper Levi subgroup $L$ of $G$, i.e. if the only induction datum for $\O$ is $(G,\O)$. We call an induction datum $(L,\O_L)$ a {\bf rigid induction datum} if $\O_L$ is a rigid nilpotent $L$-orbit in $\l^{*}$. Let us denote the set of rigid induction data for $\O$ by $$\rig(\O):=\{(L,\O_L)\mid \Ind_L^G(\O_L)=\O\}.$$ The adjoint (resp. coadjoint) action of $G$ on $G$ (resp. $\g^{*}$) induces an action of $G$ on $\rig(\O)$. We denote $$\Rig(\O):=\rig(\O)/ G,$$ and we will abuse notation slightly to also refer to elements of $\Rig(\O)$ as rigid induction data for $\O$. By the semisimple corank of an element of $\Rig(\O)$, we will mean the semisimple corank of the Levi subgroup in any representative of such equivalence class.
	
	The classification of rigid nilpotent $G$-orbits is given in \cite[Theorem 7.2.3]{CM}, \cite[Corollary 7.3.5]{CM} and \cite{EdG}, as follows. As with the classification of nilpotent orbits in general, it depends only on the root system $\Phi$.
	
	\begin{prop}\label{p: rignilporbs}
		The rigid nilpotent $G$-orbits, when $\Phi$ is of classical type, can be indexed as follows:
		\begin{enumerate}
			\item When $\Phi=A_n$, $n\geq 1$, the only rigid nilpotent $G$-orbit is the zero orbit.
			\item When $\Phi=B_n$, $n\geq 2$, the rigid nilpotent $G$-orbits correspond to those partitions $\underline{p}=(p_1\geq p_2\geq \cdots\geq p_r)$ of $2n+1$ described in Proposition~\ref{p: nilporbs}(2) which have the property that $p_i\leq p_{i+1}+1$ for all $i=1,\ldots, r$ (setting $p_{r+1}=0$) and which have no odd part of $\underline{p}$ occurring exactly twice.
			\item When $\Phi=C_n$, $n\geq 3$, the rigid nilpotent $G$-orbits correspond to those partitions $\underline{p}=(p_1\geq p_2\geq \cdots\geq p_r)$ of $2n$ described in Proposition~\ref{p: nilporbs}(3) which have the property that $p_i\leq p_{i+1}+1$ for all $i=1,\ldots, r$ (setting $p_{r+1}=0$) and which have no even part of $\underline{p}$ occurring exactly twice.
			\item When $\Phi=D_n$, $n\geq 4$, the rigid nilpotent $G$-orbits correspond to those partitions $\underline{p}=(p_1\geq p_2\geq \cdots\geq p_r)$ of $2n$ described in Proposition~\ref{p: nilporbs}(4) which have the property that $p_i\leq p_{i+1}+1$ for all $i=1,\ldots, r$ (setting $p_{r+1}=0$) and which have no odd part of $\underline{p}$ occurring exactly twice.
		\end{enumerate}
		When $\Phi$ is of exceptional type, the rigid $G$-orbits are listed explicitly in \cite{EdG}. There are  3 non-zero rigid nilpotent $G$-orbits in type $E_6$, 7 in type $E_7$, 17 in type $E_8$, 5 in type $F_4$ and 2 in type $G_2$.
	\end{prop}

	\subsection{Birational induction}\label{ss: BInd}
	
	In this subsection we recall the notion of birational induction from \cite{LoOM} and \cite{LMM}, which extends the idea of Lusztig-Spaltenstein induction to the setting of nilpotent covers.
	
	Maintaining the notation from Subsection~\ref{ss: LSInd}, let $\covO_L$ be an $L$-equivariant cover of $\O_L$. The covering map $\covO_L\to\O_L$ induces a map $\zeta:\Spec(\C[\covO_L])\to \overline{\O}_L$. We hence get an action of $P$ on $\Spec(\C[\covO_L])\times \p^\perp$ by letting $L$ act diagonally and letting $U$ act via $$u\cdot(x,\chi)=(x,u\cdot\zeta(x)-\zeta(x) +u\cdot \chi)$$ for $u\in U$, $x\in \Spec(\C[\covO_L])$ and $\chi\in\p^\perp$. This induces a map $$\widetilde{\mu}:G\times^{P} (\Spec(\C[\covO_L])\times \p^\perp)\to G\times^{P} (\overline{\O}_L\times \p^\perp)\xrightarrow{\mu} \g^{*}.$$ Noting that $\O=\Ind_L^G(\O_L)$ lies in the image of this map, we define $$\covO=\widetilde{\mu}^{-1}(\O).$$ This is a $G$-equivariant nilpotent cover of $\O$. We have hence defined a map $$\BInd_L^G:\Cov(L,\O_L)\to \Cov(G,\Ind_L^G(\O_L)),\quad \covO_L\mapsto \widetilde{\mu}^{-1}(\Ind_L^G(\O_L))=\covO,$$ which we call {\bf birational induction}. As with Lusztig-Spaltenstein induction, birational induction is independent of the parabolic subgroup $P$ of $G$ in which $L$ is a Levi factor and is transitive by \cite[Proposition 2.4.1]{LMM}. We also have the following result from \cite[Proposition 2.4.1]{LMM} and \cite[Proposition 2.9]{MMY}, which will be very useful in Section~\ref{s: CbC}.
	
	\begin{prop}\label{prop: deg}
		Maintain the notation from above and suppose that $\covO_L\to \O_L$ is a degree $m$ nilpotent cover of $\O_L$ and that $\breveO_L\to \O_L$ is a degree $mn$ nilpotent cover of $\O_L$ which factors through $\breveO_L\to\covO_L$, for $m,n\in\N$. Then the degree of the covering map $\BInd_L^G(\covO_L)\to \Ind_L^G(\O_L)$ is divisible by $m$, and the degree of the covering map $\BInd_L^G({\breve\O}_L)\to \BInd_L^G(\covO_L)$ is precisely $n$.
	\end{prop}

	When $\covO$ is a $G$-equivariant nilpotent cover of a nilpotent $G$-orbit $\O$, we call $(L,\covO_L)$ a {\bf birational induction datum} for $\covO$ if $\covO=\BInd_L^G(\covO_L)$. We say that $\covO$ is {\bf birationally rigid} if the only birational induction datum for $\covO$ is $(G,\covO)$, i.e. if it cannot be birationally induced from a proper Levi subgroup. If $\covO_L$ is birationally rigid and $(L,\covO_L)$ is a birational induction datum for $\covO$, we call $(L,\covO_L)$ a {\bf birationally rigid induction datum} for $\covO$.\footnote{We should technically call this a birationally rigid birational induction datum, but for ease of reference we miss out the second ``birational''. In \cite{LoOM} these are referred to as birationally minimal induction data, presumably to avoid confusion with an induction datum $(L,\O_L)$ for $\O$ in which $\O_L$ is birationally rigid but $\O=\Ind_L^G(\O_L)\neq\BInd_L^G(\O_L)$. Since we shall never use the latter concept, we think this abuse of notation is forgivable.}
	
	Analogous to the notation for rigid induction data, let us denote the set of birationally rigid induction data for $\covO$ by $$\brig(\covO):=\{(L,\covO_L)\mid \BInd_L^G(\covO_L)=\covO\}.$$ The adjoint action of $G$ on $G$ induces an action of $G$ on $\brig(\O)$. We denote $$\Brig(\covO):=\brig(\covO)/ G,$$ and we will abuse notation slightly to also refer to elements of $\Brig(\covO)$ as birationally rigid induction data for $\covO$.
	
	The following result about birational induction is due to \cite[Theorem 4.4]{LoOM} and \cite[Proposition 2.4.1]{LMM}, and forms the basis for the question presented in this paper.
	\begin{prop}\label{p: uniq brid}
		Let $\covO$ be a $G$-equivariant nilpotent cover of a nilpotent $G$-orbit $\O$. Then $\Brig(\covO)$ consists of a unique birationally rigid induction datum for $\covO$.
	\end{prop}
	
	A nilpotent $G$-orbit $\O$ is a $G$-equivariant nilpotent cover of itself; this corresponds to the subgroup $\pi_1^G(\O)$ under the bijection (\ref{e: coviso}). Note that $\BInd_L^G(\O_L)=\O$ if and only if the map $\mu$ from (\ref{e: mu}) is birational -- indeed, this was the notion of birational induction initially studied in \cite{LoOM}. With this interpretation of birational induction, we can easily describe the birationally rigid induction datum for even $\O$. When $\O$ is even we may form the Jacobson-Morozov Levi subgroup $L_{\O}$ of $G$; this is the standard Levi subgroup corresponding to the simple roots labelled by 0 in the weighted Dynkin diagram of $\O$. The following result is then standard and can be found in \cite[\S2.4]{LMM} and references therein.

	\begin{prop}\label{p: even}
		Let $\O$ be an even nilpotent $G$-orbit in $\g^{*}$. Then $\O$ is birationally induced from the zero orbit for the Jacobson-Morozov Levi subgroup $L_\O$ corresponding to $\O$.
	\end{prop}
	
	\begin{rmk}\label{rmk: orbs as covs}
		Although each nilpotent $G$-orbit can also be viewed as a $G$-equivariant nilpotent cover of itself, it is worth highlighting here a distinction between the theory of nilpotent orbits and the theory of nilpotent orbit covers. When studying nilpotent orbits of an algebraic group $G$, we generally only care about the underlying root system $\Phi$ of $G$ -- in particular, for most questions we care about it doesn't usually matter whether $G$ is reductive or semisimple, or what the isogeny type of $G$ is. For example, this is true for most questions about Lusztig-Spaltenstein induction and so we don't normally need to concern ourselves with the precise form of Levi subgroups.
		
		On the other hand, in the theory of nilpotent orbit covers the precise form of the algebraic group does matter. This is because $G$-equivariant nilpotent orbit covers of $\O$ (up to isomorphism) are indexed by conjugacy classes of subgroups of $\pi_1^G(\O)=G_e/G_e^\circ$, and this latter group cannot be described purely in terms of the underlying root system of $G$. For example, when $G=\PGL_n$ for some $n\in\N$ we get that $\pi_1^G(\O)=1$ for all nilpotent $G$-orbits $\O$, while for $G=\SL_n$ we get that $\pi_1^G(\O)=\bZ/h\bZ$ where $h$ is the greatest common divisor of all the parts of the partition of $n$ corresponding to $\O$. 
		
		When $G$ is semisimple and the root system of $G$ is indecomposable of exceptional type, we note that $G$ is uniquely defined for types $E_8$, $F_4$ and $G_2$. On the other hand, for types $E_6$ and $E_7$ there are two isogeny classes: simply connected and adjoint. Given a nilpotent $G$-orbit in such cases, the $G$-equivariant fundamental group may include an extra $\bZ/3\bZ$ factor (in the $E_6$ case) or an extra $\bZ/2\bZ$-factor (in the $E_7$ case) when $G$ is simply connected versus when $G$ is adjoint. In this paper, we focus exclusively on the simply connected exceptional groups; in such cases, the $G$-equivariant fundamental groups can be found in \cite[\S 8.4]{CM}.\footnote{For the $G$-equivariant fundamental groups for the classical cases when $G$ is simply connected or adjoint, the reader can consult \cite[Corollary 6.1.6]{CM}.}
		
		Let $L$ be a Levi subgroup of $G$ and $\O_L$ a nilpotent $L$-orbit. In Section~\ref{s: CbC}, we often need to compute $\pi_1^L(\O_L)$ in order to determine the nilpotent covers of $\O_L$. We cannot do this by appealing to any list; instead, Section~\ref{s: CompGrps} is devoted for determining these groups when $G$ is simply connected of exceptional type.
	\end{rmk}
	
	\subsection{Namikawa space and the Namikawa Weyl group}\label{ss: brc}
	In this subsection we recall the basics on the Namikawa space and Namikawa Weyl group that are relevant for this paper, as can be found in \cite{LoOM,LMM,MM,MMY,Nam11}.
	
	When $\covO$ is a $G$-equivariant nilpotent orbit cover of a nilpotent $G$-orbit $\O$, the affine variety $X\coloneqq \Spec(\C[\covO])$ is a conical symplectic singularity by \cite[Lemma 2.5]{Lo21}. Let $X_1:=X\setminus X^{\tiny \reg}$, where $X^{\tiny \reg}$ is the regular locus of $X$, and let $\cL_1,\ldots,\cL_k$ be the irreducible components of $X_1$ of codimension 2; we call these the {\bf symplectic leaves of codimension 2} in $X$. Following \cite{LoOM} and \cite{Nam11} (see also \cite[\S4.5]{LMM}), we may assign to each $\cL_i$ a simple complex Lie algebra $\g_i$ of type $A$, $D$ or $E$. For such Lie algebra $\g_i$, we fix a Cartan subalgebra $\h_i$ and denote by $\Phi_i$ the corresponding root system and by $W_i$ the corresponding Weyl group. The fundamental group $\pi_1(\cL_i)$ acts by monodromy on $\Phi_i$ and thus on $\h_i^{*}$ and on $W_i$. We define the {\bf partial Namikawa space} for $\cL_i$ by $$\P_i:=(\h_i^{*})^{\pi_1(\cL_i)}$$ and define $\P_0:=H^2(X^{\tiny \reg},\C).$ By \cite[Lemma 2.8]{LoOM} the {\bf Namikawa space} of $X$ is then $$\P(\covO)\coloneqq\P_0\oplus \bigoplus_{i=1}^{k}\P_i.$$  Furthermore, the {\bf Namikawa Weyl group} of $X$ is defined by $$W(\covO):=W_1^{\pi_1(\cL_1)}\times\cdots\times W_k^{\pi_1(\cL_k)},$$ which acts component-wise on $\P(\covO)$ (with trivial action on $\P_0$).
	
	The following result, allowing us to describe the Namikawa space of $X$ using birational induction, is \cite[Proposition 7.2.2(i)]{LMM}.
	\begin{prop}\label{p: NamCent}
		Let $\covO\in\Cov(G,\O)$ and let $(L,\covO_L)$ be the birationally rigid induction datum for $\covO$. Then $$\P(\covO)\cong \z(\l\cap[\g,\g])^{*}.$$ 
	\end{prop}
	
	We may similarly describe the Namikawa Weyl group via birational induction. Since we will only need this in the case when $\covO$ is birationally induced from a nilpotent orbit $\O_L$, we only give the result in this case; for the more general case, see \cite{Lo22}, \cite[\S 7]{LMM}. 
	
	We define $$N_G(L,\O_L)=\{g\in G\mid g\cdot L=L\mbox{ and } g\cdot\O_L=\O_L\}$$ and define the {\bf extended Namikawa Weyl group} by $$\widetilde{W}(\covO)=N_G(L,\O_L)/L.$$ The following result is a special case of \cite[Proposition 7.2.2]{LMM} (see also \cite{Lo22}), but is all we need in this paper.
	
	\begin{prop}\label{p: NamW}
		Let $\covO\to \O$ be a $G$-equivariant nilpotent cover of a nilpotent $G$-orbit $\O$, and suppose that $\Brig(\covO)=\{(L,\O_L)\}$ for some Levi subgroup $L$ of $G$ and rigid nilpotent $L$-orbit $\O_L$. Then $W(\covO)$ is a normal subgroup of $N_G(L,\O_L)/L$. Furthermore, if $\covO=\O$ then $W(\covO)=W(\O)=N_G(L,\O_L)/L$. 
	\end{prop}

	\subsection{Birationally rigid nilpotent covers}\label{ss brnc}
	
	In this subsection, we give some results on birational induction and birationally rigid orbit covers. The vast majority of this material can be found in \cite{LMM,MM,MMY}.
	
	In describing the birationally rigid nilpotent $G$-orbits (that is to say, nilpotent $G$-orbits which are birationally rigid when trivially viewed as nilpotent covers), the following result allows us to reduce to semisimple algebraic groups of a given isogeny type.
	
	\begin{theorem}\label{th: isogindep}
		Let $G_1$ and $G_2$ be reductive algebraic groups with the same root system $\Phi$ (associated to fixed choices of maximal tori). Let $\g_1$ and $\g_2$ be the Lie algebras of $G_1$ and $G_2$ respectively. Let $\O_1$ and $\O_2$ be corresponding nilpotent orbits in $\g_1^{*}$ and $\g_2^{*}$. Then $\O_1$ is birationally rigid if and only if $\O_2$ is birationally rigid.
	\end{theorem}
	
	\begin{proof}
		For $i=1,2$, let $(L_i,\O_{L_i})$ be the birationally rigid induction datum for $\O_i$ and write $\l_i=\Lie(L_i)$. According to Proposition~\ref{p: NamCent}, the Namikawa space $\P(\O_i)$ for $\Spec(\C[\O_i])$ is isomorphic to $\z(\l_i\cap[\g,\g])^{*}$. Thus, $\O_i$ is birationally rigid if and only if $\P(\O_i)=0$. On the other hand, $$\P(\O_i)=H^2(\Spec(\C[\O_i])^{\tiny \reg},\C)\oplus (\h_1^{i,*})^{\pi_1(\cL_1^i)}\oplus \cdots \oplus (\h_{k_i}^{i,*})^{\pi_1(\cL_{k_i}^i)} $$ where $\cL_1^i,\ldots,\cL_{k_i}^i$ are the symplectic leaves of codimension 2 in $\Spec(\C[\O_i])$ and the $\h_j^{i,*}$ corresponding to these leaves are as defined at the beginning of Subsection~\ref{ss: brc}. Hence $\P(\O_i)=0$ if and only if $H^2(\Spec(\C[\O_i])^{\tiny \reg},\C)=0$ and $k_i=0$. These two properties are geometric in nature; in particular, since $\O_1$ and $\O_2$ are isomorphic as varieties, they hold for $\O_1$ if and only if they hold for $\O_2$. Hence, $\O_1$ is birationally rigid if and only if $\O_2$ is.
		
	\end{proof}
	
	As discussed in the prior proof, there is a geometric criterion for a nilpotent orbit cover to be birationally rigid due to \cite[Corollary 7.6.1]{LMM}.
	\begin{prop}
		Let $G$ be a semisimple simply connected algebraic group, and let $\covO$ be a $G$-equivariant nilpotent cover of a nilpotent $G$-orbit $\O$ in $\g=\Lie(G)$. Then $\covO$ is birationally rigid if and only if $H^2(\covO,\C)=0$ and $\Spec(\C[\covO])$ has no symplectic leaves of codimension 2.
	\end{prop}
	
	Using this result, \cite[Proposition 7.6.3]{LMM} and \cite[Proposition 3.8.3]{MM} give the following classification of birationally rigid $G$-orbits.\footnote{The result in \cite{MM} actually misses out $A_4+2A_1\subseteq E_8$, but the argument in that paper shows that this orbit is birationally rigid. Indeed, this can also be seen in \cite{Fu}.} By Theorem~\ref{th: isogindep}, this classification in exceptional types can in fact already be found in \cite{Fu}. Note that rigid orbits are automatically birationally rigid.
	
	\begin{prop}\label{p: birigid orbs}
		The birationally rigid nilpotent $G$-orbits can be indexed as follows:
		\begin{enumerate}
			\item When $G=\SL_n$, $n\geq 1$, the only birationally rigid nilpotent $G$-orbit is the zero orbit.
			\item When $G=\SO_{2n+1}$, $n\geq 2$, the birationally rigid nilpotent $G$-orbits correspond to those partitions $\underline{p}=(p_1\geq p_2\geq \cdots\geq p_r)$ of $2n+1$ described in Proposition~\ref{p: nilporbs}(2) with the property that $p_i\leq p_{i+1}+1$ for all $i=1,\ldots, r$ (setting $p_{r+1}=0$).
			\item When $G=\Sp_{2n}$, $n\geq 2$, the birationally rigid nilpotent $G$-orbits correspond to those partitions $\underline{p}=(p_1\geq p_2\geq \cdots\geq p_r)$ of $2n$ described in Proposition~\ref{p: nilporbs}(3) with the property that $p_i\leq p_{i+1}+1$ for all $i=1,\ldots, r$ (setting $p_{r+1}=0$).
			\item When $G=\SO_{2n}$, $n\geq 4$, the birationally rigid nilpotent $G$-orbits correspond to those partitions $\underline{p}=(p_1\geq p_2\geq \cdots\geq p_r)$ of $2n$ described in Proposition~\ref{p: nilporbs}(4) with the property that $p_i\leq p_{i+1}+1$ for all $i=1,\ldots, r$, excluding those partitions of the form $(2^m,1^2)$ for $m\in\N$.
			\item When $G$ is semisimple simply connected of type $E_6$, $F_4$ or $G_2$, the birationally rigid nilpotent $G$-orbits are precisely the rigid nilpotent $G$-orbits.
			\item When $G$ is semisimple simply connected of type $E_7$, the birationally rigid nilpotents $G$-orbits are the rigid nilpotent $G$-orbits together with the nilpotent $G$-orbits with Bala-Carter labels $A_2+A_1$ and $A_4+A_1$.
			\item When $G$ is semisimple simply connected of type $E_8$, the birationally rigid nilpotent $G$-orbits are the rigid nilpotent $G$-orbits together with the nilpotent $G$-orbits with Bala-Carter label $A_4+A_1$ and $A_4+2A_1$.
		\end{enumerate}
	\end{prop}
	
	Note that a birationally rigid nilpotent orbit may admit a birationally induced nilpotent cover (although rigid nilpotent orbits can only admit birationally rigid nilpotent covers) and, conversely, a birationally induced nilpotent orbit may admit a birationally rigid nilpotent cover. For the exceptional groups, the classification of birationally induced nilpotent orbits which admit birationally rigid nilpotent covers is given in \cite[Proposition 3.9.5]{MM}, as follows:
	
	\begin{prop}\label{p: bsr}
		When $G$ is semisimple simply connected of exceptional type, the birationally induced nilpotent $G$-orbits with birationally rigid nilpotent covers can be classified as follows:
		\begin{enumerate}
			\item When $G$ is of type $E_6$, such nilpotent $G$-orbits are the ones with Bala-Carter labels $A_2$, $D_4(a_1)$, $2A_2$, $A_5$ and $E_6(a_3)$.
			\item When $G$ is of type $E_7$, such nilpotent $G$-orbits are the ones with Bala-Carter labels $(3A_1)''$, $A_2$, $A_2+3A_1$, $(A_3+A_1)''$, $D_4(a_1)$, $A_3+2A_1$, $D_4(a_1)+A_1$, $A_3+A_2+A_1$, $A_5+A_1$, $D_5(a_1)+A_1$, $E_7(a_5)$ and $E_7(a_4)$.
			\item When $G$ is of type $E_8$, such nilpotent $G$-orbits are the ones with Bala-Carter labels $A_2$, $2A_2$, $D_4(a_1)$, $D_4(a_1)+A_2$, $D_4+A_2$, $D_6(a_2)$, $E_6(a_3)+A_1$, $E_7(a_5)$, $E_8(a_7)$ and $E_8(b_6)$.
			\item When $G$ is of type $F_4$, such nilpotent $G$-orbits are the ones with Bala-Carter labels $A_2$, $B_2$, $C_3(a_1)$ and $F_4(a_3)$.
			\item When $G$ is of type $G_2$, the only such nilpotent $G$-orbit has Bala-Carter label $G_2(a_1)$.
		\end{enumerate}
	\end{prop}
	
	The proofs of \cite[Proposition 7.6.16]{LMM} and \cite[Proposition 3.9.5]{MM} show that in all these cases -- except possibly $D_4(a_1)\subseteq E_6$ -- the universal cover is birationally rigid.\footnote{The proof of \cite[Proposition 3.9.5]{MM} doesn't explicitly contain this statement for $A_3+A_2+A_1\subseteq E_8$, but it follows immediately from the fact that the orbit itself is not birationally rigid and the universal cover is its only nontrivial cover.} We see in Subsection~\ref{sss: D4a1 in E6} that the universal cover of $D_4(a_1)\subseteq E_6$ turns out not to be birationally rigid; in fact, the nilpotent cover of $D_4(a_1)$ corresponding to the subgroup $\Alt_3\subseteq \Sym_3$ is the only birationally rigid one.
	
	That the universal cover of a nilpotent orbit is birationally rigid generally neither implies nor is implied by the birational rigidity of any other nilpotent cover. However, there will be situations in which the birational rigidity of one nilpotent cover does tell us something about the birational rigidity of another. We say that a $G$-equivariant nilpotent cover $\covO$ is {\bf 2-leafless} if $\Spec(\C[\covO])$ has no codimension 2 symplectic leaves. The following proposition shows how 2-leafless nilpotent covers interact with each other.
	
	\begin{prop}\label{prop: 2leafless covers}
		Let $\O$ be a nilpotent $G$-orbit in $\g^{*}$. Then any $G$-equivariant nilpotent cover of $\O$ which lies over a 2-leafless $G$-equivariant nilpotent cover of $\O$ is also 2-leafless.
	\end{prop}
	
	\begin{proof}
		This is Lemma 7.6.6 in \cite{LMM}.
	\end{proof}
	
	The following straightforward result shows us how we can extend this property, in certain circumstances, to birational rigidity.
	
	\begin{cor}\label{cor: brcovs}
		Suppose $G$ is semisimple and simply connected. Let $\O$ be a nilpotent $G$-orbit in $\g^{*}$ and let $\r$ be the Lie algebra of the reductive part of the centraliser of any $\chi\in\O$. Suppose that $\r$ is semisimple. Then any $G$-equivariant nilpotent cover of $\O$ which lies over a birationally rigid $G$-equivariant nilpotent cover of $\O$ is also birationally rigid. 
	\end{cor}
	
	\begin{proof}
		By \cite[Corollary 7.6.1]{LMM}, a nilpotent cover $\covO$ is birationally rigid if and only if it is 2-leafless and $H^2(\covO,\C)=0$. Let $x\in\covO$, let $R_x$ be the reductive part of the centraliser $G_x$ of $x$ in $G$, and let $\r_x=\Lie(R_x)$. By \cite[Lemma 7.2.7]{LMM}, $H^2(\covO,\C)=\X(\r_x)^{\pi_1^G(\covO)}$. Note that $\r=\r_x$; therefore, $\r$ being semisimple implies $H^2(\covO,\C)=0$. Since this applies for all nilpotent covers of $\O$, we conclude that being birationally rigid is equivalent to being 2-leafless for $G$-equivariant nilpotent covers of such $\O$. The result then follows from Proposition~\ref{prop: 2leafless covers}.
	\end{proof}
	
	Since we use it a couple of times later, we emphasise the point made in this proof that
	\begin{equation}\label{e: H20}
		\r\mbox{ semisimple} \implies H^2(\covO,\C)=0\mbox{ for all nilpotent covers }\covO \mbox{ of }\O.
	\end{equation}

	Another important property of birational induction is discussed in \cite[Lemma 2.5.1]{LMM}. Suppose that $\covO\to \O$ is a $G$-equivariant nilpotent cover which is birationally induced from an $L$-equivariant nilpotent cover $\covO_L\to\O_L$ for some Levi subgroup $L$ of $G$. Pick $x\in \covO$. Then the inclusion $\covO\hookrightarrow G\times^{P} (\covO_L\times \p^\perp)$ induces a surjective homomorphism $$\pi_1(\covO)\twoheadrightarrow \pi_1(G\times^{P} (\covO_L\times \p^\perp)),$$ where we take $x$ as the basepoint for both fundamental groups. Choose $x\in\covO$ and $y\in \covO_L$ such that $x$ maps to the point $(1,y)$ under the map $\covO\hookrightarrow G\times^{P} (\covO_L\times \p^\perp)\twoheadrightarrow G\times^{P} \covO_L$. As in \cite[\S 2.5]{LMM}, when we use $1\in G$, $x\in\covO$ and $y\in\covO_L$ as the appropriate basepoints, the group homomorphism fits into the diagram 
	\begin{eqnarray}
		\label{e: pidiagprop1}
		\begin{array}{c}\xymatrix{
				\pi_1(G) \ar@{->}[rr] \ar@{=}[d] & & \pi_1(\covO) \ar@{->}[rr] \ar@{->}[d] & & \pi_1^G(\covO) \ar@{->}[rr] & & 1 \\
				\pi_1(G) \ar@{->}[rr] & & \pi_1(G\times^{P} (\covO_L\times \p^\perp)) \ar@{->}[rr] & & \pi_1^L(\covO_L) \ar@{->}[rr] & & 1,
			}
		\end{array}
	\end{eqnarray}
	where the rows are exact sequences of homotopy groups obtained from appropriate fibrations. The following result is \cite[Lemma 2.5.1]{LMM}.
	
	\begin{prop}\label{p: surjfund}
		Let $\O$ be a nilpotent $G$-orbit in $\g^{*}$ and let $\covO$ be a $G$-equivariant nilpotent cover of $\O$. Suppose $(L,\covO_L)$ is a birational induction datum for $\covO$. Then there exists a surjective group homomorphism $\pi_1^G(\covO)\twoheadrightarrow \pi_1^L(\covO_L)$ such that Diagram (\ref{e: pidiagprop1}) commutes.
	\end{prop}
	
	We conclude this section with two examples coming from the classical cases.
	
	\begin{exa}
		Let $G=\SL_6$ and consider the regular nilpotent orbit $\O$, which corresponds to the partition $(6)$. Then $\pi_1^G(\O)=\bZ/6\bZ$ has four subgroups: $1$, $\bZ/2\bZ$, $\bZ/3\bZ$ and $\bZ/6\bZ$. These correspond to the four $G$-equivariant nilpotent covers of $\O$, which we label respectively $\hatO$ (the universal cover), $\breveO$, $\covO$ and $\O$ (the orbit viewed as a cover of itself). These fit into the following diagrams 
	$$
			\begin{array}{c}\xymatrix{
					& \hatO \ar@{->>}[dl] \ar@{->>}[dr] & \\
					\breveO \ar@{->>}[dr] & &  \covO \ar@{->>}[dl] \\
					& \O. &
				}
			\end{array} \qquad
			\begin{array}{c}\xymatrix{
					& 1 \ar@{^{(}->}[dl] \ar@{_{(}->}[dr] & \\
					\bZ/2\bZ \ar@{^{(}->}[dr] & &  \bZ/3\bZ \ar@{_{(}->}[dl] \\
					& \bZ/6\bZ. &
				}
			\end{array}
$$
	There are four induction data for $\O$ which are relevant for our discussion here. The corresponding Levi subgroups are $G$ itself, $L_3=\Spe((\GL_3)^2)$, $L_2=\Spe((\GL_2)^3)$ and $L_1=\Spe((\GL_1)^6)$, where for example $S((\GL_2)^3)$ consists of triples of $2\times 2$ invertible matrices over $\C$ such that the product of their determinants is 1. In each case, we take the regular nilpotent orbit corresponding to such subgroup to give the induction datum; in terms of partitions, this means we have $(G,(6))$, $(L_3,((3),(3)))$, $(L_2,((2),(2),(2)))$ and $(L_3,0)$. By \cite[Proposition 7.6.4]{LMM} and \cite[Remark 7.6.5]{LMM}, the universal covers of each of these induction data are birationally rigid, and thus give the birationally rigid induction data for the nilpotent covers of $\O$.
	
	We already know that $\hatO$ is birationally rigid, and since in type $A$ birational induction of nilpotent orbits viewed as covers of themselves has the same effect as Lusztig-Spaltenstein induction we get that $\O$ has birationally rigid induction datum $(L_3,0)$. It thus remains to determine the birationally rigid induction data for $\breveO$ and $\covO$. Since $\covO$ corresponds to an index 2 subgroup of $\pi_1^G(\O)$ it is a 2-fold cover of $\O$, and similarly $\breveO$ is a 3-fold cover of $\O$. Furthermore, the universal cover of $(L_2,((2),(2),(2)))$ is a 2-fold cover of the underlying nilpotent orbit. Proposition~\ref{prop: deg} then implies that the universal cover of $(L_2,((2),(2),(2)))$ birationally induces to a 2-fold cover of $\BInd_{L_2}^G(((2),(2),(2)))=\BInd_{L_1}^G(0)=\O$, i.e. $\covO$. By a similar argument (or process of elimination) we also have that the universal cover of $(L_3,((3),(3)))$ birationally induces to $\breveO$. We thus have the following table of birationally rigid induction data (which we write as (Levi subgroup $L$, nilpotent $L$-orbit $\O_L$, subgroup $\pi_1^L(\covO_L)$ of $\pi_1^L(\O_L)$)):
				\renewcommand{\arraystretch}{2}
	\begin{center}
		\begin{longtable}{|c|c|c|}
			\hline
			Cover & $\pi_1^G(\covO)$ &  Birationally rigid induction datum\\
			\hline
			$\O$ & $\bZ/6\bZ$ & $(L_1, 0, 1\leq 1)$\\
			\hline
			$\covO$ & $\bZ/3\bZ$ & $(L_2, ((2),(2),(2)), 1\leq \bZ/2\bZ)$  \\
			\hline
			$\breveO$ & $\bZ/2\bZ$ & $(L_3, ((3),(3)), 1\leq \bZ/3\bZ)$  \\
			\hline
			$\hatO$ & $1$ &  $(G, (6), 1\leq \bZ/6\bZ)$  \\
			\hline
			\caption{Birationally rigid induction data for covers of $\O$}\label{ta: AInd}
		\end{longtable}
	\end{center}
	\renewcommand{\arraystretch}{}
	Note that in this case $\r=0$ by \cite[Theorem 6.1.3]{CM} and, consistent with Corollary~\ref{cor: brcovs}, there are not any birationally rigid covers which are covered by birationally induced ones. We can also see the existence of the surjections $\pi_1^G(\covO)\twoheadrightarrow \pi_1^L(\covO_L)$ from this table.
	\end{exa}
	
	\begin{exa}
		Let $G=\SO_8$ and consider the nilpotent orbit $\O$ corresponding to the partition $(3,2^2,1)$. This has $\pi_1^G(\O)=\bZ/2\bZ$ (see Table~\ref{ta: ADLevis} below), and thus has two covers: the orbit $\O$ itself and its universal cover $\hatO$. By Proposition~\ref{p: birigid orbs}, this orbit is birationally rigid. It is not rigid, however, because it is induced from the zero orbit for the Levi subgroup $\GL_2\times\SO_4$; this induction datum must therefore birationally induce to $\hatO$. In this case, we have $\r\cong \sp_2\oplus \C$ by \cite[Theorem 6.1.3]{CM}, which is not semisimple. Since here $\O$ is birationally rigid but $\hatO$ is not, we see that the semisimplicity of $\r$ is required in Corollary~\ref{cor: brcovs}.
	\end{exa}
	
	\section{Computing equivariant fundamental groups for Levi subgroups}\label{s: CompGrps}
	
	For the case-by-case arguments applied in the next section, one of the main pieces of data we need is a determination of the $L$-equivariant fundamental groups $\pi_1^L(\O_L)$ for all standard Levi subgroups $L$ of semisimple simply connected algebraic groups of exceptional type and all nilpotent $L$-orbits $\O_L$. One tool which could be used to obtain this is the {\fontfamily{pcr}\selectfont atlas} software \cite{At}; this is the approach used in the case-by-case computations in \cite{MM} and in a previous version of this paper. In this current version of the paper, however, we prefer to determine these groups through more theoretical arguments, which we develop in this section. 
	
	Throughout this section, $G$ is a semisimple simply connected algebraic group with indecomposable root system of exceptional type. We fix a maximal torus $T$ in $G$ and let $\Phi$ be the associated root system of $G$. Define $X(T)=\Hom(T,\C^\times)$ to be the character group of $T$ and $Y(T)=\Hom(\C^\times,T)$ to be the cocharacter group of $T$. Note that there is a natural perfect pairing $\langle -, -\rangle: X(T)\times Y(T)\to \bZ$ such that for $\lambda\in X(T)$ and $\gamma\in Y(T)$ we have $\lambda(\gamma(t))=t^{\langle \lambda,\gamma\rangle}$ for each $t\in\C^\times$. We view $\Phi$ as a subset of $X(T)$, and for each $\alpha\in \Phi$ we write $\alpha^\vee\in Y(T)$ for the corresponding coroot (so $\langle \alpha,\alpha^\vee\rangle=2$ for each $\alpha\in\Phi$).  We fix a base $\Pi$ of $\Phi$, which we enumerate as $\{\alpha_1,\ldots,\alpha_r\}$ in the Bourbaki labelling; for example, this means that for $\Phi$ of type $E_7$ the simple roots are labelled
	$$\begin{array}{c c c c c c}
		& & 2 & & & \\
		1 & 3 & 4 & 5 & 6 & 7. \\
	\end{array}$$
	Recall that $$G \mbox{ is simply connected }\iff Y(T)=\bigoplus_{i=1}^{r} \bZ \alpha_i^\vee$$ and  $$G \mbox{ is adjoint } \iff  X(T)=\bigoplus_{i=1}^{r} \bZ \alpha_i.$$
	
	Finally, we note the following fact about standard Levi subgroups $L_\Delta$ of $G$ (in fact, about all reductive algebraic groups; see, for example, \cite[II.1.18]{JanRAGS}): letting $X_\Delta(T):=X(T)\cap(\bigoplus_{\alpha_i\in\Delta}\bQ \alpha_i)$, we have \begin{equation}\label{eq: ZL} Z(L_\Delta)=\bigcap_{\lambda\in \bZ\Delta}\ker(\lambda)\qquad\mbox{and}\qquad Z(L_\Delta)^\circ=\bigcap_
		{\lambda\in X_\Delta(T)} \ker(\lambda).\end{equation}
	
	For the remainder of this section we assume that $L=L_\Delta$ is a standard Levi subgroup of $G$ corresponding to a subset $\Delta\subseteq \Pi$, and we write $L_{\ad}:=L/Z(L)$ and $\overline{L}=L/Z(L)^\circ$; these are semisimple algebraic groups with the same root system as $L$, the former of which is of adjoint type. In particular, the nilpotent orbits of $L$, $\overline{L}$ and $L_{\ad}$ in their respective Lie algebras are identical and we thus abuse notation by writing $\O_L$ for the corresponding nilpotent orbit in each setting. Note that for $e\in\O_L$, we clearly have\footnote{Throughout this section, $L_e^\circ $ (and similar terms) will always mean $(L_e)^\circ$ rather than $(L^\circ)_e$, but we omit the parentheses for ease of notation.} $Z(L)^\circ \subseteq L_e^\circ$ and hence
	\begin{equation}\label{eq: piover=ad}
		\pi_1^L(\O_L)=L_e/L_e^\circ \cong (L_e/Z(L)^\circ)/(L_e^\circ/Z(L)^\circ)=\overline{L}_e/\overline{L}_e^\circ=\pi_1^{\overline{L}}(\O_L).
	\end{equation}
	
	Using this, we can prove the following result.
	
	\begin{prop}\label{prop: adj fun gp}
		Suppose that $Z(G)=1$. Then $\pi_1^L(\O_L)=\pi_1^{L_\ad}(\O_L)$ for all standard Levi subgroups $L$ and all nilpotent $L$-orbits $\O_L$.
	\end{prop}
	
	\begin{proof}
		By \eqref{eq: piover=ad}, it suffices to show that $\overline{L}=L_\ad$ and thus suffices to show that $Z(L)^\circ=Z(L)$. By \eqref{eq: ZL}, this follows if $\bZ\Delta=X_\Delta(T)$. Since $Z(G)=1$, $G$ has adjoint type. Combining this with the $\bQ$-linear independence of the elements of $\Pi$ yields $$X_\Delta(T)=(\bigoplus_{i=1}^{r}\bZ\alpha_i)\cap(\bigoplus_{\alpha_i\in\Delta}\bQ\alpha_i)=\bZ\Delta$$ as required.
	\end{proof}
	
	Since the semisimple simply connected groups of types $G_2$, $F_4$ and $E_8$ are all centreless (and since we know $\pi_1^{L_\ad}(\O_L)$ in all cases by \cite[Corollary 6.1.6]{CM} and \cite[\S 8.4]{CM}), this reduces the problem to considering the semisimple simply connected groups of types $E_6$ and $E_7$. Many Levi subgroups in these cases can be  tackled through the following lemma, which was already used in the proof of Proposition~\ref{prop: adj fun gp}.
	
	\begin{lemma}\label{cor: piL=ad}
		Let $L$ be a standard Levi subgroup of $G$ such that $Z(L)=Z(L)^\circ$. Then $\pi_1^L(\O_L)=\pi_1^{L_\ad}(\O_L)$ for all nilpotent $L$-orbits $\O_L$.
	\end{lemma}
	
	\begin{proof}
		This is immediate from \eqref{eq: piover=ad} and the definitions of $\overline{L}$ and $L_\ad$.
	\end{proof}
	
	Using \eqref{eq: ZL}, one can calculate that for simply connected $E_6$ and $E_7$ a standard Levi subgroup $L$ can only satisfy $Z(L)\neq Z(L)^\circ$ if $Z(G)$ is contained in the derived subgroup of $L$. It thus remains to determine $\pi_1^L(\O_L)$ for such standard Levi subgroups.
	
	Our next step is therefore to compute these $\pi_1^L(\O_L)$, using case-by-case arguments; we tackle the simply connected group of type $E_6$ in Subsection~\ref{ss: E6Funds} and the simply connected group of type $E_7$ in Subsection~\ref{ss: E7Funds}. For the benefit of these subsections, we record here the following table of $G$-equivariant fundamental groups for semisimple groups $G$ with indecomposable root systems of types $A$ or $D$; the results here come from \cite[Corollary 6.1.6]{CM} and \cite[Pages 298--299]{Car} (recalling that $\Spin_{2n}$ is the simply connected group of type $D_n$ and $\PCO_{2n}^\circ$ is the adjoint group of type $D_n$). In the table, the orbit $\O$ corresponds to a partition $\underline{p}$ and $h$ denotes the highest common factor of the non-zero entries of $\underline{p}$ while $a$ denotes the number of distinct odd entries in $\underline{p}$. We furthermore recall that a partition is called {\em rather odd} if each odd entry has multiplicity at most 1, and for present purposes we call a partition {\em evenly odd} if each odd entry has even multiplicity.
	\renewcommand{\arraystretch}{2}
	\begin{center}
		\begin{longtable}{|c|c|c|}
			\hline
			$G$ & Partition & $\pi_1^G(\O)$ \\
			\hline
			$\SL_{n}$ & Any &  $\bZ/h\bZ$ \\
			\hline
			$\PGL_n$ & Any & $1$ \\
			\hline
			$\Spin_{2n}$ & Rather odd & Central extension of $(\bZ/2\bZ)^{\max(a-1,0)}$ by $\bZ/2\bZ$ \\
			\hline
			$\Spin_{2n}$ &  Not rather odd & $(\bZ/2\bZ)^{\max(a-1,0)}$  \\
			\hline
			$\SO_{2n}$ & Any & $(\bZ/2\bZ)^{\max(a-1,0)}$ \\
			\hline 
			$\PCO_{2n}^\circ$ & Evenly odd & $(\bZ/2\bZ)^{\max(a-1,0)}$\\
			\hline
			$\PCO_{2n}^\circ$ & Not evenly odd & $(\bZ/2\bZ)^{\max(a-2,0)}$\\
			\hline
			\caption{Equivariant fundamental groups in types $A$ and $D$.}\label{ta: ADLevis}
		\end{longtable}
	\end{center}
	\renewcommand{\arraystretch}{}
	
	We record our results from the next two subsections in Tables~\ref{ta: E6Levis} and \ref{ta: E7Levis}; in these tables, we indicate the standard Levi subgroup $L$ in the first column, the nilpotent orbit $\O_L$ in the second column (via its partition) and third column (via its weighted Dynkin diagram), and the group $\pi_1^L(\O_L)$ in the fourth column. We only list those $(L,\O_L)$ for which $\pi_1^L(\O_L)\neq \pi_1^{L_\ad}(\O_L)$. Following the conventions of \cite{CM}, the reader may determine $\pi_1^{L_\ad}(\O_L)$ in these cases by omitting a $\bZ/3\bZ$ (in Table~\ref{ta: E6Levis}) or $\bZ/2\bZ$ (in Table~\ref{ta: E7Levis}), noting that we write $\Sym_2$ in place of $\bZ/2\bZ$ for cyclic groups of order 2 which should not be omitted.
	
	\begin{rmk}\label{rmk: brid covs}
		In Tables~\ref{ta: E6Levis} and \ref{ta: E7Levis} we also add a $(\dagger)$ to those induction data which are induced from earlier induction data in the same table. With the exception of the orbit labelled $(5,3,1^2)\times (2)$ in the $D_5+A_1$ Levi subgroup of simply connected $E_7$ this induction must be birational, in the sense that it has the same effect as birational induction on the induction data. Furthermore, for those induction data in Tables~\ref{ta: E6Levis} and \ref{ta: E7Levis} not labelled with a $(\dagger)$, either the orbit or its universal cover must be birationally rigid (since all the proper induction data for these induction data must have trivial equivariant fundamental groups and since we can see by examination that the $\pi_1^L(\O_L)$ are all simple in these cases). We can tell whether the orbit is birationally rigid by appealing to Theorem~\ref{th: isogindep} and Proposition~\ref{p: birigid orbs}; in fact, the only entry without a $(\dagger)$ for which the orbit is birationally rigid is the orbit corresponding to the partition $(3,2^4,1)$ in the Levi subgroup of type $D_6$ in simply connected $E_7$. Therefore, for all the remaining entries of Tables~\ref{ta: E6Levis} and \ref{ta: E7Levis} without a $(\dagger)$ the universal cover is birationally rigid.
	\end{rmk}

	\subsection{Equivariant fundamental groups of Levi subgroups of simply connected $E_6$}\label{ss: E6Funds}
	
	In this subsection, $G$ is the semisimple simply connected algebraic group with indecomposable root system of type $E_6$.
	
	Since $Z(G)=\bigcap_{\alpha\in\Phi}\ker(\alpha)$, it is straightforward to compute that $$Z(G)=\langle \alpha_1^\vee(\omega)\alpha_3^\vee(\omega^2)\alpha_5^\vee(\omega)\alpha_6^\vee(\omega^2) \rangle\cong \bZ/3\bZ$$ where $\omega=e^{2\pi i/3}$. Therefore, the only standard Levi subgroups with derived subgroup containing $Z(G)$ are $2A_2$, $2A_2+A_1$ and $A_5$ (here and throughout this section we refer to a standard Levi subgroup by its Dynkin type, with possible decorations in type $E_7$, as discussed in Subsection~\ref{ss: AlgGps}).
	
	Since $G$ is simply connected, the derived subgroup $[L,L]$ of each standard Levi subgroup $L$ is also simply connected. Furthermore, $L=[L,L]Z(L)^\circ$ and so $L/Z(L)^\circ\cong [L,L]/[L,L]\cap Z(L)^\circ$. Computing $Z(L)^\circ$ using \eqref{eq: ZL}, we can see that for $L$ of types $2A_2$, $2A_2+A_1$ and $A_5$ we obtain that $\overline{L}=L/Z(L)^\circ$ is 	
	\begin{equation}\label{eq: E6L1}
		\frac{[L,L]}{\langle \alpha_1^\vee(\omega)\alpha_3^\vee(\omega^2)\alpha_5^\vee(\omega^2)\alpha_6^\vee(\omega)\rangle}=\frac{\SL_3\times \SL_3}{\langle (\omega I_3,\omega^2 I_3)\rangle},
	\end{equation}
	\begin{equation}\label{eq: E6L2}		
		\frac{[L,L]}{\langle \alpha_1^\vee(\omega)\alpha_3^\vee(\omega^2)\alpha_5^\vee(\omega^2)\alpha_6^\vee(\omega), \alpha_2(-1)\rangle}=\frac{\SL_3\times \SL_3}{\langle (\omega I_3,\omega^2 I_3)\rangle}\times \PGL_2,
	\end{equation}
	\begin{equation}\label{eq: E6L3}
		\frac{[L,L]}{\langle \alpha_1^\vee(-1)\alpha_4^\vee(-1)\alpha_6^\vee(-1)\rangle}=\frac{\SL_6}{\langle -I_6 \rangle},
	\end{equation}
	respectively. As in the proof of Proposition~\ref{prop: adj fun gp}, $\pi_1^L(\O_L)=\pi_1^{\overline{L}}(\O_L)$ for each nilpotent $L$-orbit $\O_L$; our goal now is thus to determine $\pi_1^{\overline{L}}(\O_L)$ for each nilpotent orbit for the groups \eqref{eq: E6L1}, \eqref{eq: E6L2} and \eqref{eq: E6L3}. 
	
	Before proceeding further, we note some general principles which will inform our calculations (these will also apply in Subsection~\ref{ss: E7Funds} with suitable modifications). Note that there are natural surjections $$[L,L]\twoheadrightarrow \overline{L}\twoheadrightarrow L_\ad;$$ fixing $e\in\O_L$, these induce surjections $$[L,L]_e\twoheadrightarrow \overline{L}_e\twoheadrightarrow (L_\ad)_e$$ and 
	%$$\frac{[L,L]_e}{([L,L]_e)^\circ}\twoheadrightarrow \frac{\overline{L}_e}{(\overline{L}_e)^\circ}\twoheadrightarrow \frac{(L_\ad)_e}{((L_\ad)_e)^{\circ}}.$$ 
	$$\pi_1^{[L,L]}(\O_L)\twoheadrightarrow \pi_1^{\overline{L}}(\O_L)\twoheadrightarrow \pi_1^{L_\ad}(\O_L).$$
	Fix the notation $Z=\ker([L,L]\to \overline{L})$ and $\widehat{Z}$ for the centre of $[L,L]$; in particular, $Z$ is an index 3 subgroup of $\widehat{Z}$ for the groups \eqref{eq: E6L1}--\eqref{eq: E6L3}. It is straightforward to see that %$$\ker\left(\frac{[L,L]_e}{([L,L]_e)^\circ}\twoheadrightarrow \frac{\overline{L}_e}{(\overline{L}_e)^\circ}\right)= \frac{([L,L]_e)^\circ Z}{([L,L]_e)^\circ},\qquad \ker\left(\frac{[L,L]_e}{([L,L]_e)^\circ}\twoheadrightarrow \frac{(L_\ad)_e}{((L_\ad)_e)^{\circ}}\right)= \frac{([L,L]_e)^\circ \widehat{Z}}{([L,L]_e)^\circ}.$$
	$$\ker\left(\pi_1^{[L,L]}(\O_L)\twoheadrightarrow \pi_1^{\overline{L}}(\O_L)\right)= \frac{Z[L,L]_e^\circ }{[L,L]_e^\circ},\qquad \ker\left(\pi_1^{[L,L]}(\O_L)\twoheadrightarrow \pi_1^{L_\ad}(\O_L)\right)= \frac{\widehat{Z}[L,L]_e^\circ}{[L,L]_e^\circ}.$$
	We therefore conclude that $$\pi_1^{\overline{L}}(\O_L)=\pi_1^{L_\ad}(\O_L) \iff \frac{Z [L,L]_e^\circ}{[L,L]_e^\circ}=\frac{\widehat{Z}[L,L]_e^\circ}{[L,L]_e^\circ},$$
	and that if $\pi_1^{\overline{L}}(\O_L)\neq\pi_1^{L_\ad}(\O_L)$ then $\pi_1^{\overline{L}}(\O_L)$ is a central extension of $\pi_1^{L_\ad}(\O_L)$ by $\bZ/3\bZ$. Furthermore, $Z[L,L]_e^\circ /[L,L]_e^\circ=\widehat{Z}[L,L]_e^\circ /[L,L]_e^\circ$ if and only if at least one element of $\widehat{Z}\setminus Z$ lies in $[L,L]_e^\circ$ (since $\widehat{Z}/Z$ is cyclic of order 3).
	
	%We therefore conclude that $\pi_1^{\overline{L}}(\O_L)=\pi_1^{L_\ad}(\O_L)$ if and only if $([L,L]_e)^\circ Z/([L,L]_e)^\circ=([L,L]_e)^\circ \widehat{Z}/([L,L]_e)^\circ$, and that if $\pi_1^{\overline{L}}(\O_L)\neq\pi_1^{L_\ad}(\O_L)$ then $\pi_1^{\overline{L}}(\O_L)$ is a central extension of $\pi_1^{L_\ad}(\O_L)$ by $\bZ/3\bZ$. Furthermore, $([L,L]_e)^\circ Z/([L,L]_e)^\circ=([L,L]_e)^\circ \widehat{Z}([L,L]_e)^\circ$ if and only if at least one element of $\widehat{Z}\setminus Z$ lies in $([L,L]_e)^\circ$ (since $\widehat{Z}/Z$ is cyclic of order 3).
	
	We now apply these principles to determine $\pi_1^L(\O_L)$ for each standard Levi subgroup $L$ of $G$ and each nilpotent $L$-orbit $\O_L$.
	
	\begin{prop}\label{prop: E6 Levi Gps}
		Let $(L,\O_L)$ be a pair consisting of a standard Levi subgroup of $G$ and a nilpotent $L$-orbit in $\Lie(L)^*$. Then $\pi_1^{L}(\O_L)=\pi_1^{L_\ad}(\O_L)$ if and only if the pair $(L,\O_L)$ does not lie in Table~\ref{ta: E6Levis}. If $(L,\O_L)$ does lie in Table~\ref{ta: E6Levis}, then $\pi_1^{L}(\O_L)$ is as given in that table.
	\end{prop}
	
	\begin{proof}
		
		By \eqref{eq: piover=ad} and Lemma~\ref{cor: piL=ad} it suffices to determine $\pi_1^{\overline{L}}(\O_L)$ for the $\overline{L}$ labelled \eqref{eq: E6L1}--\eqref{eq: E6L3}. Since $\pi_1^{\tiny \PGL_2}(\O)=1$ for all nilpotent orbits $\O$ by Table~\ref{ta: ADLevis}, we need only consider \eqref{eq: E6L1} and \eqref{eq: E6L3}.
		
		For \eqref{eq: E6L1}, Table~\ref{ta: ADLevis} gives $\pi_1^{\tiny \SL_3}(\O)=\bZ/3\bZ$ for the orbit corresponding to the partition $(3)$ and $\pi_1^{\tiny \SL_3}(\O)=1$ for the orbits corresponding to the partitions $(2,1)$ and $(1^3)$. We therefore conclude that both $\omega I_3$ and $\omega^2 I_3$ do not lie in $(\SL_3)_e^\circ$ in the former case but both do in the latter case. Hence, arguing as before the proposition, we see from the description of $Z$ that $\pi_1^{L}(\O_L)=\pi_1^{L_\ad}(\O_L)$ if and only if $\O_L$ has partition $(2,1)$ or $(1^3)$, and that when $\O_L$ has partition $(3)$ then $\pi_1^{L}(\O_L)$ is an extension of $\pi_1^{L_\ad}(\O_L)=1$ by $\bZ/3\bZ$. 
		
		%For \eqref{eq: E6L1}, note that $\pi_1^{\tiny \SL_3}(\O_{(3)})=\bZ/3\bZ$, and $\pi_1^{\tiny \SL_3}(\O_{(2,1)})=\pi_1^{\tiny \SL_3}(\O_{(1^3)})=1$ by Table~\ref{ta: ADLevis}. We therefore conclude that both $\omega I_3$ and $\omega^2 I_3$ do not lie in $(\SL_3)_e^\circ$ in the former case but both do in the latter case. Hence, arguing as before the proposition, we see from the description of $Z$ that $\pi_1^{L}(\O_L)=\pi_1^{L_\ad}(\O_L)$ if and only if $\O_L$ has partition $(2,1)$ or $(1^3)$, and that when $\O_L$ has partition $(3)$ then $\pi_1^{L}(\O_L)$ is an extension of $\pi_1^{L_\ad}(\O_L)=1$ by $\bZ/3\bZ$. 
		
		For \eqref{eq: E6L3}, we note that $Z[L,L]_e^\circ /[L,L]_e^\circ$ has index 1 or 3 in $\widehat{Z} [L,L]_e^\circ / [L,L]_e^\circ$. Clearly it has index 1 if $\left\vert \widehat{Z} [L,L]_e^\circ /[L,L]_e^\circ \right\vert$ is not divisible by 3. On the other hand, if $\left\vert \widehat{Z} [L,L]_e^\circ /[L,L]_e^\circ \right\vert$  is divisible by 3 then the elements $\omega I_6$, $\omega^2 I_6$, $-\omega I_6$ and $-\omega^2 I_6$ cannot lie in $[L,L]_e^\circ$, and so we may argue as before the proposition to get $\pi_1^L(\O_L)\neq \pi_1^{L_\ad}(\O_L)$ in these cases. Examining the equivariant fundamental groups for $\SL_6$ gives the result.
	\end{proof}
	
	\renewcommand{\arraystretch}{2}
	\begin{center}
		\begin{longtable}{|c|c|c|c|}
			\hline
			Levi type & Nilpotent orbit (partition) & Nilpotent orbit (diagram) &  $\pi_1^L(\O_L)$\\
			\hline
			$2A_2$ & $(3)\times(3)$ & $\renewcommand{\arraystretch}{}\begin{array}{c c c c c}
				2 & 2 & & 2 & 2   \\
			\end{array}$ &  $\bZ/3\bZ$ \\
			\hline
			$2A_2+A_1$ & $(3)\times (3)\times (1^2)$ & $\renewcommand{\arraystretch}{}\begin{array}{c c c c c}
				&  & 0 & &  \\
				2 & 2 &  & 2 & 2   \\
			\end{array}$ & $\bZ/3\bZ$\\
			$(\dagger)$ & $(3)\times (3)\times (2)$  & $\renewcommand{\arraystretch}{}\begin{array}{c c c c c}
				&  & 2 & &  \\
				2 & 2 &  & 2 & 2   \\
			\end{array}$ & $\bZ/3\bZ$\\
			\hline
			$A_5$ & $(3^2)$ & $\renewcommand{\arraystretch}{}\begin{array}{c c c c c}
				0 & 2 & 0  & 2 & 0   \\
			\end{array}$ & $\bZ/3\bZ$ \\
			$(\dagger)$ & $(6)$ & $\renewcommand{\arraystretch}{}\begin{array}{c c c c c}
				2 & 2 & 2 & 2 & 2   \\
			\end{array}$ & $\bZ/3\bZ$ \\
			\hline
			\caption{Equivariant fundamental groups for nilpotent orbits of standard Levi subgroups in simply connected $E_6$}\label{ta: E6Levis}
		\end{longtable}
	\end{center}
	
	\subsection{Equivariant fundamental groups of Levi subgroups of simply connected $E_7$}\label{ss: E7Funds}
	
	In this subsection, $G$ is the semisimple simply connected algebraic group with indecomposable root system of type $E_7$.
	
	As for $E_6$, we may compute that $$Z(G)=\langle \alpha_2^\vee(-1)\alpha_5^\vee(-1)\alpha_7^\vee(-1)\rangle\cong \bZ/2\bZ.$$ Therefore, the only standard Levi subgroups with derived subgroup containing $Z(G)$ are those with Dynkin types $(3A_1)''$, $4A_1$, $(A_3+A_1)''$, $A_2+3A_1$, $A_3+2A_1$, $D_4+A_1$, $(A_5)''$, $D_5+A_1$, $A_3+A_2+A_1$, $A_5+A_1$, and $D_6$.\footnote{For most of these Cartan types there is a unique standard Levi subgroup of $E_7$ with derived subgroup containing $Z(G)$ which has that Cartan type. The exceptions are $4A_1$, $A_3+A_1$ and $A_3+2A_1$; since $G$-conjugacy won't change the isomorphism type of $\pi_1^L(\O_L)$, we fix for the following list that $4A_1$ corresponds to the simple roots labelled $\{1,2,5,7\}$, that $A_3+A_1$ corresponds to those labelled $\{2,4,5,7\}$, and that $A_3+2A_1$ corresponds to those labelled $\{1,2,4,5,7\}$ (in the Bourbaki labelling, as usual).}
	
	Arguing as for $E_6$, we obtain that the corresponding $\overline{L}$ are, respectively,
	\begin{equation}\label{eq: E7L1}
		\frac{[L,L]}{\langle \alpha_2^\vee(-1)\alpha_5^\vee(-1), \alpha_5^\vee(-1)\alpha_7^\vee(-1)\rangle}=\frac{\SL_2\times \SL_2 \times \SL_2}{\langle (-I_2,-I_2,I_2), (I_2,-I_2,-I_2)\rangle},
	\end{equation}
	\begin{equation}\label{eq: E7L2}
		\frac{[L,L]}{\langle \alpha_1^\vee(-1),\alpha_2^\vee(-1)\alpha_5^\vee(-1), \alpha_5^\vee(-1)\alpha_7^\vee(-1)\rangle}=\PGL_2\times \frac{\SL_2\times \SL_2 \times \SL_2}{\langle (-I_2,-I_2,I_2), (I_2,-I_2,-I_2)\rangle},
	\end{equation}
	\begin{equation}\label{eq: E7L3}
		\frac{[L,L]}{\langle \alpha_2^\vee(i)\alpha_4^\vee(-1) \alpha_5^\vee(-i)\alpha_7^\vee(-1)\rangle}=\frac{\SL_4\times \SL_2}{\langle (iI_4,-I_2) \rangle},
	\end{equation}
	\begin{equation}\label{eq: E7L4}
		\frac{[L,L]}{\langle \alpha_1^\vee(\omega),\alpha_2^\vee(-1)\alpha_5^\vee(-1), \alpha_5^\vee(-1)\alpha_7^\vee(-1)\rangle}=\PGL_3\times \frac{\SL_2\times \SL_2 \times \SL_2}{\langle (-I_2,-I_2,I_2), (I_2,-I_2,-I_2)\rangle},
	\end{equation}
	\begin{equation}\label{eq: E7L5}
		\frac{[L,L]}{\langle \alpha_1^\vee(-1),  \alpha_2^\vee(i)\alpha_4^\vee(-1) \alpha_5^\vee(-i)\alpha_7^\vee(-1)\rangle}=\PGL_2\times \frac{\SL_4\times \SL_2}{\langle (iI_4,-I_2) \rangle},
	\end{equation}
	\begin{equation}\label{eq: E7L6}
		\frac{[L,L]}{\langle \alpha_2^\vee(-1)\alpha_3^\vee(-1) \alpha_7^\vee(-1),\alpha_2^\vee(-1)\alpha_5^\vee(-1)\rangle}=\frac{\Spin_8\times \SL_2}{\langle (\alpha_2^\vee(-1)\alpha_3^\vee(-1),-I_2), (\alpha_2^\vee(-1)\alpha_5^\vee(-1),I_2) \rangle},
	\end{equation}
	\begin{equation}\label{eq: E7L7}
		\frac{[L,L]}{\langle \alpha_2^\vee(\omega)\alpha_4^\vee(\omega^2) \alpha_6^\vee(\omega^2)\alpha_7^\vee(\omega)\rangle}=\frac{\SL_6}{\langle (\omega I_6) \rangle},
	\end{equation}
	\begin{equation}\label{eq: E7L8}
		\frac{[L,L]}{\langle \alpha_1^\vee(-1)\alpha_2^\vee(i)\alpha_4^\vee(-1)\alpha_5^\vee(-i)\alpha_7(-1)\rangle}=\frac{\Spin_{10}\times \SL_2}{\langle (\alpha_1^\vee(-1)\alpha_2^\vee(i)\alpha_4^\vee(-1)\alpha_5^\vee(-i), -I_2) \rangle},
	\end{equation}
	\begin{equation}\label{eq: E7L9}
		\frac{[L,L]}{\langle \alpha_1^\vee(\omega)\alpha_3^\vee(\omega^2), \alpha_2^\vee(-1)\alpha_5^\vee(-i) \alpha_6^\vee(-1)\alpha_7^\vee(i)\rangle}=\PGL_3\times \frac{\SL_4\times \SL_2}{\langle (iI_4,-I_2) \rangle},
	\end{equation}
	\begin{equation}\label{eq: E7L10}
		\frac{[L,L]}{\langle \alpha_1^\vee(-1), \alpha_2^\vee(\omega)\alpha_4^\vee(\omega^2) \alpha_6^\vee(\omega^2)\alpha_7^\vee(\omega)\rangle}=\PGL_2\times \frac{\SL_6}{\langle (\omega I_6) \rangle},
	\end{equation}
	\begin{equation}\label{eq: E7L11}
		\frac{[L,L]}{\langle \alpha_3^\vee(-1)\alpha_5^\vee(-1) \alpha_7^\vee(-1)\rangle}=\frac{\Spin_{12}}{\langle \alpha_3^\vee(-1)\alpha_5^\vee(-1) \alpha_7^\vee(-1) \rangle}.
	\end{equation}
	
	For each of these groups, we now determine $\pi_1^{\overline{L}}(\O_L)$ (which equals $\pi_1^L(\O_L)$ by \eqref{eq: piover=ad}) for each nilpotent $L$-orbit $\O_L$. Our strategy from the previous subsection can be employed here as well, except that we must replace $\bZ/3\bZ$ with $\bZ/2\bZ$ as appropriate.
	
	\begin{prop}\label{prop: E7 Levi Gps}
		Let $(L,\O_L)$ be a pair consisting of a standard Levi subgroup of $G$ and a nilpotent $L$-orbit in $\Lie(L)^*$. Then $\pi_1^{L}(\O_L)=\pi_1^{L_\ad}(\O_L)$ if and only if the pair $(L,\O_L)$ does not lie in Table~\ref{ta: E7Levis}. If $(L,\O_L)$ does lie in Table~\ref{ta: E7Levis}, then $\pi_1^{L}(\O_L)$ is as given in that table.
	\end{prop}
	
	\begin{proof}
		By \eqref{eq: piover=ad} and Lemma~\ref{cor: piL=ad} it suffices to determine $\pi_1^{\overline{L}}(\O_L)$ for the $\overline{L}$ labelled \eqref{eq: E7L1}--\eqref{eq: E7L11}. In fact, since the $\PGL_2$- and $\PGL_3$-equivariant fundamental groups are trivial for all nilpotent $\PGL_2$- and $\PGL_3$-orbits by Table~\ref{ta: ADLevis}, we need only consider the groups \eqref{eq: E7L1}, \eqref{eq: E7L3}, \eqref{eq: E7L6}, \eqref{eq: E7L7}, \eqref{eq: E7L8} and \eqref{eq: E7L11}.
		
		The arguments for \eqref{eq: E7L1} and \eqref{eq: E7L7} are essentially the same as those used for \eqref{eq: E6L1} and \eqref{eq: E6L3}, respectively, in the proof of Proposition~\ref{prop: E6 Levi Gps}. We leave the details to the reader.
		
		For \eqref{eq: E7L3}, we need to check whether the elements $(I_4,-I_2)$, $(iI_4,I_2)$, $(-I_4,-I_2)$ and $(-iI_4,I_2)$ lie in $[L,L]_e^\circ$. When the component of $e$ in the $A_1$ factor is zero, it is clear that $(I_4,-I_2)\in [L,L]_e^\circ$ and thus that $\pi_1^{\overline{L}}(\O_L)=\pi_1^{L_\ad}(\O_L)$ in those cases. When, on the other hand, this component is regular, we need only check whether $iI_4$ and $-iI_4$ lie in $(\SL_4)_e^\circ$ (replacing here $e$ with its component in $A_4$). This is straightforward to determine from our knowledge of the equivariant fundamental groups for $\SL_4$ from Table~\ref{ta: ADLevis}.
		
		For \eqref{eq: E7L6}, it will be helpful to note that $$\overline{L}\cong\frac{\SO_8\times \SL_2}{\langle \alpha_2^\vee(-1)\alpha_3^\vee(-1),-I_2\rangle}.$$ Arguing as in the proof of Proposition~\ref{prop: E6 Levi Gps}, we need only determine  -- subject to a complication to be discussed momentarily -- for which $e$ the elements $(\alpha_2^\vee(-1)\alpha_3^\vee(-1),I_2)$ or $(I_2,-I_2)$ lie in $(\SO_8\times \SL_2)_e^\circ$. The latter will do so whenever the $A_1$ component is zero, as for \eqref{eq: E7L3}. For the former, we easily see that $\alpha_2^\vee(-1)\alpha_3^\vee(-1)\in (\SO_8)_e^\circ$ if and only if $\pi_1^{\tiny \SO_8}(\O_L)=\pi_1^{\tiny \PCO_8^\circ}(\O_L)$ (here replacing $e$ and $\O_L$ with their components in $D_4$). We may determine this easily using Table~\ref{ta: ADLevis}. 
		
		Now for the complication. The symmetric group on 3 letters acts on $\Spin_8\times \SL_2$ in such a way that the nilpotent orbits corresponding to the partitions $(3,1^5)\times (2)$, $(2^4)_\I \times (2)$ and $(2^4)_\II\times (2)$ are permuted and the nilpotent orbits corresponding to the partitions $(5,1^3)\times (2)$, $(4^2)_\I\times (2)$ and $(4^2)_\II\times (2)$ are permuted. Thus, all we are able determine for these partitions from the above argument is that two partitions of each of these three have $\pi_1^{\overline{L}}(\O_L)=\pi_1^{L_\ad}(\O_L)=1$ and the third has $\pi_1^{\overline{L}}(\O_L)=\bZ/2\bZ$ but $\pi_1^{L_\ad}(\O_L)=1$. The nilpotent orbits with weighted Dynkin diagrams
		$$\renewcommand{\arraystretch}{}\begin{array}{c c c c c}
			& 0 & & &  \\
			2 & 0 & 0 &  & 2   \\
		\end{array}\quad \mbox{and}\quad  \begin{array}{c c c c c}
			& 0 & & &  \\
			2 & 2 & 0 &  & 2   \\
		\end{array}$$ are induced from the orbits $$\renewcommand{\arraystretch}{}\begin{array}{c c c c}
			0 & & &  \\
			0 & 0 &  & 2   \\
		\end{array}\quad \mbox{and}\quad  \begin{array}{c c c c}
			0 & & &  \\
			2 & 0 &  & 2   \\
		\end{array}$$ in $(A_3+A_1)''$, which have equivariant fundamental groups $\bZ/2\bZ$ by case \eqref{eq: E7L3}. Hence, these must be the problematic nilpotent orbits in $D_4+A_1$ with equivariant fundamental group $\bZ/2\bZ$, while the remainder of the problematic nilpotent orbits have trivial equivariant fundamental group. We omit for the moment a characterisation of these orbits in terms of partitions, but return to this in Remark~\ref{rmk: labs}. 
		
		For \eqref{eq: E7L8}, we note similarly to \eqref{eq: E7L6} that $$\overline{L}\cong\frac{\SO_{10}\times \SL_2}{\langle (\alpha_1^\vee(-1)\alpha_2^\vee(i)\alpha_4^\vee(-1)\alpha_5^\vee(-i),-I_2)\rangle}.$$ The argument then proceeds similarly to the argument for \eqref{eq: E7L6} and we omit the details (except to note that the complication from  \eqref{eq: E7L6} doesn't arise in this case).
		
		Finally, we consider \eqref{eq: E7L11}. For this case, we need some preparatory work. Following Table~\ref{ta: ADLevis}, we list below $\pi_1^{\Spin_{12}}(\O_L)\twoheadrightarrow \pi_1^{\tiny \PCO_{12}^\circ}(\O_L)$ for each relevant combination of properties a partition can have. As in Table~\ref{ta: ADLevis}, we set $a$ to be the number of distinct odd entries in the partition corresponding to $\O_L$. We also abbreviate the phrase ``central extension'' to ``c.e''.
		\renewcommand{\arraystretch}{2}
		\begin{center}
			\begin{longtable}{|c|c|}
				\hline
				& Evenly odd \\
				\hline
				Rather odd & $(\mbox{c.e. of }(\bZ/2\bZ)^{\max(a-1,0)}\mbox{ by }\bZ/2\bZ )\twoheadrightarrow (\bZ/2\bZ)^{\max(a-1,0)}$ \\
				\hline
				Not rather odd & $(\bZ/2\bZ)^{\max(a-1,0)}\twoheadrightarrow (\bZ/2\bZ)^{\max(a-1,0)}$  \\
				\hline
				& Not evenly odd \\
				\hline
				Rather odd &  $(\mbox{c.e. of }(\bZ/2\bZ)^{\max(a-1,0)}\mbox{ by }\bZ/2\bZ)\twoheadrightarrow (\bZ/2\bZ)^{\max(a-2,0)}$  \\
				\hline
				Not rather odd & $(\bZ/2\bZ)^{\max(a-1,0)}\twoheadrightarrow (\bZ/2\bZ)^{\max(a-2,0)}$ \\
				\hline
				\caption{Equivariant fundamental groups in type $D$}\label{ta: ADfunds}
			\end{longtable}
		\end{center}
		\renewcommand{\arraystretch}{}
		
		We see immediately that for partitions in the second (substantive) row, the map $\pi_1^{[L,L]}(\O_L)\twoheadrightarrow \pi_1^{L_\ad}(\O_L)$ is the identity map, and thus $\pi_1^{L_\ad}(\O_L)=\pi_1^{\overline{L}}(\O_L)$. For partitions in third (substantive) row, note that neither $a=0$ or $a=1$ can happen in this case, and thus the map $\pi_1^{[L,L]}(\O_L)\twoheadrightarrow \pi_1^{L_\ad}(\O_L)$ has kernel of degree 4; since $\vert \widehat{Z}\vert=4$, we must have in this case that $\pi_1^{\overline{L}}(\O_L)$ is a central extension of $\pi_1^{L_\ad}(\O_L)$ by $\bZ/2\bZ$.
		
		For partitions in the fourth row, we see that $\pi_1^{[L,L]}(\O_L)=\pi_1^{\tiny \SO_{12}}(\O_L)$ by Table~\ref{ta: ADLevis}; in particular, this means that $\alpha_2^\vee(-1)\alpha_3^\vee(-1)$ lies in $[L,L]_e^\circ$ even though $\alpha_2^\vee(-1)\alpha_3^\vee(-1)\notin Z$. Hence, by our usual argument, we have $\pi_1^{L_\ad}(\O_L)=\pi_1^{\overline{L}}(\O_L)$ in these cases.
		
		We finally consider partitions in the first row. Note that the only way a partition can be both rather odd and evenly odd is if the partition is very even (i.e. has only even entries). By Proposition~\ref{p: nilporbs}, there are two distinct orbits corresponding to each of these partitions. Following our usual argument, we need to check whether either of $\alpha_2^\vee(-1)\alpha_3^\vee(-1)$ or $\alpha_2^\vee(-1)\alpha_5^\vee(-1)\alpha_7^\vee(-1)$ lies in $[L,L]_e^\circ$. Since for orbits in this row we have $\pi_1^{[L,L]}(\O_L)\neq\pi_1^{\tiny \SO_{12}}(\O_L)$, we know $\alpha_2^\vee(-1)\alpha_3^\vee(-1)\notin [L,L]_e^\circ$; the question is thus whether $\alpha_2^\vee(-1)\alpha_5^\vee(-1)\alpha_7^\vee(-1) \in [L,L]_e^\circ$. Since $\alpha_2^\vee(-1)\alpha_3^\vee(-1)\notin [L,L]_e^\circ$ and the map $\pi_1^{[L,L]}(\O_L)\twoheadrightarrow \pi_1^{L_\ad}(\O_L)$ has kernel of degree 2, it is impossible that both $\alpha_2^\vee(-1)\alpha_5^\vee(-1)\alpha_7^\vee(-1)$ and $\alpha_3^\vee(-1)\alpha_5^\vee(-1)\alpha_7^\vee(-1)$ lie in $[L,L]_e^\circ$ and impossible that they both do not. Since these elements can be mapped to each other via an automorphism of $[L,L]$, the only possibility is that for one nilpotent orbit corresponding to a given very even partition we have $\pi_1^{\overline{L}}(\O_L)= \pi_1^{L_\ad}(\O_L)=1$ and for the other we have $\pi_1^{\overline{L}}(\O_L)=\bZ/2\bZ$.
		
		By examining the tables of \cite{EdG}, we see that the nilpotent orbits in simply connected $E_7$ induced from the nilpotent orbits with weighted Dynkin diagrams
		$$\renewcommand{\arraystretch}{}\begin{array}{c c c c c}
			& 0 & & &  \\
			2 & 2 & 0 & 2 & 0   \\
		\end{array},\quad\ \quad \begin{array}{c c c c c}
			& 0 &  & &\\
			2 & 0 & 0 & 2 & 0\\
		\end{array}\quad  \quad\mbox{and}\quad \quad \begin{array}{c c c c c}
			& 0 & &  &\\
			2 & 0 & 0 & 0 & 0 \\
		\end{array}$$
		in the unique standard Levi subgroup of type $D_6$ have $\pi_1^G(\O)=1$; this must therefore mean by Proposition~\ref{p: uniq brid} that $\pi_1^L(\O_L)=1$ for each of these induction data. Conversely, this must mean the nilpotent orbits with weighted Dynkin diagrams
		$$\renewcommand{\arraystretch}{}\begin{array}{c c c c c}
			& 2 & & &  \\
			0 & 2 & 0 & 2 & 0   \\
		\end{array},\quad\ \quad \begin{array}{c c c c c}
			& 2 &  & &\\
			0 & 0 & 0 & 2 & 0\\
		\end{array}\quad  \quad\mbox{and}\quad \quad \begin{array}{c c c c c}
			& 2 & &  &\\
			0 & 0 & 0 & 0 & 0 \\
		\end{array}$$
		in the unique standard Levi subgroup of type $D_6$ have $\pi_1^L(\O_L)=\bZ/2\bZ$.
	\end{proof}
	
	\begin{rmk}\label{rmk: labs}
		In Proposition~\ref{p: nilporbs} we observed that for $\Phi=D_n$ each very even partition corresponds to two distinct nilpotent orbits, which are distinguished by a label of $\I$ or $\II$. In general, we do not concern ourselves with precisely which orbit has which label, but the proof of Proposition~\ref{prop: E7 Levi Gps} shows that we do need to more careful for orbits in the Levi subgroups of types $D_4$ (really $D_4+A_1$) and $D_6$ in the semisimple simply connected group of type $E_7$. For $D_6$, we thus make the following convention (consistent with \cite[Lemma 5.3.5]{CM}): the nilpotent orbits with weighted Dynkin diagram $$\renewcommand{\arraystretch}{}\begin{array}{c c c c c}
			& 2 & & &  \\
			0 & 2 & 0 & 2 & 0   \\
		\end{array},\quad\ \quad \begin{array}{c c c c c}
			& 2 &  & &\\
			0 & 0 & 0 & 2 & 0\\
		\end{array}\quad  \quad\mbox{and}\quad \quad \begin{array}{c c c c c}
			& 2 & &  &\\
			0 & 0 & 0 & 0 & 0 \\
		\end{array}$$ (corresponding to the partitions $(6^2)$, $(4^2,2^2)$ and $(2^6)$, respectively) are labelled with $\I$ and the nilpotent orbits with weighted Dynkin diagram 
		$$\renewcommand{\arraystretch}{}\begin{array}{c c c c c}
			& 0 & & &  \\
			2 & 2 & 0 & 2 & 0   \\
		\end{array},\quad\ \quad \begin{array}{c c c c c}
			& 0 &  & &\\
			2 & 0 & 0 & 2 & 0\\
		\end{array}\quad  \quad\mbox{and}\quad \quad \begin{array}{c c c c c}
			& 0 & &  &\\
			2 & 0 & 0 & 0 & 0 \\
		\end{array}$$
		(corresponding to the same partitions) are labelled with $\II$.
		
		As noted, we have a similar issue for $D_4$ but, as we saw when analysing case \eqref{eq: E7L6} in the proof of Proposition~\ref{prop: E7 Levi Gps}, the situation is in fact even worse. The orbits corresponding to the partitions $(2^4)_\I$, $(2^4)_\II$ and $(3,1^5)$ all have weighted Dynkin diagrams consisting of a 2 on a single node of valence one and zeroes elsewhere. To distinguish orbits by their weighted Dynkin diagram, we must thus fix an ``orientation'' of the $D_4$; we choose the one compatible with the induction from the standard Levi subgroup of type $D_4$ into the unique standard Levi subgroup of type $D_6$ (and with the conventions for that subgroup just established). More explicitly, we make the following identifications:
		$$(2^4)_\I = \renewcommand{\arraystretch}{}\begin{array}{c c c}
			& 2 & \\
			0 & 0 & 0    \\
		\end{array}, \quad (2^4)_\II = \renewcommand{\arraystretch}{}\begin{array}{c c c}
			& 0 & \\
			2 & 0 & 0    \\
		\end{array}, \quad\mbox{and}\quad (3,1^5) =  \renewcommand{\arraystretch}{}\begin{array}{c c c}
			& 0 & \\
			0 & 0 & 2    \\
		\end{array}.$$ We also make the following identifications, where we would otherwise have a similar issue:
		$$(4^2)_\I = \renewcommand{\arraystretch}{}\begin{array}{c c c}
			& 2 & \\
			0 & 2 & 0    \\
		\end{array}, \quad (4^2)_\II = \renewcommand{\arraystretch}{}\begin{array}{c c c}
			& 0 & \\
			2 & 2 & 0    \\
		\end{array}, \quad\mbox{and}\quad (5,1^3)= \renewcommand{\arraystretch}{}\begin{array}{c c c}
			& 0 & \\
			0 & 2 & 2    \\
		\end{array}.$$
		This is also be a problem in types $E_6$ and $E_8$, but since we never have occasion to write these weighted Dynkin diagrams for the Levi subgroup of type $D_4$ in simply connected $E_6$ or $E_8$, we decline to establish any particular convention in that case.
		
		The reader should note in particular that caution is warranted when inducing from the standard Levi subgroup of type $D_4$ to the standard Levi subgroup of type $D_5$ corresponding to the simple roots labelled $\{1,2,3,4,5\}$ and especially warranted when inducing from the standard Levi subgroup of type $D_4+A_1$ to the standard Levi subgroup of type $D_5+A_1$.

	\end{rmk}
	\renewcommand{\arraystretch}{2}
	\begin{center}
		\begin{longtable}{|c|c|c|c|}
			\hline
			Levi type & Nilpotent orbit (partition) & Nilpotent orbit (diagram) &  $\pi_1^L(\O_L)$\\
			\hline
			$(3A_1)''$ & $(2)\times (2)\times (2)$ & $\renewcommand{\arraystretch}{}\begin{array}{c c c c c c}
				&  & 2 & & &  \\
				&  &  & 2  &  & 2   \\
			\end{array}$ &  $\bZ/2\bZ$ \\
			\hline
			$4A_1$ & $(1^2)\times (2)\times (2)\times (2)$ & $\renewcommand{\arraystretch}{}\begin{array}{c c c c c c}
				&  & 2 & & &  \\
				0 &  &  & 2  &  & 2   \\
			\end{array}$ &  $\bZ/2\bZ$ \\
			$(\dagger)$ &  $(2)\times (2)\times (2)\times (2)$ & $\renewcommand{\arraystretch}{}\begin{array}{c c c c c c}
				&  & 2 & & &  \\
				2 &  &  & 2  &  & 2   \\
			\end{array}$ &  $\bZ/2\bZ$ \\
			\hline
			$(A_3+A_1)''$ & $(2^2)\times (2)$ & $\renewcommand{\arraystretch}{}\begin{array}{c c c c c c}
				&  & 0 & & & \\
				&  & 2 & 0 &  & 2   \\
			\end{array}$ & $\bZ/2\bZ$\\
			$(\dagger)$ & $(4)\times (2)$ & $\renewcommand{\arraystretch}{}\begin{array}{c c c c c c}
				&  & 2 & & & \\
				&  & 2 & 2 &  & 2   \\
			\end{array}$ & $\bZ/2\bZ$\\
			\hline
			$A_2+3A_1$ & $(1^3)\times (2)\times (2)\times (2)$ & $\renewcommand{\arraystretch}{}\begin{array}{c c c c c c}
				&  & 2 & & & \\
				0 	& 0 &  & 2 &  & 2   \\
			\end{array}$ & $\bZ/2\bZ$\\
			$(\dagger)$ & $(2,1)\times (2)\times (2)\times (2)$ & $\renewcommand{\arraystretch}{}\begin{array}{c c c c c c}
				&  & 2 & & & \\
				1	& 1 &  & 2 &  & 2   \\
			\end{array}$ & $\bZ/2\bZ$\\
			$(\dagger)$ & $(3)\times (2)\times (2)\times (2)$ & $\renewcommand{\arraystretch}{}\begin{array}{c c c c c c}
				&  & 2 & & & \\
				2	& 2 &  & 2 &  & 2   \\
			\end{array}$ & $\bZ/2\bZ$\\
			\hline
			$A_3+2A_1$ & $(2^2)\times(1^2)\times (2)$ & $\renewcommand{\arraystretch}{}\begin{array}{c c c c c c}
				&  & 0 & & & \\
				0 &  & 2 & 0 &  & 2   \\
			\end{array}$ & $\bZ/2\bZ$\\
			$(\dagger)$ & $(4)\times (1^2)\times (2)$ & $\renewcommand{\arraystretch}{}\begin{array}{c c c c c c}
				&  & 2 & & & \\
				0 &  & 2 & 2 &  & 2   \\
			\end{array}$ & $\bZ/2\bZ$\\
			$(\dagger)$ & $(2^2)\times (2)\times (2)$ & $\renewcommand{\arraystretch}{}\begin{array}{c c c c c c}
				&  & 0 & & & \\
				2 &  & 2 & 0 &  & 2   \\
			\end{array}$ & $\bZ/2\bZ$\\
			$(\dagger)$ & $(4)\times (2)\times (2)$ & $\renewcommand{\arraystretch}{}\begin{array}{c c c c c c}
				&  & 2 & & & \\
				2 &  & 2 & 2 &  & 2   \\
			\end{array}$ & $\bZ/2\bZ$\\
			\hline
			
			$D_4+A_1$ & $(2^4)_\II\times (2)$ & $\renewcommand{\arraystretch}{}\begin{array}{c c c c c c}
				&  & 0 & & & \\
				& 2 & 0 & 0 &  & 2   \\
			\end{array}$ & $\bZ/2\bZ$\\
			& $(3,2^2,1)\times (2)$ & $\renewcommand{\arraystretch}{}\begin{array}{c c c c c c}
				&  & 1 & & & \\
				& 1 & 0 & 1 &  & 2   \\
			\end{array}$ & $\bZ/2\bZ$\\
			$(\dagger)$ & $(4^2)_\II\times (2)$ & $\renewcommand{\arraystretch}{}\begin{array}{c c c c c c}
				&  & 0 & & & \\
				& 2 & 2 & 0 &  & 2   \\
			\end{array}$ & $\bZ/2\bZ$\\
			$(\dagger)$ & $(5,3)\times (2)$ & $\renewcommand{\arraystretch}{}\begin{array}{c c c c c c}
				&  & 2 & & & \\
				& 2 & 0 & 2 &  & 2   \\
			\end{array}$ & $\bZ/2\bZ$\\
			$(\dagger)$ & $(7,1)\times (2)$ & $\renewcommand{\arraystretch}{}\begin{array}{c c c c c c}
				&  & 2 & & & \\
				& 2 & 2 & 2 &  & 2   \\
			\end{array}$ & $\bZ/2\bZ$\\
			\hline
			$(A_5)''$ & $(2^3)$ & $\renewcommand{\arraystretch}{}\begin{array}{c c c c c c}
				&  & 0 & & & \\
				&  & 0 & 2 & 0 & 0   \\
			\end{array}$ & $\bZ/2\bZ$\\
			$(\dagger)$ & $(4,2)$ & $\renewcommand{\arraystretch}{}\begin{array}{c c c c c c}
				&  & 2 & & & \\
				&  & 0 & 2 & 0 & 2   \\
			\end{array}$ & $\bZ/2\bZ$\\
			$(\dagger)$ & $(6)$ & $\renewcommand{\arraystretch}{}\begin{array}{c c c c c c}
				&  & 2 & & & \\
				&  & 2 & 2 & 2 & 2   \\
			\end{array}$ & $\bZ/2\bZ$\\
			\hline
			$D_5+A_1$ & $(3,1^7)\times (2)$ & $\renewcommand{\arraystretch}{}\begin{array}{c c c c c c}
				&  & 0 & & & \\
				2 & 0 & 0 & 0  & & 2   \\
			\end{array}$ & $\bZ/2\bZ$ \\
			& $(3,2^2,1^3)\times (2)$ & $\renewcommand{\arraystretch}{}\begin{array}{c c c c c c}
				&  & 0 & & & \\
				1 & 0 & 1 & 0 & & 2   \\
			\end{array}$ & $\bZ/2\bZ$ \\
			& $(3^3,1)\times (2)$ & $\renewcommand{\arraystretch}{}\begin{array}{c c c c c c}
				&  & 0 & & & \\
				0 & 0 & 2 & 0 & & 2   \\
			\end{array}$ & $\bZ/2\bZ$ \\
			$(\dagger)$ & $(5,1^5)\times (2)$ & $\renewcommand{\arraystretch}{}\begin{array}{c c c c c c}
				&  & 0 & & & \\
				2 & 2 & 0 & 0 & & 2   \\
			\end{array}$ & $\bZ/2\bZ$ \\
			$(\dagger)$ & $(5,2^2,1)\times (2)$ & $\renewcommand{\arraystretch}{}\begin{array}{c c c c c c}
				&  & 1 & & & \\
				2 & 1 & 0 & 1 & & 2   \\
			\end{array}$ & $\bZ/2\bZ$ \\
			$(\dagger)$ & $(5,3,1^2)\times (2)$ & $\renewcommand{\arraystretch}{}\begin{array}{c c c c c c}
				&  & 0 & & & \\
				2 & 0 & 2 & 0 & & 2   \\
			\end{array}$ & $\Sym_2\times \bZ/2\bZ$ \\
			$(\dagger)$ & $(7,1^3)\times (2)$ & $\renewcommand{\arraystretch}{}\begin{array}{c c c c c c}
				&  & 0 & & & \\
				2 & 2 & 2 & 0 & & 2   \\
			\end{array}$ & $\bZ/2\bZ$ \\
			$(\dagger)$ & $(7,3)\times (2)$ & $\renewcommand{\arraystretch}{}\begin{array}{c c c c c c}
				&  & 2 & & & \\
				2 & 2 & 0 & 2 & & 2   \\
			\end{array}$ & $\bZ/2\bZ$ \\
			$(\dagger)$ & $(9,1)\times (2)$ & $\renewcommand{\arraystretch}{}\begin{array}{c c c c c c}
				&  & 2 & & & \\
				2 & 2 & 2 & 2 & & 2   \\
			\end{array}$ & $\bZ/2\bZ$ \\
			\hline
			$A_3+A_2+A_1$ & $(2^2)\times (1^3)\times (2)$ & $\renewcommand{\arraystretch}{}\begin{array}{c c c c c c}
				&  & 2 & & & \\
				0 & 0  &  & 0 & 2 & 0   \\
			\end{array}$ & $\bZ/2\bZ$\\
			$(\dagger)$ & $(4)\times (1^3)\times (2)$ & $\renewcommand{\arraystretch}{}\begin{array}{c c c c c c}
				&  & 2 & & & \\
				0 & 0 &  & 2 & 2 & 2   \\
			\end{array}$ & $\bZ/2\bZ$\\
			$(\dagger)$ & $(2^2)\times (2,1)\times (2)$ & $\renewcommand{\arraystretch}{}\begin{array}{c c c c c c}
				&  & 2 & & & \\
				1 &  1 &  & 0 & 2 & 0   \\
			\end{array}$ & $\bZ/2\bZ$\\
			$(\dagger)$ & $(4)\times (2,1)\times (2)$ & $\renewcommand{\arraystretch}{}\begin{array}{c c c c c c}
				&  & 2 & & & \\
				1 & 1 &  & 2 & 2 & 2   \\
			\end{array}$ & $\bZ/2\bZ$\\
			$(\dagger)$ & $(2^2)\times (3)\times (2)$ & $\renewcommand{\arraystretch}{}\begin{array}{c c c c c c}
				&  & 2 & & & \\
				2 & 2  & & 0 & 2 & 0   \\
			\end{array}$ & $\bZ/2\bZ$\\
			$(\dagger)$ & $(4)\times (3)\times (2)$ & $\renewcommand{\arraystretch}{}\begin{array}{c c c c c c}
				&  & 2 & & & \\
				2 & 2 & & 2 & 2 & 2   \\
			\end{array}$ & $\bZ/2\bZ$\\
			\hline
			$A_5+A_1$ & $(2^3)\times (1^2)$ & $\renewcommand{\arraystretch}{}\begin{array}{c c c c c c}
				&  & 0 & & & \\
				0 &  & 0 & 2 & 0 & 0   \\
			\end{array}$ & $\bZ/2\bZ$\\
			$(\dagger)$ & $(4,2)\times (1^2) $ & $\renewcommand{\arraystretch}{}\begin{array}{c c c c c c}
				&  & 2 & & & \\
				0 &  & 0 & 2 & 0 & 2   \\
			\end{array}$ & $\bZ/2\bZ$\\
			$(\dagger)$ & $(6)\times(1^2)$ & $\renewcommand{\arraystretch}{}\begin{array}{c c c c c c}
				&  & 2 & & & \\
				0 &  & 2 & 2 & 2 & 2   \\
			\end{array}$ & $\bZ/2\bZ$\\
			$(\dagger)$ & $(2^3)\times (2)$ & $\renewcommand{\arraystretch}{}\begin{array}{c c c c c c}
				&  & 0 & & & \\
				2 &  & 0 & 2 & 0 & 0   \\
			\end{array}$ & $\bZ/2\bZ$\\
			$(\dagger)$ & $(4,2)\times (2) $ & $\renewcommand{\arraystretch}{}\begin{array}{c c c c c c}
				&  & 2 & & & \\
				2 &  & 0 & 2 & 0 & 2   \\
			\end{array}$ & $\bZ/2\bZ$\\
			$(\dagger)$ & $(6)\times (2)$ & $\renewcommand{\arraystretch}{}\begin{array}{c c c c c c}
				&  & 2 & & & \\
				2 &  & 2 & 2 & 2 & 2   \\
			\end{array}$ & $\bZ/2\bZ$\\
			\hline
			$D_6$ & $(2^6)_\I$ & $\renewcommand{\arraystretch}{}\begin{array}{c c c c c c}
				&  & 2 & & & \\
				& 0 & 0 & 0 & 0 & 0   \\
			\end{array}$ & $\bZ/2\bZ$ \\
			& $(3,2^4,1)$ & $\renewcommand{\arraystretch}{}\begin{array}{c c c c c c}
				&  & 1 & & & \\
				& 1 & 0 & 0 & 0 & 1   \\
			\end{array}$ & $\bZ/2\bZ$ \\
			& $(4^2,2^2)_\I$ & $\renewcommand{\arraystretch}{}\begin{array}{c c c c c c}
				&  & 2 & & & \\
				& 0 & 0 & 0 & 2 & 0   \\
			\end{array}$ & $\bZ/2\bZ$ \\
			$(\dagger)$ & $(4^2,3,1)$ & $\renewcommand{\arraystretch}{}\begin{array}{c c c c c c}
				&  & 1 & & & \\
				& 1 & 0 & 1 & 1 & 0   \\
			\end{array}$ & $\bZ/2\bZ$ \\
			$(\dagger)$ & $(5,3,2^2)$ & $\renewcommand{\arraystretch}{}\begin{array}{c c c c c c}
				&  & 1 & & & \\
				& 1 & 0 & 1 & 0 & 2  \\
			\end{array}$ & $\bZ/2\bZ$ \\
			$(\dagger)$ & $(6^2)_\I$ & $\renewcommand{\arraystretch}{}\begin{array}{c c c c c c}
				&  & 2 & & & \\
				& 0 & 2 & 0 & 2 & 0   \\
			\end{array}$ & $\bZ/2\bZ$ \\
			$(\dagger)$ & $(7,2^2,1)$ & $\renewcommand{\arraystretch}{}\begin{array}{c c c c c c}
				&  & 1 & & & \\
				& 1 & 0 & 1 & 2 & 2   \\
			\end{array}$ & $\bZ/2\bZ$ \\
			$(\dagger)$ & $(7,5)$ & $\renewcommand{\arraystretch}{}\begin{array}{c c c c c c}
				&  & 2 & & & \\
				& 2 & 0 & 2 & 0 & 2   \\
			\end{array}$ & $\bZ/2\bZ$ \\
			$(\dagger)$ & $(9,3)$ & $\renewcommand{\arraystretch}{}\begin{array}{c c c c c c}
				&  & 2 & & & \\
				& 2 & 0 & 2 & 2 & 2   \\
			\end{array}$ & $\bZ/2\bZ$ \\
			$(\dagger)$ & $(11,1)$ & $\renewcommand{\arraystretch}{}\begin{array}{c c c c c c}
				&  & 2 & & & \\
				& 2 & 2 & 2 & 2 & 2   \\
			\end{array}$ & $\bZ/2\bZ$ \\
			\hline
			\caption{Equivariant fundamental groups for nilpotent orbits of standard Levi subgroups in simply connected $E_7$}\label{ta: E7Levis}
		\end{longtable}
	\end{center} 

	\section{Case-by-case arguments}\label{s: CbC}
	\renewcommand{\arraystretch}{}
	
	In this section $G$ is a semisimple simply connected algebraic group with indecomposable root system of exceptional type. For each nilpotent $G$-orbit $\O$, we determine the birationally rigid induction datum for each nilpotent cover of $\O$. We proceed based on the isomorphism type of $\pi_1^G(\O)$, with the conventions of Remark~\ref{rmk: orbs as covs}; with such conventions, there are 10 groups which can appear. These are $1$, $\bZ/2\bZ$, $\Sym_2$, $\bZ/3\bZ$, $\Sym_2\times\bZ/2\bZ$, $\Sym_2\times \bZ/3\bZ$, $\Sym_3$, $\Sym_3\times \bZ/2\bZ$, $\Sym_4$ and $\Sym_5$; note that we treat $\bZ/2\bZ$ and $\Sym_2$ separately for consistency with \cite[\S8.4]{CM} and to reflect whether or not such component exists when $G$ is of adjoint type (see Remark~\ref{rmk: orbs as covs}). As in Tables~\ref{ta: E6Levis} and \ref{ta: E7Levis}, we maintain this convention when writing $\pi_1^L(\O_L)$ for an induction datum $(L,\O_L)$, in that we write $\bZ/2\bZ$ for components which arise in $\pi_1^L(\O_L)$ only when $G$ is simply connected and write $\Sym_2$ for components which arise in $\pi_1^L(\O_L)$ in both the simply connected and adjoint cases (since $G$ is exceptional, these are the only possible isogeny types).
	
	\iffalse
	In order to determine the $L$-equivariant fundamental group $\pi_1^L(\O_L)$ of an induction datum $(L,\O_L)$ for $\O$, there are a couple of techniques we can use. One powerful tool is the {\fontfamily{pcr}\selectfont atlas} software \cite{At}; we use the command {\fontfamily{pcr}\selectfont Levi\textunderscore datum} to isolate the Levi subgroup, and the command {\fontfamily{pcr}\selectfont component\textunderscore datum(orbit).orders} to get the orders of elements of $\pi_1^L(\O_L)$. More theoretically, we note that there are surjections \begin{equation}\label{e: piforms}
		\pi_1^{[L,L]_{\sc}}(\O_L)\twoheadrightarrow \pi_1^L(\O_L)\twoheadrightarrow \pi_1^{[L,L]_{\ad}}(\O_L),
	\end{equation}
	where $[L,L]_{\sc}$ (resp. ($[L,L]_{\ad}$) denotes the simply connected (resp. adjoint) form of the derived subgroup $[L,L]$ of $L$ (see \cite[\S 1.2, \S 6.1]{CM} for more details). When $\pi_1^{[L,L]_{\sc}}(\O_L)=\pi_1^{[L,L]_{\ad}}(\O_L)$ this clearly allows us to determine the $L$-equivariant fundamental group $\pi_1^L(\O_L)$. 
	\fi
	
	We make some further notational conventions. By conjugation, we may assume that all the Levi subgroups we consider are standard Levi subgroups. As discussed in Subsection~\ref{ss: AlgGps}, these may be described by Dynkin type (uniquely in most cases, although in some cases we will need additional decoration as discussed in that subsection); we therefore often refer to Levi subgroups via their Dynkin types. For example, we may write $\Ind_{D_4}^{E_8}$ to mean induction from the (unique to to conjugacy) standard Levi subgroup of $G$ with Dynkin type $D_4$ to the standard Levi subgroup of $G$ with Dynkin type $E_8$. Based on these conventions, we often use the phrase ``equivariant fundamental group'' in this section rather than ``$L$-equivariant fundamental group'' when $L$ is clear, in order to avoid undue clutter. As in Subsection~\ref{ss: NilpOrbs}, we also use the phrase ``universal cover'' of a nilpotent $L$-orbit to mean the $L$-equivariant cover of that orbit corresponding to the trivial subgroup of its equivariant fundamental group (i.e. the universal $L$-equivariant cover of said nilpotent orbit).
	
	This section is structured as follows. We have one subsection for each group which can appear as $\pi_1^G(\O)$. At the beginning of each of these subsections we discuss how many nilpotent $G$-orbits in each exceptional type have such $G$-equivariant fundamental group, and we resolve as many cases as we can using class-wide arguments. For those nilpotent $G$-orbits which can't easily be dealt with using general arguments, we instead make orbit-by-orbit arguments in following (sub-)subsections. The results of this section are compiled into tables in Section \ref{s: Tables}.
	
	We conclude the introduction to this section by noting that throughout the following subsections we prefer to cite to results as labelled in Sections~\ref{s: Prelim} and \ref{s: CompGrps} rather than to their original sources, for ease of reference for the reader. Nonetheless, we emphasise that many of these results are due to other authors, and refer the reader back to Sections~\ref{s: Prelim} and \ref{s: CompGrps} for the proper citations.

	\subsection{\texorpdfstring{$\pi_1^G(\O)=1$}{pi=1}}\label{ss: pi=1}
	
	There are 45 induced nilpotent orbits in exceptional cases with $\pi_1^G(\O)=1$; that is to say, there is 1 such orbit for $G_2$, 4 such orbits for $F_4$, 10 such orbits for $E_6$, 8 such orbits\footnote{Note that the tables in \cite[\S 8.4]{CM} list 9 induced orbits with $\pi_1^G(\O)=1$ for simply connected $G$ of type $E_7$, however this is due to erroneously stating that the nilpotent orbit $(A_5)''$ has this property; in fact, this nilpotent orbit has $\pi_1^G(\O)=\bZ/2\bZ$.} for $E_7$ and 22 such orbits for $E_8$. That $\pi_1^G(\O)=1$ means that $\O$ has no non-trivial nilpotent covers. Since any rigid induction datum for $\O$ must birationally induce to a nilpotent cover of $\O$, we must get that each $\O$ with $\pi_1^G(\O)=1$ has a unique rigid induction datum (we can also check this directly from \cite{EdG}) and that $\O$ must be birationally induced from this rigid (and thus birationally rigid) induction datum.
	
	\subsection{\texorpdfstring{$\pi_1^G(\O)=\bZ/2\bZ$}{pi=Z2}}\label{ss: pi=Z2}
	
	This case can only arise in type $E_7$; in this case, there are 16 induced nilpotent orbits with $\pi_1^G(\O)=\bZ/2\bZ$. That $\pi_1^G(\O)=\bZ/2\bZ$ implies that the only non-trivial nilpotent cover of $\O$ is the universal cover $\hatO$. If the universal cover is birationally rigid, then arguing as in Subsection~\ref{ss: pi=1} shows that $\O$ has a unique rigid induction datum and that $\O$ is birationally induced from this rigid induction datum. This handles 7 nilpotent orbits (using Proposition~\ref{p: bsr} to see which induced nilpotent orbits for simply connected $E_7$ have birationally rigid nilpotent covers).
	
	We tackle the remaining 9 nilpotent orbits individually.
	
	\subsubsection{\texorpdfstring{$D_4+A_1\subseteq E_7$}{D4A1 in E7}}\label{sss: D4A1 in E7}
	
	The nilpotent orbit in $E_7$ with Bala-Carter label $D_4+A_1$ has weighted Dynkin diagram $$\begin{array}{c c c c c c}
		& & 1 & & & \\
		2 & 1 & 0 & 0 & 0 & 1 \\
	\end{array}$$ and unique rigid induction datum
	$$\begin{array}{c c c c c}
		& 1 & & & \\
		1 & 0 & 0 & 0 & 1 \\
	\end{array}\subseteq D_6.$$ 
	Note that this rigid induction datum has equivariant fundamental group $\bZ/2\bZ$ by Table~\ref{ta: E7Levis}; the universal cover of this nilpotent orbit must therefore also be birationally rigid. We thus have $$\O=\BInd_{D_6}^{E_7}\left( \begin{array}{c c c c c}
		& 1 & & & \\
		1 & 0 & 0 & 0 & 1 \\
	\end{array} \right)$$ and $$\hatO=\BInd_{D_6}^{E_7}\left( \mbox{univ. cover of }\begin{array}{c c c c c}
		& 1 & & & \\
		1 & 0 & 0 & 0 & 1 \\
	\end{array} \right).$$

	\subsubsection{\texorpdfstring{$(A_5)''\subseteq E_7$}{A5'' in E7}}\label{sss: A5'' in E7}
	
	The nilpotent orbit in $E_7$ with Bala-Carter label $(A_5)''$ has weighted Dynkin diagram $$\begin{array}{c c c c c c}
		& & 0 & & & \\
		2 & 0 & 0 & 0 & 2 & 2 \\
	\end{array}$$ and unique rigid induction datum
	$$\begin{array}{c c c}
		& 0 & \\
		0 & 0 & 0  \\
	\end{array}\subseteq D_4.$$
	By Remark~\ref{rmk: brid covs} and Table~\ref{ta: E7Levis}, the universal cover of $$\begin{array}{c c c c c c}
		& & 0 & & & \\
		2 & 0 & 0 & 0 & & 2 \\
	\end{array}\subseteq D_5+A_1$$ is a birationally rigid induction datum for $\O$. We must therefore have
	$$\O = \BInd_{D_4}^{E_7}\left(\begin{array}{c c c}
		& 0 & \\
		0 & 0 & 0  \\
	\end{array} \right)$$
	and
	$$\hatO = \BInd_{D_5+ A_1}^{E_7}\left( \mbox{univ. cover of }\begin{array}{c c c c c c}
		& & 0 & & & \\
		2 & 0 & 0 & 0 & & 2 \\
	\end{array} \right).$$

	\subsubsection{\texorpdfstring{$D_6(a_2)\subseteq E_7$}{D6a2 in E6}}\label{sss: D6a2 in E6}
	
	The nilpotent orbit in $E_7$ with Bala-Carter label $D_6(a_2)$ has weighted Dynkin diagram $$\begin{array}{c c c c c c}
		& & 1 & & & \\
		0 & 1 & 0 & 1 & 0 & 2 \\
	\end{array}$$ and unique rigid induction datum
	$$\begin{array}{c c c c}
		& & 0 & \\
		1 & 0 & 1 & 0 \\
	\end{array}\subseteq D_5.$$ By Remark~\ref{rmk: brid covs} and Table~\ref{ta: E7Levis}, the universal cover of $$\begin{array}{c c c c c c}
		& & 0 & & & \\
		1 & 0 & 1 & 0 & & 2 \\
	\end{array}\subseteq D_5+A_1$$ is a birationally rigid induction datum for $\O$. We must therefore have
	$$\O = \BInd_{D_5}^{E_7}\left(\begin{array}{c c c c}
		& & 0 & \\
		1 & 0 & 1 & 0 \\
	\end{array} \right)$$
	and
	$$\hatO = \BInd_{D_5+ A_1}^{E_7}\left( \mbox{univ. cover of }\begin{array}{c c c c c c}
		& & 0 & & & \\
		1 & 0 & 1 & 0 & & 2 \\
	\end{array} \right).$$
	
	\iffalse
	 This orbit has trivial equivariant fundamental group by Proposition~\ref{prop: E7 Levi Gps} and Table~\ref{ta: ADLevis} (noting that it corresponds to the partition $(3,2^2,1^3)$). Induction from this rigid induction datum passes through the induction datum 
	$$\begin{array}{c c c c c c}
		& & 0 & & & \\
		1 & 0 & 1 & 0 & & 2\\
	\end{array}\subseteq D_5 + A_1,$$ which has equivariant fundamental group $\bZ/2\bZ$ by Table~\ref{ta: E7Levis}. By Theorem~\ref{th: isogindep} and Proposition~\ref{p: birigid orbs}, this nilpotent orbit is birationally induced; its universal cover must therefore be birationally rigid. We hence have $$\O=\BInd_{D_5}^{E_7}\left(\begin{array}{c c c c}
		& & 0 & \\
		1 & 0 & 1 & 0 \\
	\end{array} \right)$$ and $$\hatO=\BInd_{D_5+ A_1}^{E_7}\left(\mbox{univ. cover of }\begin{array}{c c c c c c}
		& & 0 & & & \\
		1 & 0 & 1 & 0 & & 2\\
	\end{array} \right).$$
	\fi
	
	\subsubsection{\texorpdfstring{$D_5+A_1\subseteq E_7$}{D5A1 in E7}}\label{sss: D5A1 in E7}
	
	The nilpotent orbit in $E_7$ with Bala-Carter label $D_5+A_1$ has weighted Dynkin diagram $$\begin{array}{c c c c c c}
		& & 1 & & & \\
		2 & 1 & 0 & 1 & 1 & 0 \\
	\end{array}$$ and unique rigid induction datum
	$$\begin{array}{c c c c}
		0 & & & \\
		0 & & 0 & 0 \\
	\end{array}\subseteq 2 A_2.$$ By Remark~\ref{rmk: brid covs} and Table~\ref{ta: E7Levis}, the universal cover of $$\begin{array}{c c c c}
	0 & & & \\
	  0 & 2 & 0 & 0 \\
	\end{array}\subseteq (A_5)''$$ is a birationally rigid induction datum for $\O$. We must therefore have
	$$\O = \BInd_{2A_2}^{E_7}\left(\begin{array}{c c c c}
		0 & & & \\
		0 & & 0 & 0 \\
	\end{array} \right).$$
	and
	$$\hatO = \BInd_{(A_5)''}^{E_7}\left( \mbox{univ. cover of }\begin{array}{c c c c}
	0 & & & \\
	0 & 2 & 0 & 0 \\
	\end{array} \right).$$
	
	\iffalse

	This nilpotent orbit clearly has trivial equivariant fundamental group, but induction from this rigid induction datum passes through the induction datum 
	$$\begin{array}{c c c c}
		0 & & & \\
		0 & 2 & 0 & 0\\
	\end{array}\subseteq (A_5)'',$$ which has equivariant fundamental group $\bZ/2\bZ$ by Table~\ref{ta: E7Levis}. By Theorem~\ref{th: isogindep} and Proposition~\ref{p: birigid orbs}, this nilpotent orbit is birationally induced; its universal cover must therefore be birationally rigid. We hence have $$\O=\BInd_{2 A_2}^{E_7}\left( \begin{array}{c c c c}
		0& & & \\
		0 & & 0 & 0 \\
	\end{array} \right)$$ and $$\hatO=\BInd_{(A_5)''}^{E_7}\left(\mbox{univ. cover of } \begin{array}{c c c c}
		0 & & & \\
		0 & 2 & 0 & 0\\
	\end{array}\right).$$
	\fi
	
	\subsubsection{\texorpdfstring{$D_6(a_1)\subseteq E_7$}{D6a1 in E7}}\label{sss: D6a1 in E7}
	
	The nilpotent orbit in $E_7$ with Bala-Carter label $D_6(a_1)$ has weighted Dynkin diagram $$\begin{array}{c c c c c c}
		& & 1 & & & \\
		2 & 1 & 0 & 1 & 0 & 2 \\
	\end{array}$$ and unique rigid induction datum
	$$\begin{array}{c c c}
		0 & 0 & 0\\
	\end{array}\subseteq A_3.$$ By Remark~\ref{rmk: brid covs} and Table~\ref{ta: E7Levis}, the universal cover of $$\begin{array}{c c c c c}
		& 0 & & & \\
		2 & 0 & 0 &  & 2 \\
	\end{array}\subseteq D_4+A_1$$ is a birationally rigid induction datum for $\O$. We must therefore have
	$$\O = \BInd_{A_3}^{E_7}\left(\begin{array}{c c c}
		0 & 0 & 0\\
	\end{array}\right)$$
	and
	$$\hatO = \BInd_{D_4+ A_1}^{E_7}\left( \mbox{univ. cover of }\begin{array}{c c c c c}
		& 0 & & & \\
		2 & 0 & 0 & & 2 \\
	\end{array} \right).$$
	\iffalse	
	This nilpotent orbit clearly has trivial equivariant fundamental group. Furthermore, when this orbit is induced into $D_4$ it corresponds to the partition $(3,1^5)$, $(2^4)_\I$ or $(2^4)_{\II}$ (depending on how one embeds the $A_3=D_3$ Levi subgroup into $D_4$) and when this orbit is induced into $(A_3+A_1)''$ it corresponds to the partition $(1^4)\times (2)$; in all of these cases, the equivariant fundamental group is trivial by Proposition~\ref{prop: E7 Levi Gps} and Table~\ref{ta: ADLevis}. Induction from the unique rigid induction datum then passes through the induction datum 
	$$\begin{array}{c c c c c}
		& 0 & & & \\
		2 & 0 & 0 & & 2\\
	\end{array}\subseteq D_4+A_1,$$ corresponding to the partition $(2^4)_\II\times (2)$, which by Table~\ref{ta: E7Levis} has equivariant fundamental group $\bZ/2\bZ$ (see also Remark~\ref{rmk: labs}). By Theorem~\ref{th: isogindep} and Proposition~\ref{p: birigid orbs}, this nilpotent orbit is birationally induced; its universal cover must therefore be birationally rigid. We hence have $$\O=\BInd_{A_3}^{E_7}\left( \begin{array}{c c c}
		0 & 0 & 0\\
	\end{array} \right)$$ and $$\hatO=\BInd_{D_4+ A_1}^{E_7}\left(\mbox{univ. cover of }\begin{array}{c c c c c c}
		& 0 & & & \\
		2 & 0 & 0 & & 2\\
	\end{array}\right).$$
	\fi

	\subsubsection{\texorpdfstring{$D_6\subseteq E_7$}{D6 in E7}}\label{sss: D6 in E7}
	
	The nilpotent orbit in $E_7$ with Bala-Carter label $D_6$ has weighted Dynkin diagram $$\begin{array}{c c c c c c}
		& & 1 & & & \\
		2 & 1 & 0 & 1 & 2 & 2 \\
	\end{array}$$ and unique rigid induction datum
	$$\begin{array}{c c c}
		& 1 & \\
		1 & 0 & 1\\
	\end{array}\subseteq D_4.$$ 
	By Remark~\ref{rmk: brid covs} and Table~\ref{ta: E7Levis}, the universal cover of $$\begin{array}{c c c c c}
		& 1 & & & \\
		 1 & 0 & 1 & & 2 \\
	\end{array}\subseteq D_5+A_1$$ is a birationally rigid induction datum for $\O$. We must therefore have
	$$\O = \BInd_{D_4}^{E_7}\left( \begin{array}{c c c}
		& 1 & \\
		1 & 0 & 1\\
	\end{array}\right).$$
	and
	$$\hatO = \BInd_{D_4+ A_1}^{E_7}\left( \mbox{univ. cover of }\begin{array}{c c c c c}
		 & 1 & & & \\
		 1 & 0 & 1 & & 2 \\
	\end{array} \right).$$
	
	\iffalse
	
	 This orbit has trivial equivariant fundamental group by Proposition~\ref{prop: E7 Levi Gps} and Table~\ref{ta: ADLevis} (noting that it corresponds to the partition $(3,2^2,1)$).
	Induction from this rigid induction datum passes through the induction datum 
	$$\begin{array}{c c c c c c}
		& 1 & & & \\
		1 & 0 & 1 & & 2\\
	\end{array}\subseteq D_4+A_1,$$ which has equivariant fundamental group $\bZ/2\bZ$ by Table~\ref{ta: E7Levis}. By Theorem~\ref{th: isogindep} and Proposition~\ref{p: birigid orbs}, this nilpotent orbit is birationally induced; its universal cover must therefore be birationally rigid. We hence have $$\O=\BInd_{D_4}^{E_7}\left( \begin{array}{c c c}
		& 1 & \\
		1 & 0 & 1\\
	\end{array} \right)$$ and $$\hatO=\BInd_{D_4+ A_1}^{E_7}\left(\mbox{univ. cover of }\begin{array}{c c c c c}
		& 1 & & &\\
		1 & 0 & 1 & & 2\\
	\end{array}\right).$$
	\fi 
	
	\subsubsection{\texorpdfstring{$E_7(a_2)\subseteq E_7$}{E7a2 in E7}}\label{sss: E7a2 in E7}
	
	The nilpotent orbit in $E_7$ with Bala-Carter label $E_7(a_2)$ has weighted Dynkin diagram $$\begin{array}{c c c c c c}
		& & 2 & & & \\
		2 & 2 & 0 & 2 & 0 & 2 \\
	\end{array}$$ and unique rigid induction datum
	$$\begin{array}{c c c}
		0 & & 0 \\
	\end{array}\subseteq 2 A_1.$$
	By Remark~\ref{rmk: brid covs} and Table~\ref{ta: E7Levis}, the universal cover of $$\begin{array}{c c c c c c}
		2 & & & \\
		 & 0 & 2 & 0 \\
	\end{array}\subseteq (A_3+A_1)''$$ is a birationally rigid induction datum for $\O$. We must therefore have
	$$\O = \BInd_{2A_1}^{E_7}\left(\begin{array}{c c c}
		0 & & 0 \\
	\end{array} \right)$$
	and
	$$\hatO = \BInd_{(A_3+A_1)''}^{E_7}\left( \mbox{univ. cover of }\begin{array}{c c c c c c}
		 2 & & & \\
		 & 0 & 2 & 0 \\
	\end{array} \right).$$
	
	\iffalse
	
	Since $\O$ is even, we know that $$\O=\BInd_{2A_1}^{E_7}\left(\begin{array}{c c c}
		0 & & 0\\
	\end{array}\right).$$ To determine the birationally rigid induction datum for $\hatO$, we can find the induction datum for $\O$ with maximal semisimple corank such that the equivariant fundamental group is nontrivial. Since the nilpotent orbits in $(3A_1)''$ and $A_3$ induced from the zero orbit in $2A_1$ correspond to the partitions $(2)\times (1^2)\times (1^2)$ and $(2^2)$ respectively, these orbits have trivial equivariant fundamental group by Proposition~\ref{prop: E7 Levi Gps} and Table~\ref{ta: ADLevis}. On the other hand, the nilpotent orbit in $(A_3+A_1)''$ so induced has equivariant fundamental group $\bZ/2\bZ$ by Table~\ref{ta: E7Levis}. From this, we get $$\hatO=\BInd_{(A_3+A_1)''}^{E_7}\left( \mbox{univ. cover of } \begin{array}{c c c c}
		2 & & & \\
		& 0 & 2 & 0 \\
	\end{array}\right).$$ 
	\fi

	\subsubsection{\texorpdfstring{$E_7(a_1)\subseteq E_7$}{E7a1 in E7}}\label{sss: E7a1 in E7}
	
	The nilpotent orbit in $E_7$ with Bala-Carter label $E_7(a_1)$ has weighted Dynkin diagram $$\begin{array}{c c c c c c}
		& & 2 & & & \\
		2 & 2 & 0 & 2 & 2 & 2 \\
	\end{array}$$ and unique rigid induction datum
	$$\begin{array}{c}
		0 \\
	\end{array}\subseteq A_1.$$ By Remark~\ref{rmk: brid covs} and Table~\ref{ta: E7Levis}, the universal cover of $$\begin{array}{c c c c c c}
	& 2 & & & \\
	0  & & 2 & & 2 \\
	\end{array}\subseteq 4A_1$$ is a birationally rigid induction datum for $\O$. We must therefore have
	$$\O = \BInd_{A_1}^{E_7}\left(\begin{array}{c}
		0 \\
	\end{array}\right)$$
	and
	$$\hatO = \BInd_{4A_1}^{E_7}\left( \mbox{univ. cover of }\begin{array}{c c c c c c}
	& 2 & & & \\
	 0 &  & 2 & & 2 \\
	\end{array} \right).$$
	\iffalse	
	
	Since $\O$ is even, we know that $$\O=\BInd_{A_1}^{E_7}\left(\begin{array}{c}
		0 \\
	\end{array}\right).$$ To determine the birationally rigid induction datum for $\hatO$, we can find the induction datum for $\O$ with maximal semisimple corank such that the equivariant fundamental group is nontrivial. All nilpotent orbits in $2A_1$ and $(3A_1)'$ and all except the regular nilpotent orbit in $(3A_1)''$ have trivial equivariant fundamental group by Proposition~\ref{prop: E7 Levi Gps} and Table~\ref{ta: ADLevis}. Since the zero orbit in $A_1$ does not induce to the regular nilpotent orbit in $(3A_1)''$, by Table~\ref{ta: E7Levis} we must get $$\hatO=\BInd_{4 A_1}^{E_7}\left( \mbox{univ. cover of } \begin{array}{c c c c c}
		&  2 & & & \\
		0 & & 2 & & 2 \\
	\end{array}\right).$$
	\fi
	
	\subsubsection{\texorpdfstring{$E_7\subseteq E_7$}{E7 in E7}}\label{sss: E7 in E7}
	
	The nilpotent orbit in $E_7$ with Bala-Carter label $E_7$ has weighted Dynkin diagram $$\begin{array}{c c c c c c}
		& & 2 & & & \\
		2 & 2 & 2 & 2 & 2 & 2 \\
	\end{array}$$ and unique rigid induction datum consisting of the zero orbit in the Lie algebra of the torus $T$. By Remark~\ref{rmk: brid covs} and Table~\ref{ta: E7Levis}, the universal cover of $$\begin{array}{c c c c}
		2 & & & \\
		  & 2 & & 2 \\
	\end{array}\subseteq (3A_1)''$$ is a birationally rigid induction datum for $\O$. We must therefore have
	$$\O = \BInd_{T}^{E_7}\left( \{0\}\right).$$
	$$\hatO = \BInd_{(3A_1)''}^{E_7}\left( \mbox{univ. cover of }\begin{array}{c c c c}
		2 &  & & \\
		 & 2 & & 2 \\
	\end{array} \right).$$
	\iffalse
	
	 Since $\O$ is even, we know that $$\O=\BInd_{T}^{E_7}(\{0\}).$$ Furthermore, as any nilpotent orbit induced from the zero orbit for the torus $T$ is regular, to determine the birationally rigid induction datum for $\hatO$ it suffices to find the Levi subgroup which has minimal rank with respect to the regular nilpotent orbit having non-trivial equivariant fundamental group. From Proposition~\ref{prop: E7 Levi Gps} and Table~\ref{ta: E7Levis}, we get $$\hatO=\BInd_{(3 A_1)''}^{E_7}\left( \mbox{univ. cover of } \begin{array}{c c c c}
		2 & & & \\
		& 2 & & 2 \\
	\end{array}\right).$$
	\fi
	
	\subsection{\texorpdfstring{$\pi_1^G(\O)=\Sym_2$}{pi=S2}}\label{ss: pi=S2}
	
	There are 38 induced nilpotent orbits in exceptional cases with $\pi_1^G(\O)=\Sym_2$; that is, there are no such nilpotent orbits in type $G_2$, 5 such in type $F_4$, 1 such in type $E_6$, 8 such in type $E_7$ and 24 such in type $E_8$. As in Subsection~\ref{ss: pi=Z2}, the only non-trivial nilpotent cover of such $\O$ is the universal cover $\hatO$ and if the universal cover is birationally rigid then $\O$ is birationally induced from its unique rigid induction datum. Using Proposition~\ref{p: bsr}, this handles 3 orbits for $F_4$,  the orbit for $E_6$, 1 orbit for $E_7$, and 6 orbits for $E_8$.
	
	By a similar argument, if $\O$ is birationally rigid then $\O$ must have a unique rigid induction datum and the universal cover of $\O$ must be birationally induced from such rigid induction datum. This resolves another 2 orbits in type $E_7$ and 2 orbits in type $E_8$. We thus have 2 remaining induced nilpotent orbits in type $F_4$, 5 in type $E_7$ and 16 in type $E_8$.
	
	If $\O$ is even then we know by Proposition~\ref{p: even} that $\O$ is birationally induced from the zero orbit in the Jacobson-Morozov Levi. Thus, if $\O$ is even and has exactly two rigid induction data, then we know that $\O$ is birationally induced from the zero orbit in the Jacobson-Morozov Levi and the universal cover of $\O$ must be birationally induced from the other rigid induction datum. This resolves the remaining 2 orbits in type $F_4$, 3 more orbits in type $E_7$ and 8 more orbits in type $E_8$. 
	
	We deal with the remaining 10 nilpotent orbits individually.
	
	\subsubsection{\texorpdfstring{$A_3+A_2\subseteq E_7$}{A3A2 in E7}}\label{sss: A3A2 in E7}
	
	The nilpotent orbit $\O$ in $E_7$ with Bala-Carter label $A_3+A_2$ has weighted Dynkin diagram
	$$\begin{array}{c c c c c c}
		& & 0 & & & \\
		0 & 0 & 1 & 0 & 1 & 0 \\
	\end{array}$$ and has two rigid induction data $$\begin{array}{c c c c c c}
		& & 0 & & & \\
		0 & 1 & 0 & 0 & & 0  \\
	\end{array}\subseteq D_5+A_1,\quad\mbox{and}\quad \begin{array}{c c c c c}
		& 0 &  & &\\
		0 & 0 & 1 & 0 & 1\\
	\end{array}\subseteq D_6.$$
	By the proof of \cite[Proposition 3.1]{Fu},\footnote{The result in \cite{Fu} is for $E_7$ of adjoint type, but the argument works equally well when $G$ is simply connected.} we have $$\O= \BInd_{D_5+A_1}^{E_7}\left(\begin{array}{c c c c c c}
		& & 0 & & & \\
		0 & 1 & 0 & 0 & & 0  \\
	\end{array}\right)$$ and therefore 
	$$\hatO= \BInd_{D_6}^{E_7}\left(\begin{array}{c c c c c}
		& 0 & & &\\
		0 & 0 &  1 & 0 & 1\\
	\end{array}\right).$$

	\subsubsection{\texorpdfstring{$D_5(a_1)\subseteq E_7$}{D5a1 in E7}}\label{sss: D5a1 in E7}
	
	The nilpotent orbit in $E_7$ with Bala-Carter label $D_5(a_1)$ has weighted Dynkin diagram
	$$\begin{array}{c c c c c c}
		& & 0 & & & \\
		2 & 0 & 1 & 0 & 1 & 0 \\
	\end{array}$$ and has unique rigid induction datum $$\begin{array}{c c c c c c}
		& & 0 \\
		0 & 0 & 0  \\
	\end{array}\subseteq A_4.$$ Induction from this rigid induction datum to $\O$ passes through the nilpotent orbit $$\begin{array}{c c c c c}
		& 0 & & &  \\
		0 & 1 & 0 & 1 & 0 \\
	\end{array}\subseteq D_6,$$ which has equivariant fundamental group $\Sym_2$ by Proposition~\ref{prop: adj fun gp} and Table~\ref{ta: ADLevis}. By Proposition~\ref{p: birigid orbs}, the nilpotent orbit in $\so(12)$ with this weighted Dynkin diagram is birationally rigid; by Theorem~\ref{th: isogindep}, this nilpotent orbit is also birationally rigid in the standard Levi subalgebra of $E_7$ of type $D_6$. Birational induction therefore sends the unique rigid induction datum to the universal cover of this induction datum, and thus we get $$\O=\BInd_{D_6}^{E_7}\left(\begin{array}{c c c c c}
		& 0 & & &  \\
		0 & 1 & 0 & 1 & 0 \\
	\end{array}\right)$$ and $$\hatO=\BInd_{A_4}^{E_7}\left(\begin{array}{c c c c c c}
		& & 0 \\
		0 & 0 & 0  \\
	\end{array}\right).$$
	
	\subsubsection{\texorpdfstring{$A_3+A_2\subseteq E_8$}{A3A2 in E8}}\label{sss: A3A2 in E8}
	
	The nilpotent orbit in $E_8$ with Bala-Carter label $A_3+A_2$ has weighted Dynkin diagram
	$$\begin{array}{c c c c c c c}
		& & 0 & & & & \\
		1 & 0 & 0 & 0 & 1 & 0 & 0\\
	\end{array}$$ and has rigid induction data $$\begin{array}{c c c c c c}
		& & 0 & & & \\
		0 & 1 & 0 & 0 & 0 & 0  \\
	\end{array}\subseteq E_7,\quad\mbox{and}\quad \begin{array}{c c c c c c}
		& 0 & &  & &\\
		0 & 0 & 0 & 0 & 1 & 0\\
	\end{array}\subseteq D_7.$$
	By \cite[Proposition 3.1]{Fu}, we have $$\O= \BInd_{D_7}^{E_8}\left(\begin{array}{c c c c c c}
		& 0 & &  & &\\
		0 & 0 & 0 & 0 & 1 & 0\\
	\end{array}\right)$$ and therefore 
	$$\hatO= \BInd_{E_7}^{E_8}\left(\begin{array}{c c c c c c}
		& & 0 & & &\\
		0 & 1 & 0 & 0 & 0 & 0\\
	\end{array}\right).$$

	\subsubsection{\texorpdfstring{$A_4\subseteq E_8$}{A4 in E8}}\label{sss: A4 in E8}
	
	The nilpotent orbit in $E_8$ with Bala-Carter label $A_4$ has weighted Dynkin diagram
	$$\begin{array}{c c c c c c c}
		& & 0 & & & & \\
		2 & 0 & 0 & 0 & 0 & 0 & 2 \\
	\end{array}$$ and has unique rigid induction datum $$\begin{array}{c c c c c}
		& 0 & & & \\
		0 & 0 & 0 & 0 & 0 \\
	\end{array}\subseteq D_6.$$
	Since $\O$ is even, we know that $$\O=\BInd_{D_6}^{E_8}\left(\begin{array}{c c c c c}
		& 0 & & & \\
		0 & 0 & 0 & 0 & 0 \\
	\end{array}\right).$$ To determine the birationally rigid induction datum for the universal cover of $\O$, we just need to find the induction datum for $\O$ with maximal semisimple corank such that the equivariant fundamental group is non-trivial. By Proposition~\ref{prop: adj fun gp}, we immediately see that $$\hatO=\BInd_{E_7}^{E_8}\left(\mbox{univ. cover of }\begin{array}{c c c c c c}
		& & 0 & & & \\
		2 & 0 & 0 & 0 & 0 & 0 \\
	\end{array}\right).$$
	
	\subsubsection{\texorpdfstring{$D_5(a_1)\subseteq E_8$}{D5a1 in E8}}\label{sss: D5a1 in E8}
	
	The nilpotent orbit in $E_8$ with Bala-Carter label $D_5(a_1)$ has weighted Dynkin diagram
	$$\begin{array}{c c c c c c c}
		& & 0 & & & & \\
		1 & 0 & 0 & 0 & 1 & 0 & 2 \\
	\end{array}$$ and has unique rigid induction datum $$\begin{array}{c c c c c}
		&  & 1 & & \\
		0 & 0 & 0 & 0 & 0 \\
	\end{array}\subseteq E_6.$$ Induction from this rigid induction datum to $\O$ passes through the nilpotent orbit $$\begin{array}{c c c c c c c}
		& & 0 & & & \\
		1 & 0 & 0 & 0 & 1 & 0 \\
	\end{array}\subseteq E_7,$$ which has equivariant fundamental group $\Sym_2$ by Proposition~\ref{prop: adj fun gp} and \cite[\S 8.4]{CM}. By Proposition~\ref{p: birigid orbs}, the nilpotent orbit in the Lie algebra of the semisimple simply connected algebraic group of type $E_7$ with the same weighted Dynkin diagram is birationally rigid; by Theorem~\ref{th: isogindep}, this is also true in this Levi subalgebra. Birational induction therefore sends the unique rigid induction datum to the universal cover of this induction datum, and thus we get $$\O=\BInd_{E_7}^{E_8}\left(\begin{array}{c c c c c c c}
		& & 0 & & & \\
		1 & 0 & 0 & 0 & 1 & 0 \\
	\end{array}\right)$$ and $$\hatO=\BInd_{E_6}^{E_7}\left(\begin{array}{c c c c c}
		&  & 1 & & \\
		0 & 0 & 0 & 0 & 0 \\
	\end{array}\right).$$

	\subsubsection{\texorpdfstring{$D_6(a_1)\subseteq E_8$}{D6a1 in E8}}\label{sss: D6a1 in E8}
	The nilpotent orbit in $E_8$ with Bala-Carter label $D_6(a_1)$ has weighted Dynkin diagram
	$$\begin{array}{c c c c c c c}
		& & 1 & & & & \\
		0 & 1 & 0 & 0 & 0 & 1 & 2 \\
	\end{array}$$ and has unique rigid induction datum $$\begin{array}{c c c c c c}
		0 & 0 & 0 & 0 & 0 \\
	\end{array}\subseteq A_5.$$ We note that induction from this rigid induction datum to $\O$ passes through the nilpotent orbit $$
	\begin{array}{c c c c c}
		& & 2 & & \\
		0 & 0 & 0 & 0 & 0 \\
	\end{array}\subseteq E_6
	$$
	which has equivariant fundamental group $\Sym_2$ by Proposition~\ref{prop: adj fun gp} and \cite[\S 8.4]{CM}. By Proposition~\ref{p: even}, this induction datum is birationally induced from the zero orbit in $A_5$. We thus have
	$$\O=\BInd_{A_5}^{E_8}\left(\begin{array}{c c c c c c}
		0 & 0 & 0 & 0 & 0 \\
	\end{array}\right)$$ and $$\hatO=\BInd_{E_6}^{E_8}\left(\mbox{univ. cover of } \begin{array}{c c c c c}
		& & 2 & & \\
		0 & 0 & 0 & 0 & 0 \\
	\end{array}\right).$$
	Since the zero orbit in $A_5$ is the only proper induction datum for the underlying nilpotent orbit in $E_6$ (and clearly has trivial equivariant fundamental group), we immediately see that this birational induction datum for $\hatO$ is birationally rigid.

	\subsubsection{\texorpdfstring{$E_7(a_4)\subseteq E_8$}{E7a4 in E8}}\label{sss: E7a4 in E8}
	
	The nilpotent orbit in $E_8$ with Bala-Carter label $E_7(a_4)$ has weighted Dynkin diagram
	$$\begin{array}{c c c c c c c}
		& & 0 & & & & \\
		0 & 0 & 1 & 0 & 1 & 0 & 2\\
	\end{array}$$ and has two rigid induction data $$\begin{array}{c c c c c c}
		& & 0 & & & \\
		0 & 1 & 0 & 0 & & 0  \\
	\end{array}\subseteq D_5+ A_1,\quad\mbox{and}\quad \begin{array}{c c c c c}
		& 0 & &  &\\
		0 & 0 & 1 & 0 & 1\\
	\end{array}\subseteq D_6.$$
	By the proof of \cite[Proposition 3.1]{Fu}, we have $$\O=\BInd_{E_7}^{E_8}\left(\begin{array}{c c c c c c c}
		& & 0 & & & \\
		0 & 0 & 1 & 0 & 1 & 0 \\
	\end{array}\right).$$ We may therefore argue as in Subsection~\ref{sss: A3A2 in E7} to get 
	$$\O= \BInd_{D_5+A_1}^{E_8}\left(\begin{array}{c c c c c c}
		& & 0 & & & \\
		0 & 1 & 0 & 0 & & 0  \\
	\end{array}\right)$$ and therefore 
	$$\hatO= \BInd_{D_6}^{E_8}\left(\begin{array}{c c c c c}
		& 0 &  & &\\
		0 & 0 & 1 & 0 & 1\\
	\end{array}\right).$$

	\subsubsection{\texorpdfstring{$D_7(a_2)\subseteq E_8$}{D7a2 in E8}}\label{sss: D7a2 in E8}
	
	The nilpotent orbit in $E_8$ with Bala-Carter label $D_7(a_2)$ has weighted Dynkin diagram
	$$\begin{array}{c c c c c c c}
		& & 0 & & & & \\
		1 & 0 & 1 & 0 & 1 & 0 & 1\\
	\end{array}$$ and has two rigid induction data $$\O_1\coloneqq\begin{array}{c c c c c c c}
		& & 0 & & & & \\
		0 & 0 & 0 & & 0& & 0  \\
	\end{array}\subseteq A_4 + 2 A_1=:L_1$$ and $$\O_2\coloneqq\begin{array}{c c c c c c c}
		0 & 0 & 0 & & 0 & 0 & 0\\
	\end{array}\subseteq 2 A_3=: L_2.$$ By \cite{How}, we get that $$N_G(L_1)/L_1\simeq \Sym_2 \times \Sym_2 \quad\mbox{and}\quad N_G(L_2)/L_2\simeq \Dih_8.$$
	
	Note that by \cite{FJLS} $\overline{\O}$ contains a single codimension 2 nilpotent orbit $\O'$, which is the nilpotent orbit with Bala-Carter label $D_5+A_1$. Let $\cL'$ be the corresponding symplectic leaf of codimension 2 in $\Spec(\C[\O])$, as in \cite[Lemma 3.2.2]{MM}. This symplectic leaf has singularity of type $C_2$ (i.e. of type $D_3$ with the unique non-trivial graph automorphism of the Dynkin diagram of $D_3$ acting non-trivially).
	
	Let $\g'$ be a simple complex Lie algebra of type $D_3$, let $\h'$ be a Cartan subalgebra thereof, and let $W'$ be the corresponding Weyl group. Then we have $W'\cong \Sym_4$ and $\pi_1(\cL')=\pi_1(\O')=\Sym_2$ acts on $W'$ such that the non-trival element acts as conjugation by $(1,4)(2,3)$. By definition, we then get $$W(\O)=(W')^{\pi_1(\cL')}=\langle(1,4),(1,2,4,3)\rangle\simeq \Dih_8.$$
	
	Let $(L,\O_L)$ be the birationally rigid induction datum for $\O$; by Proposition~\ref{p: NamW}, we get that $W(\O)=N_G(L,\O_L)/L$. Since both $\O_1$ and $\O_2$ are the zero orbit for their respective Levi subgroups, we in fact have $W(\O)=N_G(L)/L$. By above, this therefore means that $L=L_2$ and $\O_L=\O_2$. This shows that $$\O=\BInd_{2 A_3}^{E_8}(\begin{array}{c c c c c c c}
		0 & 0 & 0 & & 0 & 0 & 0\\
	\end{array})$$ and hence that $$\hatO=\BInd_{A_4 + 2 A_1}^{E_8}\left(\begin{array}{c c c c c c c}
		& & 0 & & & & \\
		0 & 0 & 0 & & 0& & 0  \\
	\end{array}\right).$$

	\subsubsection{\texorpdfstring{$E_6(a_1)+A_1\subseteq E_8$}{E6a1A1 in E8}}\label{sss: E6a1A1 in E8}
	
	The nilpotent orbit in $E_8$ with Bala-Carter label $E_6(a_1)+A_1$ has weighted Dynkin diagram
	$$\begin{array}{c c c c c c c}
		& & 0 & & & & \\
		1 & 0 & 1 & 0 & 1 & 0 & 2 \\
	\end{array}$$ and has unique rigid induction datum $$\begin{array}{c c c c c}
		&  & 0 & &\\
		0 & 0 & 0 & & 0  \\
	\end{array}\subseteq A_4 + A_1.$$  Induction from this rigid induction datum to $\O$ passes through the nilpotent orbit $$\O_1:=\begin{array}{c c c c c c c}
		& & 0 & & & \\
		1 & 0 & 1 & 0 & 1 & 0 \\
	\end{array}\subseteq E_7$$ which has equivariant fundamental group $\Sym_2$ by Proposition~\ref{prop: adj fun gp} and \cite[\S 8.4]{CM}. By Proposition~\ref{p: birigid orbs}, the nilpotent orbit in the Lie algebra of the semisimple simply connected algebraic group of type $E_7$ with the same weighted Dynkin diagram is birationally rigid; by Theorem~\ref{th: isogindep}, this is also true in this Levi subalgebra. Birational induction therefore sends the rigid induction datum to the universal cover of this induction datum, and thus we get $$\O=\BInd_{E_7}^{E_8}\left(\begin{array}{c c c c c c c}
		& & 0 & & & \\
		1 & 0 & 1 & 0 & 1 & 0 \\
	\end{array}\right)$$ and $$\hatO=\BInd_{A_4 + A_1}^{E_7}\left(\begin{array}{c c c c c}
		&  & 0 & &\\
		0 & 0 & 0 & & 0  \\
	\end{array}\right).$$
	
	\subsubsection{\texorpdfstring{$E_7(a_3)\subseteq E_8$}{E7a3 in E8}}\label{sss: E7a3 in E8}
	
	The nilpotent orbit in $E_8$ with Bala-Carter label $E_7(a_3)$ has weighted Dynkin diagram
	$$\begin{array}{c c c c c c c}
		& & 0 & & & & \\
		2 & 0 & 1 & 0 & 1 & 0 & 2 \\
	\end{array}$$ and has unique rigid induction datum $$\begin{array}{c c c}
		&  & 0\\
		0 & 0 & 0 \\
	\end{array}\subseteq A_4.$$  Induction from this rigid induction datum to $\O$ passes through the nilpotent orbit $$\O_1:=\begin{array}{c c c c c c}
		& 0 & & & \\
		0 & 1 & 0 & 1 & 0  \\
	\end{array}\subseteq D_6$$ which has equivariant fundamental group $\Sym_2$ by Proposition~\ref{prop: adj fun gp} and Table~\ref{ta: ADLevis}. By Proposition~\ref{p: birigid orbs}, the nilpotent orbit in $\so(12)$ with this weighted Dynkin diagram is birationally rigid; by Theorem~\ref{th: isogindep}, this nilpotent orbit is also birationally rigid in the standard Levi subalgebra of $E_7$ of type $D_6$. Birational induction therefore maps the rigid induction datum to the universal cover of this induction datum, and thus we get $$\O=\BInd_{D_6}^{E_8}\left(\begin{array}{c c c c c c}
		& 0 & & & \\
		0 & 1 & 0 & 1 & 0  \\
	\end{array}\right)$$ and $$\hatO=\BInd_{A_4}^{E_7}\left(\begin{array}{c c c}
		&  & 0\\
		0 & 0 & 0 \\
	\end{array}\right).$$
	
	\subsection{\texorpdfstring{$\pi_1^G(\O)=\bZ/3\bZ$}{pi=Z3}}\label{ss: pi=Z3}
	
	This case can only arise in type $E_6$; in this case, there are 4 nilpotent orbits with $\pi_1^G(\O)=\bZ/3\bZ$. As in Subsections~\ref{ss: pi=Z2} and \ref{ss: pi=S2}, the only non-trivial nilpotent cover of such $\O$ is the universal cover and if the universal cover is birationally rigid then $\O$ is birationally induced from its unique rigid induction datum. Using Proposition~\ref{p: bsr}, this handles 2 of the orbits. 
	
	The remaining two orbits have Bala-Carter labels $E_6$ and $E_6(a_1)$. These are both even and thus the orbits are birationally induced from the zero orbit in the appropriate Jacobson-Morozov Levi subalgebras. Furthermore, Remark~\ref{rmk: brid covs} yields another birationally rigid induction datum in each case. We therefore immediately deduce that for the nilpotent orbit with Bala-Carter label $E_6$ we have $$\O=\BInd_{T}^{E_6}(\{0\})$$ and $$\hatO=\BInd_{2 A_2}^{E_6}\left(\mbox{univ. cover of } \begin{array}{c c c c c }
		2 & 2 &  & 2 & 2 \\
	\end{array}\right),$$ while for the nilpotent orbit with Bala-Carter label $E_6(a_1)$ we have $$\O=\BInd_{A_1}^{E_6}\left(\begin{array}{c}
		0\end{array}\right)$$ and $$\hatO=\BInd_{2 A_2 + A_1}^{E_6}\left(\mbox{univ. cover of } \begin{array}{c c c c c }
		& & 0 & & \\
		2 & 2 &  & 2 & 2 \\
	\end{array}\right).$$

	\iffalse
	To determine the birationally rigid induction data for the universal covers, it is sufficient to determine which induction data $(M,\O_M)$ for $\O$ have non-trivial $\pi_1^M(\O_M)$. The unique such $(M,\O_M)$ with the largest semisimple corank then must have the property that the universal cover of $\O_M$ is birationally rigid and birationally induces to the universal cover of $\O$. Using Proposition~\ref{prop: E6 Levi Gps} and Table~\ref{ta: E6Levis}, we easily compute that for the nilpotent orbit $E_6$ we have $$\O=\BInd_{T}^{E_6}(\{0\})$$ and $$\hatO=\BInd_{2 A_2}^{E_6}\left(\mbox{univ. cover of } \begin{array}{c c c c c }
		2 & 2 &  & 2 & 2 \\
	\end{array}\right),$$ while for the nilpotent orbit $E_6(a_1)$ we have $$\O=\BInd_{A_1}^{E_6}\left(\begin{array}{c}
		0\end{array}\right)$$ and $$\hatO=\BInd_{2 A_2 + A_1}^{E_6}\left(\mbox{univ. cover of } \begin{array}{c c c c c }
		& & 0 & & \\
		2 & 2 &  & 2 & 2 \\
	\end{array}\right).$$
	\fi

	\subsection{\texorpdfstring{$\pi_1^G(\O)=\Sym_2\times\bZ/2\bZ$}{pi=S2Z2}}\label{ss: pi=S2Z2}
	
	This case can only arise in type $E_7$; in this case, there are 3 nilpotent orbits with $\pi_1^G(\O)=\Sym_2\times \bZ/2\bZ$. We tackle these case-by-case.
	
	Before doing so, however, let us establish some notation and conventions. We denote the generator of $\Sym_2$ by $a$ and the generator of $\bZ/2\bZ$ by $b$. The subgroups of $\pi_1^G(\O)$ are then $1$, $\Sym_2=\langle a\rangle$, $\tw(\Sym_2)=\langle ab\rangle$, $\bZ/2\bZ=\langle b\rangle$ and $\Sym_2\times \bZ/2\bZ$. Let us denote the corresponding nilpotent orbit covers by $\hatO$, $\covO_a$, $\covO_{ab}$, $\covO_b$ and $\O$. 
	
	Since in this paper we only describe the $G$-equivariant fundamental group $\pi_1^G(\O)$ up to isomorphism, we cannot canonically distinguish between subgroups which are interchanged by an automorphism of $\pi_1^G(\O)$. Since there are automorphisms of $\Sym_2\times\bZ/2\bZ$ which permute the subgroups $\Sym_2$, $\tw(\Sym_2)$, and $\bZ/2\bZ$, we cannot distinguish amongst these subgroups. When (the image under birational induction of) a birationally rigid induction datum corresponds to such a subgroup, we thus assign it to one of $\Sym_2$, $\tw(\Sym_2)$ and $\bZ/2\bZ$ in a largely arbitrary way. When the birationally rigid induction datum would also exist for $E_7$ of adjoint type, we prefer to say that the image under birational induction corresponds to $\bZ/2\bZ$; when it doesn't, we prefer $\Sym_2$ or $\tw(\Sym_2)$. We maintain our conventions from this section in Table~\ref{ta: E7}.
	
	\subsubsection{\texorpdfstring{$D_4(a_1)+A_1\subseteq E_7$}{D4a1A1 in E7}}\label{sss: D4a1A1 in E7}
	
	The nilpotent orbit in $E_7$ with Bala-Carter label $D_4(a_1)+A_1$ has weighted Dynkin diagram
	$$\begin{array}{c c c c c c}
		& & 1 & & & \\
		0 & 1 & 0 & 0 & 0 & 1 \\
	\end{array}.$$ This has unique rigid induction datum $$\begin{array}{c c c c c}
		0 & 0 & 0 & 0 & 0  \\
	\end{array}\subseteq (A_5)',$$ which clearly has trivial equivariant fundamental group. The only induction data for $\O$ with semisimple corank 1 are therefore
	$$\begin{array}{c c c c c c}
		1 & 0 & 0 & 0 & 0 & 1 \\
	\end{array}\subseteq A_6, \quad \begin{array}{c c c c c}
	& 2 &  & & \\
	0 & 0 & 0 & 0 & 0  \\
	\end{array} \subseteq D_6\quad \mbox{and}\quad \begin{array}{c c c c c}
		& & 2 & & \\
		0 & 0 & 0 & 0 & 0  \\
	\end{array} \subseteq E_6;$$ these have equivariant fundamental groups $1$, $\bZ/2\bZ$ and $\Sym_2$, respectively, by Proposition~\ref{prop: E7 Levi Gps}, Table~\ref{ta: E7Levis}, and \cite[\S 8.4]{CM}. Note that all these nilpotent orbits are birationally induced by Theorem~\ref{th: isogindep} and Proposition~\ref{p: birigid orbs}, which means in particular that the universal cover of the latter two orbits is birationally rigid.
	
	By Proposition~\ref{p: bsr} and the discussion following it, the universal cover $\hatO$ of $\O$ is birationally rigid. Therefore, $$\O=\BInd_{(A_5)'}^{E_7}\left(\begin{array}{c c c c c}
		0 & 0 & 0 & 0 & 0  \\
	\end{array}\right)$$ while the universal covers of $\begin{array}{c c c c c}
	& 2 &  & & \\
	0 & 0 & 0 & 0 & 0  \\
	\end{array}$ and $\begin{array}{c c c c c}
		& & 2 & & \\
		0 & 0 & 0 & 0 & 0  \\
	\end{array}$ birationally induce to distinct 2-fold nilpotent covers of $\O$.
	
	We therefore conclude that two of the order 2 subgroups correspond to these birationally induced 2-fold nilpotent covers and the remaining one corresponds to a birationally rigid nilpotent cover. Following the discussion at the beginning of this subsection, we make the convention that 
	$$\covO_a=\BInd_{D_6}^{E_7}\left(\mbox{univ. cover of }\begin{array}{c c c c c}
		& 2 & & & \\
		0 & 0 & 0 & 0 & 0  \\
	\end{array} \right),$$ 
	$$\covO_b=\BInd_{E_6}^{E_7}\left(\mbox{univ. cover of }\begin{array}{c c c c c}
		& & 2 & & \\
		0 & 0 & 0 & 0 & 0  \\
	\end{array} \right),$$ 
	and that $\covO_{ab}$ is birationally rigid.
	
	\subsubsection{\texorpdfstring{$E_7(a_4)\subseteq E_7$}{E7a4 in E7}}\label{sss: E7a4 in E7}
	
	The nilpotent orbit in $E_7$ with Bala-Carter label $E_7(a_4)$ has weighted Dynkin diagram
	$$\begin{array}{c c c c c c}
		& & 0 & & & \\
		2 & 0 & 2 & 0 & 0 & 2 \\
	\end{array}.$$
	This nilpotent orbit has two rigid induction data:
	$$\begin{array}{c c c c c c}
		& 1 & & & \\
		1 & 0 & 1 &  & 0 \\
	\end{array}\subseteq D_4 + A_1\quad \mbox{and}\quad \begin{array}{c c c c c}
		& 0&  & \\
		0 & & 0 & 0   \\
	\end{array} \subseteq A_2 + 2 A_1.$$ Furthermore, by Remark~\ref{rmk: brid covs} and Table~\ref{ta: E7Levis}, the universal covers of the induction data $$\begin{array}{c c c c c c c}
		& & 2 & & &\\
		0 & 0 &  & 0 & 2 & 0   \\
	\end{array} \subseteq A_3+A_2 + A_1 \quad \mbox{and}\quad \begin{array}{c c c c c c c}
		& & 0 & & &\\
		0 & 0 & 2 & 0 & & 2   \\
	\end{array} \subseteq D_5 + A_1$$  are birationally rigid; there are hence 4 birationally rigid induction data for nilpotent covers of $\O$. By Proposition~\ref{p: bsr} and the discussion following it, the universal cover of $\O$ is birationally rigid. Therefore, all the remaining covers are birationally induced. Since $\O$ is even we have $$\O=\BInd_{A_2 + 2 A_1}^{E_7}\left(\begin{array}{c c c c c}
	& 0&  & \\
	0 & & 0 & 0   \\
	\end{array}\right),$$ and the remaining covers induce to the 3 distinct 2-fold nilpotent covers of $\O$. 	Following the conventions discussed at the beginning of this section, we say that $$\pi_1^G\left(\BInd_{D_4 + A_1}^{E_7}\left(\begin{array}{c c c c c c}
	& 1 & & & \\
	1 & 0 & 1 &  & 0 \\
	\end{array}\right)\right)=\langle b\rangle=\bZ/2\bZ,$$
	$$
	\pi_1^G\left(\BInd_{D_5 + A_1}^{E_7}\left(\mbox{univ. cover of }\begin{array}{c c c c c c}
	& & 0 & & &\\
	0 & 0 & 2 & 0 & & 2   \\
	\end{array} \right)\right)=\langle a\rangle=\Sym_2,\quad \mbox{and}$$ 
	$$
	\pi_1^G\left(\BInd_{A_3+A_2+ A_1}^{E_7}\left(\mbox{univ. cover of }\begin{array}{c c c c c c c}
	& & 2 & & &\\
	0 & 0 &  & 0 & 2 & 0   \\
	\end{array} \right)\right)=\langle ab\rangle=\tw(\Sym_2).$$

	\subsubsection{\texorpdfstring{$E_7(a_3)\subseteq E_7$}{E7a3 in E7}}\label{sss: E7a3 in E7}
	
	The nilpotent orbit in $E_7$ with Bala-Carter label $E_7(a_3)$ has weighted Dynkin diagram
	$$\begin{array}{c c c c c c}
		& & 0 & & & \\
		2 & 0 & 2 & 0 & 2 & 2 \\
	\end{array}.$$ There are two rigid induction data for this nilpotent orbit:
	$$\begin{array}{c c c}
		& 0 &\\
		0 &  & 0   \\
	\end{array}\subseteq (3 A_1)'\quad\mbox{and}\quad \begin{array}{c c}
		0 & 0 \\
	\end{array}\subseteq A_2.$$ 
	Note that $\O$ is even, and thus $\O=\BInd_{(3 A_1)'}^{E_8}\left(\begin{array}{c c c}
		& 0 &\\
		0 &  & 0   \\
	\end{array}\right)$.
	
By Table~\ref{ta: E7Levis}, the induction datum $$\O_0:=\begin{array}{c c c c c c}
		& & 0 & & & \\
		2 & 0 & 2 & 0 &  & 2 \\
	\end{array}\subseteq D_5+A_1=: L$$ for $\O$ has equivariant fundamental group equal to $\Sym_2\times\bZ/2\bZ$; determining the birationally rigid induction data for the nilpotent covers of $\O$ therefore reduces to determining the same for the nilpotent covers of $\O_0$. 
	
	From Proposition~\ref{prop: E7 Levi Gps} and Tables~\ref{ta: ADLevis} and \ref{ta: E7Levis} we get that the induction data for $(L,\O_0)$ with semisimple corank 1 and non-trivial equivariant fundamental group are\footnote{For completeness, the remaining induction data for $\O_0$ of semisimple corank 1 are $(2,1)\times(1^2)\times (1^2)\times (2)\in A_2+3A_1$ and $(2,1^2)\times (2)\times (2)\in A_3+2A_1$.} 	$$\O_1\coloneqq \begin{array}{c c c c}
		& & 0 &  \\
		2 & 0 & 2 & 0  \\
	\end{array}\subseteq D_5,\quad\quad \O_2\coloneqq \begin{array}{c c c c c c}
		& & 2 & & & \\
		0 & 0 & & 2 & & 2\\
	\end{array}\subseteq A_2 + 3 A_1, $$
	
	$$\O_3\coloneqq \begin{array}{c c c c c c}
		& & 0 & & & \\
		0 & & 2 & 0 & & 2\\
	\end{array}\subseteq  A_3 + 2 A_1,\quad\mbox{and}\quad \O_4\coloneqq \begin{array}{c c c c c}
		& 0 & & &  \\
		0 & 2 & 0 & & 2  \\
	\end{array}\subseteq D_4 + A_1.$$
	The first and fourth of these have equivariant fundamental group $\Sym_2$, while the second and third have equivariant fundamental group $\bZ/2\bZ$. The induction data $\O_1$ and $\O_4$ are induced from both of the rigid induction data for $\O$; on the other hand, $\O_2$ is only induced from the zero orbit in $A_2$ and $\O_3$ is only induced from the zero orbit in $(3 A_1)'$. By Remark~\ref{rmk: brid covs} and Table~\ref{ta: E7Levis}, the universal covers of $\O_2$ and $\O_3$ are birationally rigid; we shall denote them by $\hatO_2$ and $\hatO_3$, respectively. We deduce that $\O_2$ is birationally induced from the zero orbit in $A_2$ and $\O_3$ is birationally induced from the zero orbit in $(3 A_1)'$.
	
	Since the zero orbit in $(3 A_1)'$ birationally induces to $\O$, the zero orbit in $A_2$ birationally induces to a non-trivial nilpotent cover of $\O$. Similarly $\hatO_2$ birationally induces to a nilpotent cover of $\BInd_{A_2 + 3 A_1}^{E_8}(\O_2)=\BInd_{A_2}^{E_8}(\begin{array}{c c}
		0 & 0 \\
	\end{array})$. Therefore, we must have $$\hatO=\BInd_{A_2 + 3 A_1}^{E_8}(\hatO_2).$$ The remaining birationally induced nilpotent covers of $\O$ therefore correspond to subgroups of $\pi_1^G(\O)$ of order 2.
	
	Following the discussion in the introduction of Section~\ref{ss: pi=S2Z2}, we make the conventions that $$\covO_b=\BInd_{A_2}^{E_7}(\begin{array}{c c}
		0 & 0 \\
	\end{array})$$ 
	and $$\covO_a=\BInd_{A_3 + 2 A_1}^{E_7}(\hatO_3).$$ The only remaining nilpotent cover of $\O$ is the nilpotent cover corresponding to $\tw(\Sym_2)$; from what we have already seen, this must be induced from the nilpotent cover of $\O_0$ corresponding to $\tw(\Sym_2)$,\footnote{We assume here that we have denoted the subgroups of $\pi_1^L(\O)$ and $\pi_1^G(\O)$ in such a way that birational induction corresponds to the trivial map on conjugacy classes of subgroups.} which itself must be birationally rigid. This resolves the final nilpotent cover of $\O$.

	\subsection{\texorpdfstring{$\pi_1^G(\O)=\Sym_2\times \bZ/3\bZ$}{pi=S2Z3}}\label{ss: pi=S2Z3}
	
	This case can only arise when $G$ is of type $E_6$, in which case $\O$ is the nilpotent orbit with Bala-Carter label $E_6(a_3)$ and weighted Dynkin diagram 
	$$\begin{array}{c c c c c }
		& & 0 & & \\
		2 & 0 & 2 & 0 & 2 \\
	\end{array}.$$ There are two rigid induction data for this nilpotent orbit:
	$$\begin{array}{c c c}
		& 0 &\\
		0 &  & 0   \\
	\end{array}\subseteq 3 A_1\quad\mbox{and}\quad \begin{array}{c c}
		0 & 0 \\
	\end{array}\subseteq A_2.$$ 
	Induction from each of these induction data passes through the nilpotent orbit  
	$$\begin{array}{c c c}
		& 0 &\\
		0 & 2 & 0   \\
	\end{array}\subseteq D_4,$$ which has equivariant fundamental group $\Sym_2$ by Proposition~\ref{prop: E6 Levi Gps} and Table~\ref{ta: E6Levis}. By Lemma~\ref{prop: deg}, the (2-fold) universal cover of this nilpotent orbit must therefore birationally induce to a $2n$-fold universal cover of $\O$ for some $n\in\N$.
	
	Note that there are four subgroups of $\Sym_2\times \bZ/3\bZ$: $1$, $\Sym_2$, $\bZ/3\bZ$ and $\Sym_2\times \bZ/3\bZ$. Since $\O$ is even, we must have $$\pi_1^G\left(\BInd_{3 A_1}^{E_6}\left(\begin{array}{c c c}
		& 0 &\\
		0 & & 0   \\
	\end{array}\right)\right)=\Sym_2\times \bZ/3\bZ,$$ and, by Proposition~\ref{p: bsr} and the discussion following it, the universal cover of $\O$ is birationally rigid. Putting this all together,  we must therefore have that $$\pi_1^G\left(\BInd_{D_4}^{E_6}\left(\mbox{univ. cover of }\begin{array}{c c c}
		& 0 &\\
		0 & 2 & 0   \\
	\end{array}\right)\right)=\bZ/3\bZ$$ and $$\mbox{univ. cover of } \begin{array}{c c c}
		& 0 &\\
		0 & 2 & 0   \\
	\end{array}=\BInd_{A_2}^{D_4}\left( \begin{array}{c c}
		0 & 0   \\
	\end{array}\right).$$ This thus implies that $$\pi_1^G\left(\BInd_{A_2}^{E_6}\left(\begin{array}{c c}
		0 & 0   \\
	\end{array}\right)\right)=\bZ/3\bZ.$$ 
	
	Finally, we note that the induction datum 
	$$\begin{array}{c c c c c}
		0 & 2 & 0 & 2 & 0
	\end{array}\subseteq A_5
	$$ has equivariant fundamental group $\bZ/3\bZ$ by Table~\ref{ta: E6Levis} and its universal cover is birationally rigid by Remark~\ref{rmk: brid covs}. Since the universal cover of $\O$ is birationally rigid, by process of elimination we must have 
	$$
	\pi_1^G\left(\BInd_{A_5}^{E_6}\left(\mbox{3-fold cover of }\begin{array}{c c c c c}
		0 & 2 & 0 & 2 & 0
	\end{array}\right)\right)=\Sym_2.
	$$
	
	\subsection{\texorpdfstring{$\pi_1^G(\O)=\Sym_3$}{pi=S3}}\label{ss: pi=S3}
	
	There are 8 induced nilpotent orbits in exceptional types with $\pi_1^G(\O)=\Sym_3$; more specifically, there is one in type $G_2$, none in type $F_4$, one in type $E_6$, one in type $E_7$ and 5 in type $E_8$. Note that $\Sym_3$ (which we identify as permutations of the set $\{1,2,3\}$) has four subgroups up to conjugacy: the trivial group 1, the subgroup $\Sym_2:=\langle(1,2)\rangle$ (which is conjugate to $\langle(1,3)\rangle$ and $\langle(2,3)\rangle$), the alternating group $\Alt_3$ and the whole group $\Sym_3$. These correspond to the four $G$-equivariant nilpotent covers of $\O$ (up to isomorphism), which we denote respectively by $\hatO, \breveO, \covO$ and $\O$. 
	
	For the nilpotent orbit with Bala-Carter label $G_2(a_1)$ in $G_2$, the birationally rigid induction data are given in \cite[Example 8.4.1]{LMM}. The remaining nilpotent orbits we tackle individually.
	
	\subsubsection{\texorpdfstring{$D_4(a_1)\subseteq E_6$}{D4a1 in E7}}\label{sss: D4a1 in E6}
	
	The nilpotent orbit in $E_6$ with Bala-Carter label $D_4(a_1)$ has weighted Dynkin diagram
	$$\begin{array}{c c c c c}
		& & 0 & & \\
		0 & 0 & 2 & 0 & 0 \\
	\end{array}.$$
	The rigid induction data for this nilpotent orbit are 
	$$\begin{array}{c c c c c}
		& & 0 & & \\
		0 & 0 &  & 0 & 0 \\
	\end{array}\subseteq 2A_2+A_1,\quad \begin{array}{c c c c}
		& 0 & & \\
		0 & 0 &  & 0 \\
	\end{array}\subseteq A_3+A_1,\quad\mbox{and}\quad \begin{array}{c c c}
		& 0 & \\
		0 & 1 & 0  \\
	\end{array}\subseteq D_4.$$
	Since $\O$ is even, we have $$\O=\BInd_{2 A_2 + A_1}^{E_6}\left(\begin{array}{c c c c c}
		& & 0 & & \\
		0 & 0 &  & 0 & 0 \\
	\end{array}\right).$$
	Note furthermore that $$\Ind_{A_3 + A_1}^{D_5}\left(\begin{array}{c c c c}
		& 0 & & \\
		0 & 0 &  & 0 \\
	\end{array}\right)=\Ind_{D_4}^{D_5}\left(\begin{array}{c c c}
		& 0 & \\
		0 & 1 & 0  \\
	\end{array}\right)=\begin{array}{c c c c}
		& 0 & & \\
		0 & 0 & 2 & 0 \\
	\end{array}\subseteq D_5,$$ which has equivariant fundamental group $\Sym_2$ by Proposition~\ref{prop: E6 Levi Gps} and Table~\ref{ta: ADLevis}. Since this nilpotent orbit is even, we have $$\begin{array}{c c c c}
		& 0 & & \\
		0 & 0 & 2 & 0 \\
	\end{array}=\BInd_{A_3 + A_1}^{E_6}\left(\begin{array}{c c c c}
		& 0 & & \\
		0 & 0 &  & 0 \\
	\end{array}\right)$$ and therefore
	the other rigid induction datum birationally induces to the (2-fold) universal cover. This therefore implies by Proposition~\ref{prop: deg} that $\BInd_{D_4}^{E_6}\left(\begin{array}{c c c}
		& 0 & \\
		0 & 1 & 0  \\
	\end{array}\right)$ is a $2$-fold nilpotent cover of $\BInd_{A_3 + A_1}^{E_6}\left(\begin{array}{c c c c}
		& 0 & & \\
		0 & 0 &  & 0 \\
	\end{array}\right)$. Amongst $\hatO, \breveO$ and $\covO$, the only such possibility is therefore that $$\hatO=\BInd_{D_4}^{E_6}\left(\begin{array}{c c c}
		& 0 & \\
		0 & 1 & 0  \\
	\end{array}\right)$$ and 
	$$\breveO=\BInd_{A_3 + 
		A_1}^{E_6}\left(\begin{array}{c c c c}
		& 0 & & \\
		0 & 0 &  & 0 \\
	\end{array}\right).$$ By Proposition~\ref{p: bsr}, $\O$ has a birationally rigid nilpotent cover; this must therefore be $\covO$.
	
	\subsubsection{\texorpdfstring{$D_4(a_1)\subseteq E_7$}{D4a1 in E7}}\label{sss: D4a1 in E7}
	
	The nilpotent orbit in $E_7$ with Bala-Carter label $D_4(a_1)$ has weighted Dynkin diagram
	$$\begin{array}{c c c c c c}
		& & 0 & & & \\
		0 & 2 & 0 & 0 & 0 & 0.
	\end{array}$$
	The rigid induction data for this nilpotent orbit are 
	$$\begin{array}{c c c c c c}
		& & 0 & & & \\
		0 &  & 0 & 0 & 0 & 0 
	\end{array}\subseteq A_5+A_1\quad \mbox{and}\quad \begin{array}{c c c c c}
		& 0& &  &\\
		0 & 1 & 0 & 0 & 0\\
	\end{array}\subseteq D_6.$$
	Since $\O$ is even, we have $$\O=\BInd_{A_5 + A_1}^{E_7}\left(\begin{array}{c c c c c c}
		& & 0 & & & \\
		0 &  & 0 & 0 & 0 & 0 
	\end{array}\right).$$ Furthermore, Proposition~\ref{p: bsr} and the discussion following it show that $\hatO$ is birationally rigid. By examining the equivariant fundamental groups for the rigid induction data for $\O$, we see that one of $\breveO$ and $\covO$ must be birationally rigid and the other birationally induced from $\begin{array}{c c c c c}
		& 0 & &  &\\
		0 & 1 & 0 & 0 & 0\\
	\end{array}\subseteq D_6$. 
	
	Note that by \cite{FJLS} and \cite[Lemma 3.2.2]{MM} $\Spec(\C[\O])$ has a symplectic leaf of codimension 2 whose corresponding singularity has type $A_1$. Since $\breveO$, as a 3-fold nilpotent cover, cannot smooth any $A_1$ singularity, we must have that $\Spec(\C[\breveO])$ has a codimension 2 leaf and thus by \cite[Corollary 7.6.1]{LMM} that $\breveO$ is birationally induced. We conclude that $\covO$ is birationally rigid and that $$\breveO=\BInd_{D_6}^{E_7}\left( \begin{array}{c c c c c}
		& 0 &  &  &\\
		0 & 1 & 0 & 0 & 0\\
	\end{array} \right).$$

	\subsubsection{\texorpdfstring{$D_4(a_1)\subseteq E_8$}{D4a1 in E8}}\label{sss: D4a1 in E8}
	
	The nilpotent orbit in $E_8$ with Bala-Carter label $D_4(a_1)$ has weighted Dynkin diagram 
	$$\begin{array}{c c c c c c c}
		& & 0 & & & &  \\
		0 & 0 & 0 & 0 & 0 & 2 & 0. \\
	\end{array}$$
	The rigid induction data for this nilpotent orbit are $$ \begin{array}{c c c c c c c}
		& & 0 & & & &  \\
		0 & 0 & 0 & 0 & 0 & & 0 \\
	\end{array}\subseteq E_6+A_1\quad\mbox{and}\quad \begin{array}{c c c c c c}
		& & 0 & & &  \\
		0 & 0 & 0 & 0 & 1 & 0  \\
	\end{array}\subseteq E_7. $$
	Since the orbit is even, we get $$\O=\BInd_{E_6 + A_1}^{E_8}\left(\begin{array}{c c c c c c c}
		& & 0 & & & & \\
		0 & 0 & 0 & 0 & 0 & & 0\\
	\end{array}\right).$$ Furthermore, Proposition~\ref{p: bsr} and the discussion following it show that $\hatO$ is birationally rigid. Looking at the equivariant fundamental groups for the rigid induction data for $\O$, we see that one of $\breveO$ and $\covO$ must be birationally rigid and the other birationally induced from $\begin{array}{c c c c c c}
		& & 0 &  & &\\
		0 & 0 & 0 & 0 & 1 & 0 \\
	\end{array}\subseteq E_7$. 
	
	Note that by \cite{FJLS} and \cite[Lemma 3.2.2]{MM} $\Spec(\C[\O])$ has a symplectic leaf of codimension 2 whose corresponding singularity has type $A_1$. Since $\breveO$, as a 3-fold nilpotent cover, cannot smooth any $A_1$ singularity, we must have that $\Spec(\C[\breveO])$ has a codimension 2 leaf and thus by \cite[Corollary 7.6.1]{LMM} that $\breveO$ is birationally induced. We conclude that $\covO$ is birationally rigid and that
	$$\breveO=\BInd_{E_7}^{E_8}\left( \begin{array}{c c c c c c}
		& & 0 &  & &\\
		0 & 0 & 0 & 0 & 1 & 0\\
	\end{array} \right).$$

	\subsubsection{\texorpdfstring{$E_7(a_5)\subseteq E_8$}{E7a5 in E8}}\label{sss: E7a5 in E8}
	
	The nilpotent orbit in $E_8$ with Bala-Carter label $E_7(a_5)$ has weighted Dynkin diagram 
	$$\begin{array}{c c c c c c c}
		& & 0 & & & &  \\
		0 & 0 & 1 & 0 & 1 & 0 & 0 \\
	\end{array}.$$
	There are two rigid induction data for this nilpotent orbit: $$\begin{array}{c c c c c c}
		& & 0 & & & \\
		1 & 0 & 1 & 0 & 0 & 0  \\
	\end{array}\subseteq E_7\quad\mbox{and}\quad \begin{array}{c c c c c c c}
		& & 0 &  & & &\\
		0 & 0 & 1 & 0 & 0 & & 0\\
	\end{array}\subseteq E_6 + A_1,$$ both of which have trivial equivariant fundamental group by Proposition~\ref{prop: adj fun gp} and \cite[\S 8.4]{CM}. Since the Levi subgroups have semisimple corank 1, only two nilpotent covers of $\O$ can be birationally induced and therefore two nilpotent covers must be birationally rigid.
	
	By Proposition~\ref{p: bsr} and the discussion following it, the universal cover $\hatO$ is birationally rigid, and by \cite[Proposition 3.1]{Fu} we have $$\O=\BInd_{E_6+A_1}^{E_8}\left( \begin{array}{c c c c c c c}
		& & 0 &  & & &\\
		0 & 0 & 1 & 0 & 0 & & 0\\
	\end{array} \right).$$ What remains is therefore to determine which of $\breveO$ and $\covO$ are birationally rigid.
	
	Note that by \cite{FJLS} and \cite[Lemma 3.2.2]{MM} $\Spec(\C[\O])$ has a symplectic leaf of codimension 2 whose corresponding singularity has type $2A_1$. Since $\breveO$, as a 3-fold nilpotent cover, cannot smooth any $2A_1$ singularity, we must have that $\Spec(\C[\breveO])$ has a codimension 2 leaf and thus by \cite[Corollary 7.6.1]{LMM} that $\breveO$ is birationally induced. We conclude that $\covO$ is birationally rigid and that $$\breveO=\BInd_{E_7}^{E_8}\left( \begin{array}{c c c c c c}
		& & 0 &  & &\\
		1 & 0 & 1 & 0 & 0 & 0\\
	\end{array} \right).$$

	\subsubsection{\texorpdfstring{$E_8(b_6)\subseteq E_8$}{E8b6 in E8}}\label{sss: E8b6 in E8}
	
	The nilpotent orbit in $E_8$ with Bala-Carter label $E_8(b_6)$ has weighted Dynkin diagram 
	$$\begin{array}{c c c c c c c}
		& & 0 & & & &  \\
		0 & 0 & 2 & 0 & 0 & 0 & 2. \\
	\end{array}$$
	There are three rigid induction data for this nilpotent orbit, which are $$\begin{array}{c c c c c c}
		& & 0 & & &  \\
		0 & 0 &  & 0 & 0 & 0  \\
	\end{array}\subseteq A_3+A_2+A_1,\quad \quad \begin{array}{c c c c c c c}
		& & 0 & & & &  \\
		1 & 0 & 1 & 0 & 1 & & 0  \\
	\end{array}\subseteq E_6+A_1$$ $$\mbox{and}\quad \begin{array}{c c c c c c}
		& 0 & & & & \\
		0 & 1 & 0 & & 0 & 0 \\
	\end{array}\subseteq D_4+A_2.$$
	Since $\O$ is even we have 
	$$\O=\BInd_{A_3 + A_2 +  A_1}^{E_8}\left(\begin{array}{c c c c c c}
		& & 0 & & &  \\
		0 & 0 &  & 0 & 0 & 0  \\
	\end{array}\right),$$
	and we know from Proposition~\ref{p: bsr} and the discussion following it that $\hatO$ is birationally rigid. It thus just remains to determine the birationally rigid induction data for $\breveO$ and $\covO$. 
	
	Note that $$\O=\Ind_{D_5 + A_2}^{E_8}\left( \begin{array}{c c c c c c c}
		&  & 0 & & & & \\
		0 & 2 & 0 & 0 &  & 0 & 0  \\
	\end{array}   \right)$$ and that $$\begin{array}{c c c c c c c}
		&  & 0 & & & & \\
		0 & 2 & 0  & 0 &  & 0 & 0  \\
	\end{array}=\Ind_{A_3 + A_2 + A_1}^{D_5+A_2}\left(\begin{array}{c c c c c c c}
		& & 0 & & & & \\
		0 & & 0 & 0 & & 0 & 0  \\
	\end{array}\right)=\Ind_{D_4 + A_2}^{D_5+A_2}\left(\begin{array}{c c c c c c}
		& 0 &  & & &  \\
		0 & 1 & 0 & & 0 & 0  \\
	\end{array}\right)$$
	has equivariant fundamental group $\Sym_2$ by Proposition~\ref{prop: adj fun gp}. By Proposition~\ref{p: even}, $$\begin{array}{c c c c c c c}
		&  & 0 & & & & \\
		0 & 2 & 0 & 0 &  & 0 & 0  \\
	\end{array}=\BInd_{A_3 + A_2 + A_1}^{D_5+A_2}\left(\begin{array}{c c c c c c}
		& & 0 & & &  \\
		0 & 0 &  & 0 & 0 & 0  \\
	\end{array}\right),$$ and thus 
	$\BInd_{D_4 + A_2}^{D_5+A_2}\left(\begin{array}{c c c c c c}
		& 0 &  & & &  \\
		0 & 1 & 0 & & 0 & 0  \\
	\end{array}\right)$ is the universal cover of $\begin{array}{c c c c c c c}
		&  & 0 & & & & \\
		0 & 2 & 0 & 0 &  & 0 & 0  \\
	\end{array}$. We hence have that $$\BInd_{D_4 + A_2}^{E_8}\left(\begin{array}{c c c c c c}
		& 0 &  & & &  \\
		0 & 1 & 0 & & 0 & 0  \\
	\end{array}\right)$$ is a $2$-fold nilpotent cover of $\O$ by Proposition~\ref{prop: deg}. Since we know that $\hatO$ is birationally rigid, we must therefore have $$\covO=\BInd_{D_4 + A_2}^{E_8}\left(\begin{array}{c c c c c c}
		& 0 &  & & &  \\
		0 & 1 & 0 & & 0 & 0  \\
	\end{array}\right)$$ and, by process of elimination,
	$$\breveO=\BInd_{E_6 + A_1}^{E_8}\left(\begin{array}{c c c c c c c}
		& & 0 & & & &  \\
		1 & 0 & 1 & 0 & 1 & & 0  \\
	\end{array}\right).$$

	\subsubsection{\texorpdfstring{$E_8(a_6)\subseteq E_8$}{E8a6 in E8}}\label{sss: E8a6 in E8}
	
	The nilpotent orbit in $E_8$ with Bala-Carter label $E_8(a_6)$ has weighted Dynkin diagram
	$$\begin{array}{c c c c c c c}
		& & 0 & & & &  \\
		0 & 0 & 2 & 0 & 0 & 2 & 0. \\
	\end{array}$$ Note that $$\begin{array}{c c c c c c c}
		& & 0 & & & &  \\
		0 & 0 & 2 & 0 & 0 & 2 & 0 \\
	\end{array}=\Ind_{E_6 + A_1}^{E_8}\left(\begin{array}{c c c c c c c}
		& & 0 & & & & \\
		0 & 0 & 2 & 0 & 0 & & 0 \\
	\end{array}\right)$$ and that $\pi_1^{E_6 + A_1}\left(\begin{array}{c c c c c c c}
		& & 0 & & & & \\
		0 & 0 & 2 & 0 & 0 & & 0 \\
	\end{array}\right)=\Sym_3$ by Proposition~\ref{prop: adj fun gp} and \cite[\S 8.4]{CM}. Using Proposition~\ref{prop: deg}, it is easy to see the birationally rigid induction data for $E_8(a_6)$ in $E_8$ are the same as the birationally rigid induction data for $D_4(a_1)\times\{0\}$ in $E_6 + A_1$. Arguing as in Subsection~\ref{sss: D4a1 in E6}, we therefore get
	$$\hatO=\BInd_{D_4 + A_1}^{E_8}\left(\begin{array}{c c c c c}
		& 0 & & &\\
		0 & 1 & 0 & & 0 \\
	\end{array}\right),$$
	
	$$\breveO=\BInd_{A_3 + 2 A_1}^{E_8}\left(\begin{array}{c c c c c c}
		& 0 & & & & \\
		0 & 0 &  & 0 & & 0 \\
	\end{array}\right),$$
	
	$$\covO=\BInd_{E_6+  A_1}^{E_8}\left(\mbox{birat. rigid 2-fold nilp. cover of }\begin{array}{c c c c c c c}
		& & 0 & & & & \\
		0 & 0 & 2 & 0 & 0 & & 0 \\
	\end{array}\right),$$
	
	$$\O=\BInd_{2 A_2 + 2 A_1}^{E_8}\left(\begin{array}{c c c c c c c}
		& & 0 & & & & \\
		0 & 0 &  & 0 & 0 & & 0 \\
	\end{array}\right).$$
	
	\subsubsection{\texorpdfstring{$E_8(b_5)\subseteq E_8$}{E8b5 in E8}}\label{sss: E8b5 in E8}
	
	The nilpotent orbit in $E_8$ with Bala-Carter label $E_8(b_5)$ has weighted Dynkin diagram 
	$$\begin{array}{c c c c c c c}
		& & 0 & & & &  \\
		0 & 0 & 2 & 0 & 0 & 2 & 2. \\
	\end{array}$$
	Note that $$\begin{array}{c c c c c c c}
		& & 0 & & & &  \\
		0 & 0 & 2 & 0 & 0 & 2 & 2 \\
	\end{array}=\Ind_{E_6}^{E_8}\left(\begin{array}{c c c c c}
		& & 0 & & \\
		0 & 0 & 2 & 0 & 0 \\
	\end{array}\right)$$ and that $\pi_1^{E_6}\left(\begin{array}{c c c c c}
		& & 0 & & \\
		0 & 0 & 2 & 0 & 0 \\
	\end{array}\right)=\Sym_3$ by Proposition~\ref{prop: adj fun gp}. Using Proposition~\ref{prop: deg}, it is easy to see the birationally rigid induction data for $E_8(b_5)$ in $E_8$ are the same as the birationally rigid induction data for $D_4(a_1)$ in $E_6$. Arguing as in Subsection~\ref{sss: D4a1 in E6}, we therefore get 
	$$\hatO=\BInd_{D_4}^{E_8}\left(\begin{array}{c c c}
		& 0 & \\
		0 & 1 & 0  \\
	\end{array}\right),$$
	
	$$\breveO=\BInd_{A_3 + A_1}^{E_8}\left(\begin{array}{c c c c}
		& 0 & & \\
		0 & 0 &  & 0 \\
	\end{array}\right),$$
	
	$$\covO=\BInd_{E_6}^{E_8}\left(\mbox{birat. rigid 2-fold nilp. cover of }\begin{array}{c c c c c}
		& & 0 & & \\
		0 & 0 & 2 & 0 & 0 \\
	\end{array}\right),$$
	
	$$\O=\BInd_{2 A_2 + A_1}^{E_8}\left(\begin{array}{c c c c c}
		& & 0 & & \\
		0 & 0 &  & 0 & 0 \\
	\end{array}\right).$$
	
	\subsection{\texorpdfstring{$\pi_1^G(\O)=\Sym_3\times\bZ/2\bZ$}{pi=S3Z2}}\label{ss: pi=S3Z2}
	
	This case can only arise in type $E_7$; in this case, there is one nilpotent orbit with $\pi_1^G(\O)=\Sym_3\times \bZ/2\bZ$. This is the nilpotent orbit with Bala-Carter label $E_7(a_5)$, which has weighted Dynkin diagram 
	$$\begin{array}{c c c c c c}
		& & 0 & & & \\
		0 & 0 & 2 & 0 & 0 & 2 \\
	\end{array}.$$
	
	Since $G$-equivariant nilpotent covers of $\O$ correspond to (conjugacy classes of) subgroups of $\Sym_3\times\bZ/2\bZ$, we record for reference the diagram of such subgroups and their inclusions (by which we mean that we draw a line between two subgroups when such an inclusion exists for some pair of subgroups in the conjugacy class). We view elements of $\Sym_3$ as permutations of $\{1,2,3\}$ and denote the generator of $\bZ/2\bZ$ by $b$.
	
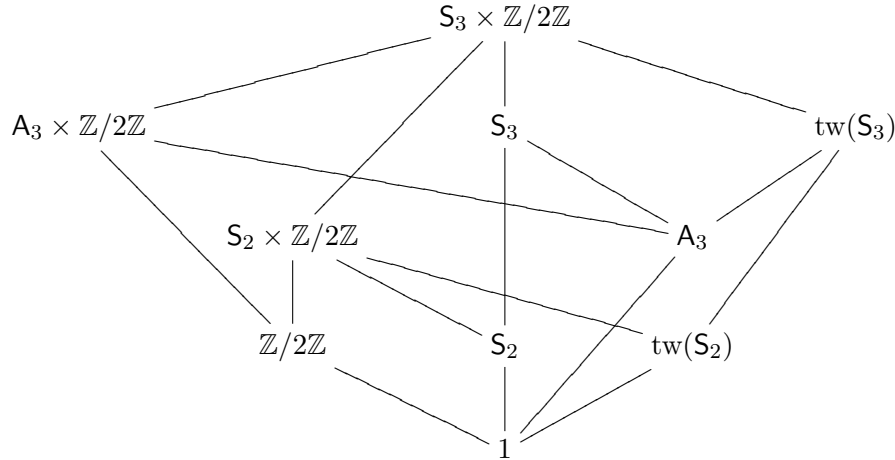
	\begin{eqnarray*}
		\label{e: S3Z2subgps}
		\begin{array}{c}\centerline{\xymatrix{
					& & \Sym_3\times \bZ/2\bZ \ar@{-}[dll] \ar@{-}[ddl] \ar@{-}[d] \ar@{-}[drr] & &  \\
					\Alt_3\times \bZ/2\bZ \ar@{-}[drrr] \ar@{-}[ddr] & & \Sym_3  \ar@{-}[dr] \ar@{-}[dd] & & \tw(\Sym_3)  \ar@{-}[dl] \ar@{-}[ddl] \\
					& \Sym_2\times \bZ/2\bZ \ar@{-}[d] \ar@{-}[dr] \ar@{-}[drr] & & \Alt_3 \ar@{-}[ddl]  &  \\
					& \bZ/2\bZ \ar@{-}[dr] & \Sym_2 \ar@{-}[d] & \tw(\Sym_2) \ar@{-}[dl] &  \\
					& & 1 & & 
		}}\end{array}
	\end{eqnarray*}
	Here we use the notation $\Sym$ and $\Alt$ for the symmetric and alternating groups. In most cases it should be clear which subgroup we are referring to (up to conjugacy), but we note
	$$\Sym_3=\langle (1,2,3), (1,2)\rangle, \quad \tw(\Sym_3)=\langle (1,2,3), (1,2)b\rangle,$$ $$\Sym_2=\langle (1,2)\rangle, \quad \tw(\Sym_2)=\langle (1,2)b\rangle,\quad\mbox{and}\quad \bZ/2\bZ=\langle b\rangle.$$
	
	There are three rigid induction data for $\O$, which are $$\begin{array}{c c c c c}
		& & 0 & & \\
		0 & 0 & & 0 & 0
	\end{array}\subseteq 2A_2+A_1,\quad \begin{array}{c c c c}
		& & 0 & \\
		0 &  & 0 & 0
	\end{array}\subseteq (A_3+A_1)',\quad\mbox{and}\quad \begin{array}{c c c}
		& 0 & \\
		0 & 1 & 0 
	\end{array}\subseteq D_4.$$
	Induction from each of these rigid nilpotent orbits passes through the nilpotent orbit
	$$\O_0:=\begin{array}{c c c c c}
		& & 0 & & \\
		0 & 0 & 2 & 0 & 0
	\end{array}\subseteq E_6=:L_0,$$ which has equivariant fundamental group $\Sym_3$ by Proposition~\ref{prop: E7 Levi Gps} and \cite[\S 8.4]{CM}. Labelling the nilpotent covers of $\O_0$ by $\hatO_0$, $\breveO_0$, $\covO_0$ and $\O_0$ in the conventions of Subsection~\ref{ss: pi=S3}, we get from Subsection~\ref{sss: D4a1 in E6} that $\covO_0$ is birationally rigid and that
	$$\O_0=\BInd_{2 A_2 + A_1}^{E_7}\left(\begin{array}{c c c c c}
		& & 0 & & \\
		0 & 0 &  & 0 & 0 \\
	\end{array}\right),$$
	$$\breveO_0=\BInd_{(A_3 + 
		A_1)'}^{E_7}\left(\begin{array}{c c c c}
		& & 0 & \\
		0 &  & 0 & 0
	\end{array}\right),$$
	$$\hatO_0=\BInd_{D_4}^{E_7}\left(\begin{array}{c c c}
		& 0 & \\
		0 & 1 & 0  \\
	\end{array}\right). $$
	
	Since $\O$ and $\O_0$ are even, we have $\O=\BInd_{E_6}^{E_7}(\O_0)$. We also know from Proposition~\ref{p: bsr} and the discussion following it that the universal cover $\hatO$ is birationally rigid.
	
	As $\breveO_0$ is a 3-fold cover of $\O_0$, Proposition~\ref{prop: deg} implies that $\BInd_{E_6}^{E_7}(\breveO_0)$ is a 3-fold cover of $\O=\BInd_{E_6}^{E_7}(\O_0)$. Thus, $\left\vert \pi_1^G(\O): \pi_1^G(\BInd_{E_6}^{E_7}(\breveO_0)) \right\vert=3$; the only option is therefore $\pi_1^G(\BInd_{E_6}^{E_7}(\breveO_0))=\Sym_2\times\bZ/2\bZ$.
	
	Similarly, as $\covO_0$ is a 2-fold cover of $\O_0$ we get  $\left\vert\pi_1^G(\O): \pi_1^G(\BInd_{E_6}^{E_7}(\covO_0)) \right\vert=2$. Furthermore, since $\pi_1^{L_0}(\covO_0)\cong\Alt_3$, Proposition~\ref{p: surjfund} implies that $\pi_1^G(\BInd_{E_6}^{E_7}(\covO))$ surjects onto $\Alt_3$. The only option is therefore $\pi_1^G(\BInd_{E_6}^{E_7}(\covO))=\Alt_3\times \bZ/2\bZ.$
	
	We also have that $\hatO_0$ is a 3-fold cover of $\covO_0$ and a 2-fold cover of $\breveO_0$. This therefore implies that $\pi_1^G(\BInd_{E_6}^{E_7}(\hatO_0)$ is an index 2 subgroup of $\Sym_2\times\bZ/2\bZ$ and an index 3 subgroup of $\Alt_3\times\bZ/2\bZ$. We must therefore have $\pi_1^G(\BInd_{E_6}^{E_7}(\hatO_0))=\bZ/2\bZ$. 
	
	Note furthermore that $\O$ has induction datum $$\begin{array}{c c c c c}
		& 2 & & &\\
		0  & 0 & 0 & 2 & 0
	\end{array}\subseteq D_6,$$ which has equivariant fundamental group $\bZ/2\bZ$ by Table~\ref{ta: E7Levis}. The universal cover of this induction datum is birationally rigid by Remark~\ref{rmk: brid covs} and Table~\ref{ta: E7Levis}. It must therefore birationally induce to a 2-fold cover of $\BInd_{(A_3+A_1)'}^{E_7}\left(\begin{array}{c c c c}
		& & 0 & \\
		0 &  & 0 & 0
	\end{array}\right)=\BInd_{E_6}^{E_7}(\breveO_0)$. The corresponding subgroup of $\Sym_3\times \bZ/2\bZ$ must be an index 2 subgroup of $\Sym_2\times\bZ/2\bZ$ other than $\bZ/2\bZ$, hence either $\Sym_2$ or $\tw(\Sym_2)$. Note that there is an automorphism of $\Sym_3\times \bZ/2\bZ$ which interchanges these two subgroups; since $\pi_1^G(\O)$ is only described up to isomorphism we cannot distinguish between the two. We therefore make the convention (here and in Table~\ref{ta: E7}) that $$\pi_1^G\left(\BInd_{D_6}^{E_7}\left(\mbox{univ. cover of }\begin{array}{c c c c c}
		& 2 & & &\\
		0  & 0 & 0 & 2 & 0
	\end{array}\right)\right)=\Sym_2.$$ Note that this nilpotent orbit in $D_6$ corresponds to the partition $(4^2,2^2)$, which is very even; following the conventions of Remark~\ref{rmk: labs}, it has label $\I$.
	
	%For the purposes of Table~\ref{ta: E7} we thus need to label it with either a $\I$ or $\II$; as in Table~\ref{ta: E7Levis}, we make the convention (consistent with \cite[Lemma 5.3.5]{CM}) that this nilpotent orbit has label I.
	
	In the notation of Corollary~\ref{cor: brcovs}, we have $\r=\{0\}$ by \cite[\S13.1]{Car} and thus that covers of birationally rigid covers are birationally rigid. As a consequence of this, if a subgroup of $\pi_1^G(\O)$ is an overgroup of a subgroup corresponding to a birationally induced cover that the overgroup must itself correspond to a birationally induced cover of $\O$. In our current setting, this means that the subgroup $\Sym_3$ must correspond to a birationally induced cover of $\O$. By examining the induction data for $\O$ of semisimple corank 1,\footnote{For completeness, these are $(5,3,1^4)\in  D_6$, $(4^2,2^2)_\I\in D_6$, $(2^2)\times (1^2)\in A_5+A_1$, $(3,1^3)\times (1^2)\in A_5+A_1$, $(2,1)\times (1^3)\times (1^2)\in A_3+A_2+A_1$, $(3,2,1^2)\in A_6$, $(2^2,1)\times (1^3)\in A_4+A_2$, $(2,1^3)\times (2,1)\in A_4+A_2$, $(3^2,1^4)\times (2)\in D_5+A_1$ and $D_4(a_1)\in E_6$. Note that two rigid induction data pass through each of $(5,3,1^4)\in  D_6$ and $(3^2,1^4)\times (2)\in D_5+A_1$, and that we have already examined the data $(4^2,2^2)_\I\in D_6$ and $D_4(a_1)\in E_6$. The remaining equivariant fundamental groups may be determined using Proposition~\ref{prop: E7 Levi Gps} and Table~\ref{ta: ADLevis}.} we conclude that we must have $$\Sym_3=\pi_1^G\left(\BInd_{A_5 + A_1}^{E_7}\left(\mbox{univ. cover of }\begin{array}{c c c c c c}
		& & 0 & & &  \\
		0 &  & 0 & 2 & 0 & 0
	\end{array}\right)\right).$$ This universal cover is birationally rigid by Remark~\ref{rmk: brid covs}. We may then check (using the previous footnote) that none of the other nilpotent covers are birationally induced; they are thus all birationally rigid.
	
	\subsection{\texorpdfstring{$\pi_1^G(\O)=\Sym_4$}{pi=S4}}\label{ss: pi=S4}
	
	This case only arises when $G$ is of type $F_4$, where there is a unique nilpotent orbit with $\pi_1^G(\O)=\Sym_4$. This is the nilpotent orbit with Bala-Carter label $F_4(a_3)$, which has weighted Dynkin diagram 
	$$\begin{array}{c c c c}
		0 & 2 & 0 & 0 \\
	\end{array}$$
	(where the first two nodes in the diagram correspond to long roots and the last two correspond to short roots). Since $G$-equivariant nilpotent covers of $\O$ correspond to (conjugacy classes of) subgroups of $\Sym_4$, we record for reference the diagram of such subgroups and their inclusions (by which we mean that we draw a line between two subgroups when such an inclusion exists for some pair of subgroups in the conjugacy class).
	\begin{eqnarray*}
		\label{e: S4subgps}
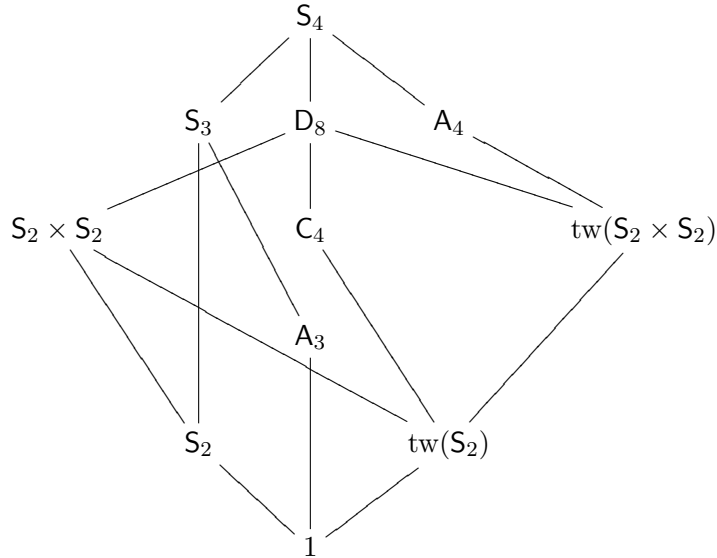
		\begin{array}{c} \centerline{\xymatrix{
					& & \Sym_4 \ar@{-}[d] \ar@{-}[dl] \ar@{-}[dr] & &  \\
					& \Sym_3 \ar@{-}[ddd] \ar@{-}[ddr] & \Dih_8 \ar@{-}[dll] \ar@{-}[d] \ar@{-}[drr] & \Alt_4 \ar@{-}[dr] & \\
					\Sym_2\times \Sym_2 \ar@{-}[ddr] \ar@{-}[ddrrr] & & \Cyc_4 \ar@{-}[ddr] & &  \tw(\Sym_2\times\Sym_2) \ar@{-}[ddl] \\
					& & \Alt_3 \ar@{-}[dd] & &  \\
					& \Sym_2 \ar@{-}[dr] & & \tw(\Sym_2) \ar@{-}[dl] &  \\
					& & 1 & & 
		}}\end{array}
	\end{eqnarray*}
	
	Here we use the notation $\Sym$, $\Alt$, $\Dih$ and $\Cyc$ for the symmetric, alternating, dihedral and cyclic groups, respectively (except that we prefer the notation $\Sym_2$ to $\Cyc_2$ and $\Alt_3$ to $\Cyc_3$). In most cases it should be clear to which subgroup we are referring (up to conjugacy), but we note
	$$\Sym_2\times \Sym_2=\langle (1,2),(3,4)\rangle, \quad \tw(\Sym_2\times\Sym_2)=\langle (1,2)(3,4), (1,3)(2,4)\rangle,$$ $$ \Sym_2=\langle (1,2)\rangle, \quad\mbox{and} \quad \tw(\Sym_2)=\langle (1,2)(3,4)\rangle.$$

	Since $\O$ is even it is birationally induced from the zero orbit in $\widetilde{A}_2 + A_1$ (as described in Section~\ref{s: Prelim}, we use the notation $\widetilde{A}$ for Levi subgroups in types $G_2$ and $F_4$ when we want to indicate that the corresponding simple roots are short). At the other extreme, we know from Proposition~\ref{p: bsr} and the discussion following it that the universal cover of $\O$ is birationally rigid. 
	
	Note that by Corollary~\ref{cor: brcovs} any nilpotent cover of a birationally rigid nilpotent cover of $\O$ is birationally rigid (since $\r=\{0\}$ by \cite[\S13.1]{Car}). Therefore, any nilpotent cover of $\O$ that is covered by a birationally induced nilpotent cover is itself birationally induced.
	
	Other then the zero orbit in $\widetilde{A}_2 + A_1$, the rigid induction data for $\O$ are the nilpotent orbits $$
	\begin{array}{c c c c}
		0 & 0 &  & 0 \\
	\end{array}\subseteq A_2 + \widetilde{A}_1\quad\mbox{and}\quad \begin{array}{c c}
		0 & 0  \\
	\end{array}\subseteq B_2=C_2.
	$$ We note also that $$\Ind_{B_2}^{B_3}\left(\begin{array}{c c}
		0 & 0 \\
	\end{array}\right)= \begin{array}{c c c}
		2 & 0 & 0 \\
	\end{array}\quad \mbox{and}\quad \Ind_{C_2}^{C_3}\left(\begin{array}{c c}
		0 & 0  \\
	\end{array}\right)= \begin{array}{c c c}
		0 & 1 & 0 \\
	\end{array}$$ and that these two induced nilpotent orbits both have equivariant fundamental group $\Sym_2$ by Proposition~\ref{prop: adj fun gp} and \cite[Corollary 6.1.6]{CM}.
	Finally, we note that by Theorem~\ref{th: isogindep} and Proposition~\ref{p: birigid orbs} the nilpotent orbit $\begin{array}{c c c}
		2 & 0 & 0\\
	\end{array}$ in $B_3$ (corresponding to the partition $(3,1^4)$) is birationally induced and the nilpotent orbit $\begin{array}{c c c}
		0 & 1 & 0 \\
	\end{array}$ in $C_3$ (corresponding to the partition $(2^2,1^2)$) is birationally rigid. This therefore implies that $$\BInd_{B_2}^{B_3}\left(\begin{array}{c c}
		0 & 0 \\
	\end{array}\right)=\begin{array}{c c c}
		2 & 0 & 0  \\
	\end{array}$$ and  $$\BInd_{C_2}^{C_3}\left(\begin{array}{c c}
		0 & 0 \\
	\end{array}\right)=\mbox{univ. cover of } \begin{array}{c c c}
		0 & 1 & 0 \\
	\end{array}.$$
	
	Putting this all together, we get that that $\BInd_{B_3}^{F_4}\left(\mbox{univ. cover of } \begin{array}{c c c}
		2 & 0 & 0\\
	\end{array}\right)$ is a nilpotent cover of $\BInd_{B_2}^{F_4}\left(\begin{array}{c c}
		0 & 0\\
	\end{array}\right)$, which itself is a nilpotent cover of $\BInd_{C_3}^{F_4}\left(\begin{array}{c c c}
		0 & 1 & 0 \\
	\end{array}\right)$. Passing to $G$-equivariant fundamental groups, this means that we get three nested subgroups of $\Sym_4$, and we know that neither $1$ not $\Sym_4$ can be in such list (as $1$ corresponds to a birationally rigid nilpotent cover and $\Sym_4$ corresponds to a nilpotent cover whose birationally rigid induction data we already know).
	
	There are a total of 5 birationally induced nilpotent covers of $\O$ (three coming from the rigid induction data, one coming from the universal cover of $\begin{array}{c c c}
		0 & 1 & 0 \\
	\end{array}$ in $C_3$ and one coming from the birationally rigid nilpotent orbit $\begin{array}{c c c}
		2 & 0 & 0  \\
	\end{array}$ in $B_3$). If the nilpotent cover corresponding to the subgroup $\tw(\Sym_2)$ is birationally induced then all nilpotent covers corresponding to overgroups of $\tw(\Sym_2)$ in $\Sym_4$ are birationally induced. There are 6 such (conjugacy classes of) overgroups; since this exceeds the number of birationally induced nilpotent covers, we get a contradiction. Therefore, the nilpotent cover corresponding to $\tw(\Sym_2)$ is birationally rigid.
	
	With this in mind, the only chain of three nested (proper, non-trivial subgroups) in $\Sym_4$ not including $\tw(\Sym_2)$ is the chain $$\Sym_2\subseteq \Sym_2\times\Sym_2 \subseteq \Dih_8.$$ This therefore implies 
	
	$$\pi_1^G\left(\BInd_{B_3}^{F_4}\left(\mbox{univ. cover of } \begin{array}{c c c}
		2 & 0 & 0 \\
	\end{array}\right)\right)=\Sym_2,$$ $$\pi_1^G\left(\BInd_{B_2}^{F_4}\left(\begin{array}{c c}
		0 & 0 \\
	\end{array}\right)\right)=\Sym_2\times\Sym_2,$$ and $$\pi_1^G\left(\BInd_{C_3}^{F_4}\left(\begin{array}{c c c}
		0 & 1 & 0 \\
	\end{array}\right)\right)=\Dih_8.$$
	
	Since $\Sym_2\subseteq \Sym_3$, the nilpotent cover corresponding to $\Sym_3$ must also be birationally induced. There is only one remaining birationally rigid induction datum, and thus we must have $$\Sym_3=\pi_1^G\left( \BInd_{A_2 + \widetilde{A}_1}^{F_4}\left(\begin{array}{c c c c}
		0 & 0 &  & 0 \\
	\end{array} \right)\right).$$
	
	\subsection{\texorpdfstring{$\pi_1^G(\O)=\Sym_5$}{pi=S5}}\label{ss: pi=S5}
	This case only arises in type $E_8$, where there is a unique nilpotent orbit with $\pi_1^G(\O)=\Sym_5$. This is the nilpotent orbit with Bala-Carter label $E_8(a_7)$, which has weighted Dynkin diagram $$\begin{array}{c c c c c c c}
		& & 0 & & & &  \\
		0 & 0 & 0 & 2 & 0 & 0 & 0 \\
	\end{array}.$$
	Since $G$-equivariant nilpotent covers of $\O$ correspond to (conjugacy classes of) subgroups of $\Sym_5$, we record for reference the diagram of such subgroups and their inclusions (by which we mean that we draw a line between two subgroups when such an inclusion exists for some pair of subgroups in the conjugacy class).

	\begin{eqnarray*}
		\label{e: S5subgps}
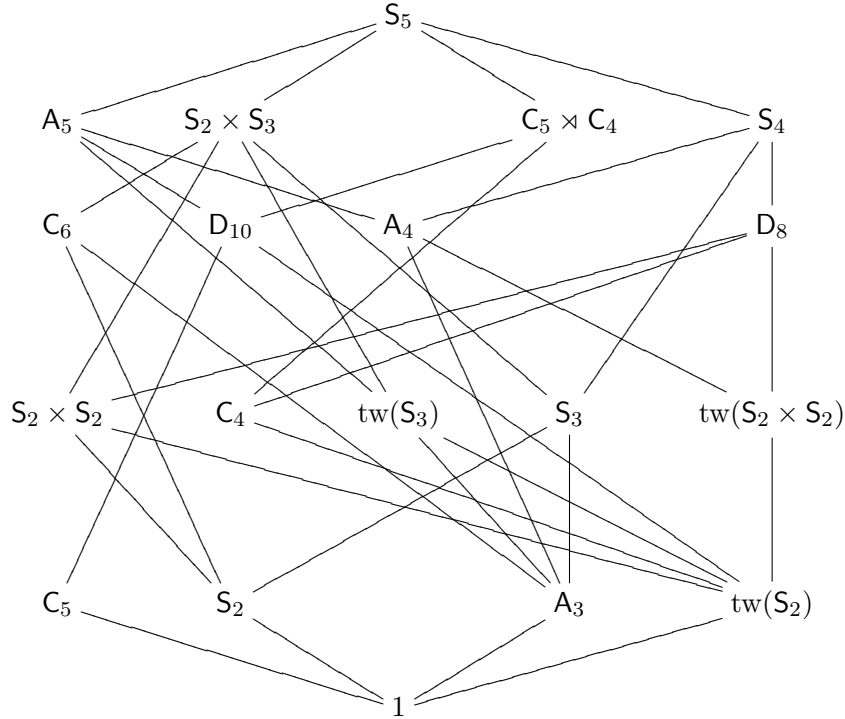
		\begin{array}{c}
			\centerline{\xymatrix{
					& & \Sym_5 \ar@{-}[dll] \ar@{-}[dl] \ar@{-}[dr] \ar@{-}[drr] & &  \\
					\Alt_5 \ar@{-}[dr] \ar@{-}[dddrr] \ar@{-}[drr] & \Sym_2\times \Sym_3 \ar@{-}[dl] \ar@{-}[dddl] \ar@{-}[dddrr] \ar@{-}[dddr] &  & \Cyc_5\rtimes \Cyc_4 \ar@{-}[dll] \ar@{-}[dddll] &  \Sym_4 \ar@{-}[d] \ar@{-}[dddl] \ar@{-}[dll] \\
					\Cyc_6 \ar@{-}[ddddrrr] \ar@{-}[ddddr] & \Dih_{10} \ar@{-}[ddddrrr] \ar@{-}[ddddl] & \Alt_4 \ar@{-}[rrdd]  \ar@{-}[ddddr] &    & \Dih_8 \ar@{-}[dd] \ar@{-}[ddlll] \ar@{-}[ddllll] \\
					& & & & \\
					\Sym_2\times \Sym_2 \ar@{-}[ddr] \ar@{-}[ddrrrr] & \Cyc_4 \ar@{-}[ddrrr]&  \tw (\Sym_3) \ar@{-}[ddr] \ar@{-}[ddrr] & \Sym_3 \ar@{-}[dd] \ar@{-}[ddll] & \tw(\Sym_2\times \Sym_2)  \ar@{-}[dd]\\
					&  & &  &  \\
					\Cyc_5 \ar@{-}[drr] & \Sym_2 \ar@{-}[dr] & & \Alt_3 \ar@{-}[dl] & \tw(\Sym_2) \ar@{-}[dll] \\
					& & 1 & & 
			}}
		\end{array}
	\end{eqnarray*}

	Here we use the notation $\Sym$, $\Alt$, $\Dih$ and $\Cyc$ for the symmetric, alternating, dihedral and cyclic groups, respectively (except that we prefer the notation $\Sym_2$ to $\Cyc_2$ and $\Alt_3$ to $\Cyc_3$). In most cases it should be clear to which subgroup we are referring (up to conjugacy), but we note that $$\Cyc_5\rtimes \Cyc_4=\langle(1,2,3,4,5),(2,3,5,4)\rangle,\quad \Sym_3=\langle(1,2,3),(1,2)\rangle, \quad \tw (\Sym_3)=\langle (1,2,3),(1,2)(4,5)\rangle,$$
	$$ \Sym_2\times \Sym_2=\langle (1,2),(3,4)\rangle, \quad \tw(\Sym_2\times\Sym_2)=\langle (1,2)(3,4), (1,3)(2,4)\rangle,$$
	$$ \Sym_2=\langle (1,2)\rangle \quad \mbox{and} \quad \tw(\Sym_2)=\langle (1,2)(3,4)\rangle. $$
	
	The nilpotent orbit $\O$ has four rigid induction data:
	
	$$\begin{array}{c c c c c c c}
		& & 0 & & & &  \\
		0 & 0 & 0 & & 0 & 0 & 0 \\
	\end{array}\subseteq A_4 + A_3,\quad \quad\quad\begin{array}{c c c c c c c}
		& & 0 & & & &  \\
		0 & 1 & 0 & 0 & & 0 & 0 \\
	\end{array}\subseteq D_5 + A_2, $$
	
	$$\begin{array}{c c c c c c}
		& & 0 & & & \\
		0 & & 0 & 0 & 0 & 0\\
	\end{array}\subseteq A_5 + A_1,\quad \,\,\mbox{and}\,\,\quad \begin{array}{c c c c c}
		& 0 & & & \\
		0 & 1 & 0 & 0 & 0 \\
	\end{array}\subseteq D_6. $$ These nilpotent orbits all have trivial equivariant fundamental groups by Proposition~\ref{prop: adj fun gp} and Table~\ref{ta: ADLevis}. 
	
	By induction from these rigid induction data, we see that the induction data for $\O$ of semisimple corank 1 are the following:
	$$\O_1:=\begin{array}{c c c c c c c}
		& & 0 & & & &  \\
		0 & 0 & 0 & & 0 & 0 & 0 \\
	\end{array}\subseteq A_4 + A_3,\quad \quad\O_2:=\begin{array}{c c c c c c c}
		& & 0 & & & &  \\
		0 & 1 & 0 & 0 & & 0 & 0 \\
	\end{array}\subseteq D_5 + A_2, $$
	
	$$\O_3:=\begin{array}{c c c c c c c}
		& & 2 & & & &  \\
		0 & 0 & 0 & 0 & 0 & & 0 \\
	\end{array}\subseteq E_6 + A_1,\quad \quad \O_4:=\begin{array}{c c c c c c c}
		0 & 1 & 0 & 0 & 0 & 1 & 0 \\
	\end{array}\subseteq A_7, $$
	
	$$\O_5:=\begin{array}{c c c c c c}
		& & 0 & & &  \\
		0 & 2 & 0 & 0 & 0 & 0 \\
	\end{array}\subseteq E_7,\quad \quad \O_6:=\begin{array}{c c c c c c c}
		& & 1 & & & &  \\
		0 &  & 0 & 0 & 0 & 0 & 1 \\
	\end{array}\subseteq A_6 + A_1, $$
	
	$$\mbox{and}\qquad\O_7:=\begin{array}{c c c c c c c}
		& 0 & & & &  \\
		0 & 0 & 1 & 0 & 1 & 0 \\
	\end{array}\subseteq D_7.$$
	Let us denote these Levi subgroups, in order, by $L_1$, $L_2$, $L_3$, $L_4$, $L_5$, $L_6$ and $L_7$. By Proposition~\ref{prop: adj fun gp}, Table~\ref{ta: ADLevis} and \cite[\S 8.4]{CM}, the equivariant fundamental groups of these nilpotent orbits are, respectively, $1$, $1$, $\Sym_2$, $1$, $\Sym_3$, $1$, and $\Sym_2$. We deduce from this that $\O$ admits eight birationally induced $G$-equivariant nilpotent covers: four birationally induced from the rigid nilpotent orbits, one birationally induced from a nilpotent cover of $\O_3$, two birationally induced from nilpotent covers of $\O_5$ and one birationally induced from a nilpotent cover of $\O_7$.
	
	Let us establish some notation for the nilpotent covers of these induction data. For $\O_3$ and $\O_7$ (whose equivariant fundamental groups have 2 elements) we denote by $\hatO_3$ and $\hatO_7$ the (2-fold) universal covers. For $\O_5$, which has equivariant fundamental group $\Sym_3$, we denote the nilpotent covers by $\hatO_5$, $\breveO_5$, $\covO_5$, and $\O_5$, following the conventions of Subsection~\ref{ss: pi=S3}.

	By Proposition~\ref{p: even}, the nilpotent orbits $\O_3$ and $\O_5$ are both birationally induced from the rigid nilpotent orbit $\begin{array}{c c c c c c}
		& & 0 & & &  \\
		0 & & 0 & 0 & 0 & 0 \\
	\end{array}\subseteq A_5 + A_1$. They therefore both       birationally induce to the same nilpotent cover of         $\O$. Furthermore, by Theorem~\ref{th: isogindep} and Proposition~\ref{p: birigid orbs}, the nilpotent orbit $\O_7$ is birationally rigid and thus the rigid induction datum $\begin{array}{c c c c c}
		& 0 & & &  \\
		0 & 1 & 0 & 0 & 0 \\
	\end{array}\subseteq D_6$ birationally induces to $\hatO_7$. Arguing as in Subsection~\ref{sss: D4a1 in E7}, we also get that $\breveO_5$ is birationally induced from this rigid induction datum; therefore both $\breveO_5$ and $\hatO_7$ birationally induce to the same nilpotent cover.

	Since $\O$ is even, we know that $$\O=\BInd_{A_4 + A_3}^{E_8}(\O_1).$$ We also know, by Proposition~\ref{p: bsr} and the discussion following it, that the universal cover of $\O$ is birationally rigid. Furthermore, in the notation of Corollary~\ref{cor: brcovs} we have $\r=\{0\}$ (by \cite[\S13.1]{Car}) and thus covers of birationally rigid covers are birationally rigid.
	
	By Proposition~\ref{p: surjfund}, $\pi_1^G(\BInd_{L_5}^{E_8}(\O_5))$ surjects onto $\Sym_3$. Furthermore, since the universal cover of $\O$ is birationally rigid the order of $\pi_1^G(\BInd_{L_5}^{E_8}(\O_5))$ must be at least 12 (were it 6, the universal cover of $\O_5$ would have to birationally induce to the universal cover of $\O$). Looking at the possible subgroups of $\Sym_5$, the only possibilities are that $\pi_1^G(\BInd_{L_5}^{E_8}(\O_5))=\Sym_2\times \Sym_3$ or $\pi_1^G(\BInd_{L_5}^{E_8}(\O_5))=\Sym_4$.
	
	Let us first consider the case where $\pi_1^G(\BInd_{L_5}^{E_8}(\O_5))=\Sym_4$. Then $\pi_1^G(\BInd_{L_5}^{E_8}(\covO_5))$ is an index 2 subgroup of $\Sym_4$ (by Proposition~\ref{prop: deg}); it must therefore be the alternating group $\Alt_4\subseteq \Sym_4$.
	
	We then get that $\pi_1^G(\BInd_{L_5}^{E_8}(\hatO_5)\leq \Alt_4$ is a subgroup of order 4 by Proposition \ref{prop: deg}. The only option is $\pi_1^G(\BInd_{L_5}^{E_8}(\hatO_5)=\tw(\Sym_2\times \Sym_2)$. Then $$\tw(\Sym_2\times \Sym_2)\leq \pi_1^G(\BInd_{L_5}^{E_8}(\breveO_5))\leq \Sym_4,$$ which implies that we must have $\pi_1^G(\BInd_{L_5}^{E_8}(\breveO_5))=\Dih_8$. We note now that $\BInd_{L_5}^{E_8}(\breveO_5)=\BInd_{L_7}^{E_8}(\hatO_7)$ is a nilpotent cover of $\BInd_{L_7}^{E_8}(\O_7)$ and so we must have $\Dih_8\leq \pi_1^G(\BInd_{L_7}^{E_8}(\O_7))$. The only options are therefore that $\pi_1^G(\BInd_{L_7}^{E_8}(\O_7))=\Sym_4$ or $\Sym_5$. Neither of these are possible, as $\Sym_4$ and $\Sym_5$ are already the equivariant fundamental groups of other birationally induced nilpotent orbits. It therefore cannot be the case that $\pi_1^G(\BInd_{L_5}^{E_8}(\O_5))=\Sym_4$; we must have $\pi_1^G(\BInd_{L_5}^{E_8}(\O_5))=\Sym_2\times \Sym_3$.
	
	Since $\hatO_5$ is a 6-fold nilpotent cover of $\O_5$, $\pi_1^{G}(\BInd_{L_5}^{E_8}(\hatO_5))$ is a subgroup of $\pi_1^{G}(\BInd_{L_5}^{E_8}(\O_5))=\Sym_2\times \Sym_3$ of index 6 (by Proposition~\ref{prop: deg}). Since $\Sym_2\times \Sym_3$ has order 12, we must have that $\pi_1^G(\BInd_{L_5}^{E_8}(\hatO_5))$ has order 2. It must therefore be either (a) $\Sym_2$ or (b) $\tw(\Sym_2)$.
	
	Suppose we are in case (b). Then we have $$\tw(\Sym_2) \leq \pi_1^{G}(\BInd_{L_5}^{E_8}(\covO_5))\leq \Sym_2\times \Sym_3$$ and $\pi_1^G(\BInd_{L_5}^{E_8}(\covO_5))$ surjects onto $\pi_1^{L_5}(\covO_5)=\Alt_3$. No such subgroups exist; we must therefore be in case (a), so 
	$$\pi_1^G(\BInd_{L_5}^{E_8}(\hatO_5))=\Sym_2.$$
	
	Once again, we must have $$\Sym_2 \leq \pi_1^G(\BInd_{L_5}^{E_8}(\covO_5))\leq \Sym_2\times \Sym_3$$ and $\pi_1^G(\BInd_{L_5}^{E_8}(\covO_5))$ surjecting onto $\pi_1^{L_5}(\covO_5)=\Alt_3$. These two conditions force $$\pi_1^G(\BInd_{L_5}^{E_8}(\covO_5))=\Cyc_6.$$
	
	Similarly, we have $$\Sym_2 \leq \pi_1^G(\BInd_{L_5}^{E_8}(\breveO_5))\leq \Sym_2\times \Sym_3$$ and $\pi_1^G(\BInd_{L_5}^{E_8}(\breveO_5))$ surjects onto $\pi_1^{L_5}(\breveO_5)=\Sym_2$. Furthermore, as $\breveO_5$ is a 3-fold nilpotent cover of $\O_5$, $\pi_1^G(\BInd_{L_5}^{E_8}(\breveO_5)$ must be a subgroup of $\Sym_2\times \Sym_3$ of index 3. This forces 
	$$\pi_1^G(\BInd_{L_5}^{E_8}(\breveO_5))=\Sym_2\times \Sym_2.$$ Recall that $\breveO_5$ is not birationally rigid, but is instead birationally induced from the rigid induction datum whose Levi subgroup has type $D_6$.
	
	Recall that nilpotent covers of birationally rigid nilpotent covers are birationally rigid in this case. This implies that, since $\Sym_2$ corresponds to a birationally induced nilpotent cover of $\O$, all overgroups of $\Sym_2$ correspond to birationally induced nilpotent covers of $\O$. The only remaining overgroups of $\Sym_2$ whose corresponding nilpotent cover we have not yet determined are $\Sym_3$, $\Dih_8$, and $\Sym_4$. Furthermore, there are three more birationally rigid induction data for covers of $\O$ for which we have yet to determine the equivariant fundamental groups of their images under birational induction: $(L_2,\O_2)$,  $(L_3,\hatO_3)$, and $(L_7,\O_7)$. We therefore have $$\{\pi_1^G(\BInd_{L_2}^{E_8}(\O_2)),\pi_1^G(\BInd_{L_3}^{E_8}(\hatO_3)),\pi_1^G(\BInd_{L_7}^{E_8}(\O_7))\}=\{\Sym_3, \Dih_8, \Sym_4\},$$ and what remains is to determine which of these birationally induced nilpotent covers corresponds to which of these subgroups.

	Since $$\pi_1^G(\BInd_{L_3}^{E_8}(\hatO_3))\leq \pi_1^G(\BInd_{L_3}^{E_8}(\O_3))=\pi_1^G(\BInd_{L_5}^{E_8}(\O_5))=\Sym_2\times \Sym_3,$$ the remaining options for $\pi_1^G(\BInd_{L_3}^{E_8}(\hatO_3))$ whose birationally rigid induction data aren't known are  $\Sym_3$, $\tw(\Sym_3)$, $\Alt_3$ or $\tw(\Sym_2)$. Since we precisely know which subgroups can arise from the such data, the only possibility is $$\pi_1^G(\BInd_{L_3}^{E_8}(\hatO_3))=\Sym_3.$$ We therefore have $$\{\pi_1^G(\BInd_{L_2}^{E_8}(\O_2)),\pi_1^G(\BInd_{L_7}^{E_8}(\O_7))\}=\{\Dih_8, \Sym_4\}.$$
	
	Since $\Sym_2\times \Sym_2=\pi_1^G(\BInd_{L_5}^{E_8}(\breveO_5))=\pi_1^G(\BInd_{L_7}^{E_8}(\hatO_7))$ is an index 2 subgroup of $\pi_1^G(\BInd_{L_7}^{E_8}(\O_7))$ by Proposition~\ref{prop: deg}, we must have $$\pi_1^G(\BInd_{L_7}^{E_8}(\O_7))=\Dih_8.$$
	Finally, by process of elimination we get $$\pi_1^G(\BInd_{L_2}^{G}(\O_2))=\Sym_4.$$ The remaining covers are all birationally rigid.

	\section{Tables}\label{s: Tables}
	
	In this section, we record the birationally rigid induction data for the nilpotent orbit covers in exceptional Lie algebras. The tables are laid out as follows. There is one table for each type: $G_2$, $F_4$, $E_6$, $E_7$ and $E_8$. Each row corresponds to a $G$-equivariant nilpotent cover $\covO$ of an induced nilpotent $G$-orbit $\O$. For a given $\O$, the nilpotent covers are ordered starting with the trivial cover $\O$ and ending with the universal cover $\hatO$. The first column contains the Bala-Carter label of $\O$ -- we only include this for the row corresponding to the nilpotent orbit. In the second column we give $\pi_1^G(\covO)$ as a subgroup of $\pi_1^G(\O)$ -- this represents the conjugacy class of subgroups which corresponds to the isomorphism class of $\covO$. In the third column we indicate whether or not $\covO$ is birationally rigid. If it is, we write `Y' and leave all the remaining columns blank; if not, we write `N' and fill in the remaining columns as follows. For a birationally induced nilpotent cover $\covO$, let $(L, \covO_L)$ be the birationally rigid induction datum for $\covO$, where $\covO_L$ is a birationally rigid $L$-equivariant nilpotent cover of a nilpotent $L$-orbit $\O_L$. If $\covO$ is birationally induced, then in the fourth column we give the Dynkin type of $L$ as described in Section~\ref{s: Prelim}; we write $T$ when $L$ is the maximal torus of $G$. In the fifth column we give either the partition or the Bala-Carter label corresponding to the nilpotent $L$-orbit $\O_L$ (depending on whether $L$ is of classical or exceptional type), except that we always write $0$ for the zero orbit instead of giving the associated partition. In the sixth column we give the nilpotent cover $\covO_L$ by means of its $L$-equivariant fundamental group $\pi_1^L(\covO_L)\subseteq \pi_1^L(\O_L)$ (up to conjugation).
	
	We mention a slight ambiguity in the table, which is only relevant for $E_7$. When a Levi datum contains two or more indecomposable components of the same Dynkin type, we do not indicate in this table which simple roots correspond to which component. This is largely harmless, since most of the time the nilpotent orbit with which we are concerned has the same description in each such component. This can cause problems in $E_7$, however; for example, the universal cover of the nilpotent orbit with Bala-Carter label $E_7(a_1)$ has as birationally rigid induction datum the universal cover of $0\times(2)\times(2)\times(2)\in 4A_1$, and it does matter in this case which $A_1$ factor corresponds to the zero orbit. Each time this ambiguity arises, the underlying orbit is one found in Table~\ref{ta: E7Levis}, which also gives the weighted Dynkin diagram for the nilpotent orbit (thus determining it uniquely). We therefore adopt the conventions of that table, and the reader may consult Table~\ref{ta: E7Levis} to determine precisely which nilpotent orbit is being described in these cases. The reader may also determine this from the case-specific calculations in Section~\ref{s: CbC}. 
	
	In describing the subgroups of $G$-equivariant fundamental groups, we use the following notation as in Section~\ref{s: CbC}. When $\pi_1^G(\O)=\Sym_2\times \bZ/2\bZ$ we let $a\in \Sym_2$ and $b\in\bZ/2\bZ$ generate $\pi_1^G(\O)$. When $\pi_1^G(\O)=\Sym_3\times \bZ/2\bZ$, we denote the elements of $\Sym_3$ as permutations of $\{1,2,3\}$ in cycle notation and let $b$ be a generator of $\bZ/2\bZ$. We maintain the conventions of the relevant subsections of Section~\ref{s: CbC} for each of the orbits in the below tables, including the notation for subgroups of $\Sym_2\times\bZ/2\bZ$, $\Sym_2\times\bZ/3\bZ$, $\Sym_3$, $\Sym_3\times \bZ/2\bZ$, $\Sym_4$ and $\Sym_5$.

	\renewcommand{\arraystretch}{2}
	\begin{center}
		\begin{longtable}{|c|c|c|c|c|c|}
			\hline
			Type & $\pi_1^G(\covO)$ & Birationally rigid? & $L$ & $\O_L$ & $\pi_1^L(\covO_L)\subseteq \pi_1^L(\O_L)$\\ \hline
			\hline
			$G_2(a_1)$ & $\Sym_3$ & N & $\widetilde{A}_1$ & $0$ & $1\subseteq1$ \\
			
			& $\Alt_3$ & Y & - & - & - \\
			
			& $\Sym_2$ & N & $A_1$ & $0$ & $1\subseteq 1$ \\
			& $1$ & Y & - & - & - \\
			\hline 
			$G_2$ & $1$ & N & $T$ & $0$ & $1\subseteq 1$  \\
			\hline
			\caption{Birationally rigid induction data for $G_2$}\label{ta: G2}
		\end{longtable}
	\end{center}

	\renewcommand{\arraystretch}{2}
	\begin{center}
		\begin{longtable}{|c|c|c|c|c|c|}
			\hline
			Type & $\pi_1^G(\covO)$ & Birationally rigid? & $L$ & $\O_L$ & $\pi_1^L(\covO_L)$\\ \hline
			\hline
			$\widetilde{A}_2$ & $1$ & N & $B_3$ & $0$ & $1\subseteq 1$  \\
			\hline
			$A_2$ & $\Sym_2$ & N & $C_3$ & $0$ & $1\subseteq 1$\\
			& $1$ & Y & - & - & - \\
			\hline
			$B_2$ & $\Sym_2$ & N & $C_3$ & $(2,1^4)$ & $1\subseteq 1$ \\
			& $1$ & Y & - & - & - \\
			\hline
			$C_3(a_1)$ & $\Sym_2$ & N & $B_3$ & $(2^2,1^3)$ & $1\subseteq 1$\\
			& $1$ & Y & - & - & - \\
			\hline
			$F_4(a_3)$ & $\Sym_4$ & N & $\widetilde{A}_2 + A_1$ & $0$ & $1\subseteq 1$ \\
			& $\Alt_4$ & Y & - & - & - \\
			& $\Dih_8$ & N & $C_3$ & $(2^2,1^2)$ & $\Sym_2\subseteq \Sym_2$ \\
			& $\Sym_3$ & N & $A_2 + \widetilde{A}_1$ & $0$ & $1\subseteq 1$ \\
			& $\Cyc_4$ & Y & - & - & - \\
			& $\Sym_2\times\Sym_2$ & N & $B_2$ & $0$ & $1\subseteq 1$\\
			& $\tw(\Sym_2\times \Sym_2)$& Y & - & - & - \\
			& $\Alt_3$  & Y & - & - & - \\
			& $\Sym_2$ & N & $B_3$ & $(3,1^4)$ & $1\subseteq \Sym_2$\\
			& $\tw(\Sym_2)$ & Y & - & - & - \\
			& $1$ & Y & - & - & - \\
			\hline
			$C_3$ & $1$ & N & $A_2$ & $0$ & $1\subseteq 1$\\
			\hline
			$B_3$ & $1$ & N & $\widetilde{A}_2$ & $0$ & $1\subseteq 1$ \\
			\hline
			$F_4(a_2)$ & $\Sym_2$ & N & $A_1 + \widetilde{A}_1$ & $0$ & $1\subseteq 1$\\
			& $1$ & N & $B_2$ & $(2^2,1)$ & $1\subseteq 1$\\
			\hline
			$F_4(a_1)$ & $\Sym_2$ & N& $\widetilde{A}_1$ & $0$ & $1\subseteq 1$\\
			& $1$ & N & $A_1$ & $0$ & $1\subseteq 1$\\
			\hline
			$F_4$ & 1 & N & $T$ & $0$ & $1\subseteq 1$ \\
			\hline 
			\caption{Birationally rigid induction data for $F_4$}\label{ta: F4}
		\end{longtable}
	\end{center}

	\renewcommand{\arraystretch}{2}
	\begin{center}
		\begin{longtable}{|c|c|c|c|c|c|}
			\hline
			Type & $\pi_1^G(\covO)$ & Birationally rigid? & $L$ & $\O_L$ & $\pi_1^L(\covO_L)$\\ \hline
			\hline
			$2A_1$ & $1$ & N & $D_5$ & $0$ &  $1\subseteq 1$ \\
			\hline
			$A_2$ & $\Sym_2$ & N & $A_5$ & $0$ & $1\subseteq 1$ \\
			& $1$ & Y & - & - & - \\
			\hline
			$A_2+A_1$ & $1$ & N & $D_5$ & $(2^2,1^6)$ & $1\subseteq 1$ \\
			\hline
			$2A_2$ & $\bZ/3\bZ$ & N & $D_4$ & $0$ & $1\subseteq 1$ \\
			& $1$ & Y & - & - & - \\
			\hline
			$A_2+2A_1$ & $1$ & N & $A_4 + A_1$ & $0$ & $1\subseteq 1$\\
			\hline
			$A_3$ & $1$ & N & $A_4$ & $0$ & $1\subseteq 1$ \\
			\hline
			$A_3+A_1$ & $1$ & N & $D_5$ & $(3,2^2,1^3)$ & $1\subseteq 1$ \\
			\hline
			$D_4(a_1)$ & $\Sym_3$ & N & $2 A_2 + A_1$ & $0$ & $1\subseteq 1$ \\
			& $\Alt_3$ & Y & - & - & - \\
			& $\Sym_2$ & N & $A_3 + A_1$ & $0$ & $1\subseteq 1$ \\
			& $1$ & N & $D_4$ & $(2^2,1^4)$ & $1\subseteq 1$ \\
			\hline
			$A_4$ & $1$ & N & $A_3$ & $0$ & $1\subseteq 1$\\
			\hline
			$D_4$ & $1$  & N & $2 A_2$ & $0$ & $1\subseteq 1$ \\
			\hline
			$A_4+A_1$ & $1$ & N & $A_2 + 2 A_1$ & $0$ & $1\subseteq 1$ \\
			\hline
			$D_5(a_1)$ & $1$ & N & $A_2 + A_1$ & $0$ & $1\subseteq 1$\\
			\hline
			$A_5$ & $\bZ/3\bZ$ & N & $D_4$ & $(3,2^2,1)$ & $1\subseteq 1$ \\
			& $1$ & Y & - & - & - \\
			\hline
			$E_6(a_3)$ & $\Sym_2\times \bZ/3\bZ$ & N & $3 A_1$ & $0$ & $1\subseteq 1$ \\
			& $\bZ/3\bZ$ & N & $A_2$ & $0$ & $1\subseteq 1$ \\
			& $\Sym_2$ & N & $A_5$ & $(3^2)$ & $1\subseteq \bZ/3\bZ$ \\
			& $1$ & Y & - & - & - \\
			\hline
			$D_5$ & $1$ & N & $2 A_1$ & $0$ & $1\subseteq 1$ \\
			\hline
			$E_6(a_1)$ & $\bZ/3\bZ$ & N & $A_1$ & $0$ & $1\subseteq 1$ \\
			& $1$ & N & $2 A_2 + A_1$ &  $(3) \times (3)\times 0$ & $1\subseteq \bZ/3\bZ$ \\
			\hline
			$E_6$ & $\bZ/3\bZ$ & N & $T$ & $0$ & $1\subseteq 1$ \\
			& $1$ & N & $2 A_2$ & $(3)\times (3)$ & $1\subseteq \bZ/3\bZ$ \\
			\hline
			\caption{Birationally rigid induction data for (simply connected) $E_6$}\label{ta: E6}
		\end{longtable}
	\end{center}
	
	\renewcommand{\arraystretch}{2}
	\begin{center}
		\begin{longtable}{|c|c|c|c|c|c|}
			\hline
			Type & $\pi_1^G(\covO)$ & BR? & $L$ & $\O_L$ & $\pi_1^L(\covO_L)$\\ \hline
			\hline
			$(3A_1)''$ & $\bZ/2\bZ$ & N & $E_6$\footnote{In \cite{EdG} this rigid induction datum is erroneously listed as being the zero orbit in $2A_2+A_1$ rather than in $E_6$.} & $0$ & $1\subseteq 1$ \\
			& $1$ & Y & - & - & - \\
			\hline
			$A_2$ & $\Sym_2$ & N & $D_6$ & $0$ & $1\subseteq 1$ \\
			& $1$ & Y & - & - & - \\
			\hline
			$A_2+A_1$ & $\Sym_2$ & Y & - & - & - \\
			& $1$ & N & $E_6$ & $A_1$ & $1\subseteq 1$\\
			\hline
			$A_2+3A_1$ & $\bZ/2\bZ$ & N & $A_6$ & $0$ & $1\subseteq 1$ \\
			& $1$ & Y & - & - & - \\
			\hline
			$2A_2$ & $1$ & N & $D_5+ A_1$ & $0$ & $1\subseteq 1$ \\
			\hline
			$A_3$ & $1$ & N & $D_6$ & $(2^2,1^8)$ & $1\subseteq 1$ \\
			\hline
			$(A_3+A_1)''$ & $\bZ/2\bZ$ & N & $D_5$ & $0$ & $1\subseteq 1$ \\
			& $1$ & Y & - & - & - \\
			\hline
			$D_4(a_1)$ & $\Sym_3$ & N & $A_5 + A_1$ & $0$ & $1\subseteq 1$ \\
			& $\Alt_3$ & Y & - & - & - \\
			& $\Sym_2$ & N & $D_6$ & $(2^4,1^4)$ & $1\subseteq 1$ \\
			& $1$ & Y & - & - & - \\
			\hline
			$A_3+2A_1$ & $\bZ/2\bZ$ & N & $E_6$ & $3A_1$ & $1\subseteq 1$\\
			& $1$ & Y & - & - & - \\
			\hline
			$D_4$ & $1$ & N & $(A_5)''$ & $0$ & $1\subseteq 1$ \\
			\hline
			$D_4(a_1)+A_1$ & $\Sym_2\times \bZ/2\bZ$ & N & $(A_5)'$ & $0$  & $1\subseteq 1$ \\
			& $\Sym_2$ & N & $D_6$ & $(2^6)_\I$ & $1\subseteq \bZ/2\bZ$ \\
			& $\tw(\Sym_2)$ & Y & - & - & - \\
			& $\bZ/2\bZ$ & N & $E_6$ & $A_2$ & $1\subseteq \Sym_2$ \\
			& $1$ & Y & - & - &  -\\
			\hline
			$A_3+A_2$ & $\Sym_2$ & N & $D_5 + A_1$ & $(2^2,1^6)\times 0$ & $1\subseteq 1$ \\
			& $1$ & N & $D_6$ & $(3,2^2,1^5)$ & $1\subseteq 1$ \\
			\hline
			$A_3+A_2+A_1$ & $\bZ/2\bZ$ & N & $A_4 + A_2$ & $0$ & $1\subseteq 1$ \\
			& $1$ & Y & - & - & - \\
			\hline
			$A_4$ & $\Sym_2$ & N & $D_4 + A_1$ & $0$ &  $1\subseteq 1$\\
			& $1$ & N & $D_5$ & $(2^2,1^6)$ & $1\subseteq 1$ \\
			\hline
			$D_4+A_1$ & $\bZ/2\bZ$ & N & $D_6$ & $(3,2^4,1)$ & $\bZ/2\bZ\subseteq \bZ/2\bZ$ \\
			& $1$ & N & $D_6$ & $(3,2^4,1)$ & $1\subseteq \bZ/2\bZ$ \\
			\hline
			$(A_5)''$ & $\bZ/2\bZ$ & N & $D_4$ & $0$ & $1\subseteq 1$ \\
			& $1$ & N & $D_5 + A_1$ & $(3,1^7)\times (2)$ & $1\subseteq \bZ/2\bZ$\\
			\hline
			$A_4+A_1$ & $\Sym_2$ & Y & - & -  & - \\
			& $1$ & N & $A_4 + A_1$ & 
			$0$ & $1\subseteq 1$\\
			\hline
			$A_4+A_2$ & $1$ & N & $A_3 + A_2 + A_1$ & $0$ & $1\subseteq 1$\\
			\hline
			$D_5(a_1)$ & $\Sym_2$ & N & $D_6$ & $(3^2,2^2,1^2)$ & $\Sym_2\subseteq \Sym_2$ \\
			& $1$ & N & $A_4$ & $0$ & $1\subseteq 1$ \\
			\hline
			$D_5(a_1)+A_1$ & $\bZ/2\bZ$ & N & $A_3 + A_2$ & $0$ & $1\subseteq 1$ \\
			& $1$ & Y & - & - & - \\
			\hline
			$(A_5)'$ & $1$ & N & $D_5 + A_1$ & $(3,2^2,1^3)\times 0$ & $1\subseteq 1$ \\
			\hline
			$A_5+A_1$ & $\bZ/2\bZ$ & N & $E_6$ & $2A_2+A_1$ & $1\subseteq 1$ \\
			& $1$ & Y & - & - & - \\
			\hline
			$E_6(a_3)$ & $\Sym_2$ & N & $A_3 + 2 A_1$ & $0$ & $1\subseteq 1$ \\
			& $1$ & N & $D_4 + A_1$ & $(2^2,1^4)\times 0$ & $1\subseteq 1$ \\
			\hline
			$D_6(a_2)$ & $\bZ/2\bZ$ & N & $D_5$ & $(3,2^2,1^3)$ & $1\subseteq 1$\\
			& $1$ & N & $D_5 + A_1$ & $(3,2^2,1^3)\times (2)$ & $1\subseteq \bZ/2\bZ$ \\
			\hline
			$E_7(a_5)$ & $\Sym_3\times \bZ/2\bZ$ & N & $2 A_2 + A_1$ & $0$ & $1\subseteq 1$ \\
			& $\Alt_3\times\bZ/2\bZ$ & N & $E_6$ & $D_4(a_1)$ & $\Alt_3\subseteq \Sym_3$ \\
			& $\Sym_3$ & N & $A_5+A_1$ & $(2^3) \times 0$ & $1\subseteq \bZ/2\bZ$ \\
			& $\tw(\Sym_3)$ & Y & - & - & - \\
			& $\Sym_2\times \bZ/2\bZ$ & N & $(A_3 + A_1)'$ & $0$ & $1\subseteq 1$ \\
			& $\Alt_3$ & Y & - & - &  -\\
			& $\Sym_2$ & N & $D_6$ & $(4^2,2^2)_\I$\footnote{The meaning of the numeral $\I$ here is given in Remark~\ref{rmk: labs}.}& $1\subseteq \bZ/2\bZ$ \\
			& $\tw(\Sym_2)$ & Y & - & - & - \\
			& $\bZ/2\bZ$ & N & $D_4$ & $(2^2,1^4)$ & $1\subseteq 1$\\
			& $1$ & Y & - & - & - \\
			\hline
			$D_5$ & $1$ & N & $(A_3 + A_1)''$ & $0$ & $1\subseteq 1$ \\
			\hline
			$A_6$ & $1$ & N & $A_2 + 3 A_1$ & $0$ & $1\subseteq 1$ \\
			\hline
			$D_5+A_1$ & $\bZ/2\bZ$ & N & $2 A_2$ & $0$ & $1\subseteq 1$ \\
			& $1$ & N & $(A_5)''$ & $(2^3)$ & $1\subseteq \bZ/2\bZ$ \\
			\hline
			$D_6(a_1)$ & $\bZ/2\bZ$ & N & $A_3$ & $0$ & $1\subseteq 1$ \\
			& $1$ & N & $D_4 + A_1$ & $(2^4)_\II\times (2)$\footnote{The meaning of the numeral $\II$ here is given in Remark~\ref{rmk: labs}.} & $1\subseteq \bZ/2\bZ$ \\
			\hline
			$E_7(a_4)$ & $\Sym_2\times\bZ/2\bZ$ & N & $A_2 + 2 A_1$ & $0$ & $1\subseteq 1$ \\
			& $\Sym_2$ & N & $D_5 + A_1$ & $(3^3,1)\times (2)$ & $1\subseteq \bZ/2\bZ$ \\
			& $\tw(\Sym_2)$ & N & $A_3+A_2+A_1$ & $(2^2)\times 0\times (2)$ & $1\subseteq \bZ/2\bZ$ \\
			& $\bZ/2\bZ$ & N & $D_4+A_1$ & $(3,2^2,1)\times 0$ & $1\subseteq 1$ \\
			& $1$ & Y & - & - & - \\
			\hline
			$E_6(a_1)$ & $\Sym_2$ & N & $4 A_1$ & $0$ & $1\subseteq 1$ \\
			& $1$ & N & $A_2 + A_1$ & $0$ & $1\subseteq 1$\\
			\hline
			$D_6$ & $\bZ/2\bZ$ & N & $D_4$ & $(3,2^2,1)$ & $1\subseteq 1$ \\
			& $1$ & N & $D_4 + A_1$ & $(3,2^2,1)\times (2)$ & $1\subseteq \bZ/2\bZ$\\
			\hline
			$E_6$ & $1$ & N & $(3 A_1)''$ & $0$ & $1\subseteq 1$ \\
			\hline
			$E_7(a_3)$ & $\Sym_2\times\bZ/2\bZ$ & N & $(3 A_1)'$ & $0$ & $1\subseteq 1$ \\
			& $\Sym_2$ & N & $A_3 + 2 A_1$ & $(2^2)\times 0\times (2)$ & $1\subseteq \bZ/2\bZ$ \\
			& $\tw(\Sym_2)$ & N & $D_5 + A_1$ & $(5,3,1^2)\times (2)$ & $\tw(\Sym_2)\subseteq \Sym_2\times\bZ/2\bZ$ \\
			& $\bZ/2\bZ$& N & $A_2$ & $0$ & $1\subseteq 1$ \\
			& $1$ & N & $A_2 + 3 A_1$ & $0 \times (2) \times (2) \times (2)$ & $1\subseteq \bZ/2\bZ$ \\
			\hline
			$E_7(a_2)$ & $\bZ/2\bZ$ & N & $2 A_1$ & $0$ & $1\subseteq 1$ \\
			& $1$ & N & $(A_3 + A_1)''$ & $(2^2) \times (2)$ & $1\subseteq \bZ/2\bZ$ \\
			\hline
			$E_7(a_1)$ & $\bZ/2\bZ$ & N & $A_1$ & $0$ & $1\subseteq 1$ \\
			& $1$ & N & $4 A_1$ & $0\times (2)\times (2)\times (2) $ & $1\subseteq \bZ/2\bZ$ \\
			\hline
			$E_7$ & $\bZ/2\bZ$ & N & $T$ & $0$ & $1\subseteq 1$ \\
			& $1$  & N & $(3 A_1)''$ & $(2)\times(2)\times (2)$ & $1\subseteq \bZ/2\bZ$ \\
			\hline
			\caption{Birationally rigid induction data for (simply connected) $E_7$}\label{ta: E7}
		\end{longtable}
	\end{center}
	
	\renewcommand{\arraystretch}{2}
	\begin{center}
		\begin{longtable}{|c|c|c|c|c|c|}\hline
			Type & $\pi_1^G(\covO)$ & BR? & $L$ & $\O_L$ & $\pi_1^L(\covO_L)$\\ \hline
			\hline
			$A_2$ & $\Sym_2$ & N & $E_7$ & $0$ & $1\subseteq 1$ \\
			& $1$ & Y & - & - & - \\
			\hline
			$A_3$ & $1$ & N & $E_7$ & $A_1$ & $1\subseteq 1$ \\
			\hline
			$2A_2$ & $\Sym_2$ & N & $D_7$ & $0$ & $1\subseteq 1$ \\
			& $1$ & Y & - & - & - \\
			\hline
			$D_4(a_1)$ & $\Sym_3$ & N & $E_6 + A_1$ & $0$ & $1\subseteq 1$ \\
			& $\Alt_3$ & Y & - & - & - \\
			& $\Sym_2$ & N & $E_7$ & $2A_1$ & $1\subseteq 1$ \\
			& $1$ & Y & - & - & -\\
			\hline
			$D_4$ & $1$ & N & $E_6$ & $0$ & $1\subseteq 1$ \\
			\hline
			$A_3+A_2$ & $\Sym_2$ & N & $D_7$ & $(2^2,1^{10})$ & $1\subseteq 1$ \\
			& $1$ & N & $E_7$ & $(3A_1)'$ & $1\subseteq 1$ \\
			\hline
			$A_4$ & $\Sym_2$ & N & $D_6$ & $0$ & $1\subseteq 1$ \\
			& $1$ & N & $E_7$ & $A_2$ & $1\subseteq \Sym_2$ \\
			\hline
			$D_4(a_1)+A_2$ & $\Sym_2$ & N & $A_7$ & $0$ & $1\subseteq 1$ \\
			& $1$ & Y & - & - & - \\
			\hline
			$D_4+A_1$ & $1$ & N & $E_7$ & $4A_1$ & $0$ \\
			\hline
			$A_4+A_1$ & $\Sym_2$ & Y & - & - & - \\
			& $1$ & N & $E_6 + A_1$ & $A_1\times 0$ & $1\subseteq 1$ \\
			\hline
			$D_5(a_1)$ & $\Sym_2$ & N & $E_7$ & $A_2+A_1$ & $\Sym_2\subseteq \Sym_2$\\
			& $1$ & N & $E_6$ & $A_1$ & $1\subseteq 1$ \\
			\hline
			$A_4+2A_1$ & $\Sym_2$ & Y & - & - & - \\
			& $1$ & N & $D_7$ & $(2^4,1^6)$ & $1\subseteq 1$ \\
			\hline
			$A_4+A_2$ & $1$ & N & $D_5 + A_2$ & $0$ & $1\subseteq 1$ \\
			\hline
			$A_4+A_2+A_1$ & $1$ & N & $A_6 + A_1$ & $0$ & $1\subseteq 1$ \\
			\hline
			$D_5(a_1)+A_1$ & $1$ & N & $E_7$ & $A_2+2A_1$ & $1\subseteq 1$ \\
			\hline
			$A_5$ & $1$ & N & $D_7$ & $(3,2^2,1^7)$ & $1\subseteq 1$ \\
			\hline
			$D_4+A_2$ & $\Sym_2$ & N & $A_6$ & $0$ & $1\subseteq 1$ \\
			& $1$ & Y & - & - & - \\
			\hline
			$E_6(a_3)$ & $\Sym_2$ & N & $D_5 + A_1$ & $0$ & $1\subseteq 1$ \\
			& $1$ & N & $D_6$ & $(2^2,1^8)$ & $1\subseteq 1$ \\
			\hline
			$D_5$ & $1$ & N & $D_5$ & $0$ & $1\subseteq 1$ \\
			\hline
			$E_6(a_3)+A_1$ & $\Sym_2$ & N & $E_7$ & $A_1+2A_2$ & $1\subseteq 1$ \\
			& $1$ & Y & - & - & - \\
			\hline
			$D_6(a_2)$ & $\Sym_2$ & N & $D_7$ & $(3,2^4,1^3)$ & $1\subseteq 1$ \\
			& $1$ & Y & - & - & - \\
			\hline
			$E_7(a_5)$ & $\Sym_3$ & N & $E_6+A_1$ & $3A_1\times 0$ & $1\subseteq 1$ \\
			& $\Alt_3$ & Y & - & - & - \\
			& $\Sym_2$ & N & $E_7$ & $(A_1+A_3)'$ & $1\subseteq 1$ \\
			& $1$ & Y & - & - & - \\
			\hline
			$D_5+A_1$ & $1$ & N & $E_6$ & $3A_1$ & $1\subseteq 1$ \\
			\hline
			$E_8(a_7)$ & $\Sym_5$ & N & $A_4 + A_3$ & $0$ & $1\subseteq 1$ \\
			& $\Alt_5$ & Y & - & - & - \\
			& $\Sym_4$ & N & $D_5 + A_2$ & $(2^2,1^6)\times 0$ & $1\subseteq 1$ \\
			& $\Cyc_5\rtimes \Cyc_4$ & Y & - & - & - \\
			& $\Alt_4$ & Y & - & - & - \\
			& $\Sym_2\times \Sym_3$ & N & $A_5 + A_1$ & $0$ & $1\subseteq 1$\\
			& $\Dih_{10}$ & Y & - & - & - \\
			& $\Dih_8$ & N & $D_7$ & $(3^2,2^2,1^4)$ & $\Sym_2\subseteq \Sym_2$ \\
			& $\Sym_3$ & N & $E_6 + A_1$ & $A_2\times 0$ & $1\subseteq \Sym_2$ \\
			& $\tw \Sym_3$ & Y & - & - & - \\
			& $\Cyc_6$ & N & $E_7$ & $D_4(a_1)$ & $\Alt_3\subseteq \Sym_3$ \\
			& $\Cyc_5$ & Y & - & - & - \\
			& $\Cyc_4$ & Y & - & - & - \\
			& $\Sym_2\times \Sym_2$ & N & $D_6$ & $(2^4,1^4)$ & $1\subseteq 1$ \\
			& $\tw(\Sym_2\times\Sym_2)$ & Y & - & - & - \\
			& $\Alt_3$ & Y & - & - & - \\
			& $\Sym_2$ & N & $E_7$ & $D_4(a_1)$ & $1\subseteq \Sym_3$ \\
			& $\tw(\Sym_2)$ & Y & - & - & - \\
			& $1$ & Y & - & - & - \\
			\hline
			$A_6$ & $1$ & N & $D_4 + A_2$ & $0$ & $1\subseteq 1$ \\
			\hline
			$D_6(a_1)$ & $\Sym_2$ & N & $A_5$ & $0$ & $1\subseteq 1$ \\
			& $1$ & N & $E_6$ & $A_2$ & $1\subseteq \Sym_2$ \\
			\hline
			$A_6+A_1$ & $1$ & N & $A_4 + A_2 + A_1$ & $0$ & $1\subseteq 1$ \\
			\hline
			$E_7(a_4)$ & $\Sym_2$ & N & $D_5 + A_1$ & $(2^2,1^6)\times 0$ & $1\subseteq 1$ \\
			& $1$ & N & $D_6$ & $(3,2^2,1^5)$ & $1\subseteq 1$  \\
			\hline
			$D_5+A_2$ & $\Sym_2$ & N & $A_4 + A_2$ & $0$ & $1\subseteq 1$ \\
			& $1$ & N & $D_7$ & $(3^3,2^2,1)$ & $1\subseteq 1$ \\
			\hline
			$E_6(a_1)$ & $\Sym_2$ & N & $D_4 + A_1$ & $0$ & $1\subseteq 1$ \\
			& $1$ & N & $D_5$ & $(2^2,1^6)$ & $1\subseteq 1$ \\
			\hline
			$D_6$ & $1$ & N & $D_6$ & $(3,2^4,1)$ & $1\subseteq 1$ \\
			\hline
			$D_7(a_2)$ & $\Sym_2$ & N & $2 A_3$ & $0$ & $1\subseteq 1$\\
			& $1$ & N & $A_4 + 2 A_1$ & $0$ & $1\subseteq 1$ \\
			\hline
			$E_6$ & $1$ & N & $D_4$ & $0$ & $1\subseteq 1$ \\
			\hline
			$A_7$ & $1$ & N & $D_5 + A_2$ & $(3,2^2,1^3)\times 0$ & $1\subseteq 1$ \\
			\hline
			$E_6(a_1)+A_1$ & $\Sym_2$ & N & $E_7$ & $A_4+A_1$ & $\Sym_2\subseteq \Sym_2$ \\
			& $1$ & N & $A_4 + A_1$ & $0$ & $1\subseteq 1$ \\
			\hline
			$E_8(b_6)$ & $\Sym_3$ & N & $A_3 + A_2 + A_1$ & $0$ & $1\subseteq 1$ \\
			& $\Alt_3$ & N &  $D_4+ A_2$ & $(2^2,1^4)\times 0$ & $1\subseteq 1$ \\
			& $\Sym_2$ & N & $E_6 + A_1$ & $(2A_2 + A_1)\times 0$ & $1\subseteq 1$ \\
			& $1$ & Y & - & - & - \\
			\hline
			$E_7(a_3)$ & $\Sym_2$ & N & $D_6$ & $(3^2,2^2,1^2)$ & $\Sym_2\subseteq \Sym_2$\\
			& $1$ & N & $A_4$ & $0$ & $1\subseteq 1$ \\
			\hline
			$E_6+A_1$ & $1$ & N & $E_6$ & $2A_2+A_1$ & $1\subseteq 1$ \\
			\hline
			$D_7(a_1)$ & $\Sym_2$ & N & $A_3 + A_2$ & $0$ & $1\subseteq 1$ \\
			& $1$ & N & $ D_5 + A_1$ & $(3,2^2,1^3)\times 0$ & $1\subseteq 1$\\
			\hline
			$E_8(a_6)$ & $\Sym_3$ & N & $2 A_2 + 2 A_1$ & $0$ & $1\subseteq 1$ \\
			& $\Alt_3$ & N & $E_6 + A_1$ & $D_4(a_1)\times 0$ & $\Alt_3\subseteq \Sym_3$ \\
			& $\Sym_2$   & N & $A_3 + 2 A_1$ & $0$ & $1\subseteq 1$  \\
			& $1$ & N & $D_4 + A_1$ & $(2^2,1^4)\times 0$ & $1\subseteq 1$ \\
			\hline
			$E_7(a_2)$ & $1$ & N & $D_5$ & $(3,2^2,1^3)$ & $1\subseteq 1$ \\
			\hline
			$D_7$ & $1$ & N & $D_4 + A_2$ & $(3,2^2,1)\times 0$ & $1\subseteq 1$ \\
			\hline
			$E_8(b_5)$ & $\Sym_3$ & N & $2 A_2 + A_1$ & $0$ & $1\subseteq 1$ \\
			& $\Alt_3$ & N & $E_6$ & $D_4(a_1)$ & $\Alt_3\subseteq \Sym_3$  \\
			& $\Sym_2$ & N & $A_3 + A_1$ & $0$ & $1\subseteq 1$ \\
			& $1$ & N & $D_4$ & $(2^2,1^4)$ & $1\subseteq 1$ \\
			\hline
			$E_8(a_5)$ & $\Sym_2$ & N & $A_2 + 3 A_1$ & $0$ & $1\subseteq 1$ \\
			& $1$ & N & $2 A_2$ & $0$ & $1\subseteq 1$ \\
			\hline
			$E_7(a_1)$ & $1$ & N & $A_3$ & $0$ & $1\subseteq 1$ \\
			\hline
			$E_8(b_4)$ & $\Sym_2$ & N & $A_2 + 2 A_1$ & $0$ & $1\subseteq 1$ \\
			& $1$ & N & $D_4 + A_1$ & $(3,2^2,1)\times 0$ & $1\subseteq 1$ \\
			\hline
			$E_8(a_4)$ & $\Sym_2$ & N & $4 A_1$ & $0$ & $1\subseteq 1$ \\
			& $1$ & N & $A_2 + A_1$ & $0$ & $1\subseteq 1$ \\
			\hline
			$E_7$ & $1$ & N & $D_4$ & $(3,2^2,1)$ & $1\subseteq 1$\\
			\hline
			$E_8(a_3)$ & $\Sym_2$ & N & $3 A_1$ & $0$ & $1\subseteq 1$ \\
			& $1$ & N & $A_2$ & $0$ & $1\subseteq 1$ \\
			\hline
			$E_8(a_2)$ & $1$ & N & $2 A_1$ & $0$ & $1\subseteq 1$ \\
			\hline
			$E_8(a_1)$ & $1$ & N & $A_1$ & $0$ & $1\subseteq 1$ \\
			\hline
			$E_8$ & $1$ & N & $T$ & $0$ & $1\subseteq 1$ \\
			\hline
			\caption{Birationally rigid induction data for $E_8$}\label{ta: E8}
		\end{longtable}
	\end{center}

\end{document}